\documentclass[CJK,11pt]{amsart}
\usepackage{geometry}
\usepackage{amsmath,amsfonts,amssymb,amsthm}
\usepackage{graphics}
\usepackage{dsfont}
\usepackage{epsfig}
\usepackage{esint}
\usepackage[colorlinks,linkcolor=blue]{hyperref}
\newcommand{\lessim}{\stackrel{<}{\sim}}

\newcommand{\ignore}[1]{}

\newtheorem{definition}{Definition}
\newtheorem{proposition}{Proposition}
\newtheorem{theorem}{Theorem}

\newtheorem{lemma}{Lemma}
\newtheorem{corollary}{Corollary}
\newtheorem{assumption}{Assumption}

\newcommand{\R}{\mathbb{R}}

\newcommand{\lt}{\left}
\newcommand{\rt}{\right}
\newcommand{\Boule}{B}
\newcommand{\langl}{\langle}
\newcommand{\rangl}{\rangle}

\newcommand{\LL}{L}

\newcommand{\abar}{\overline{a}}
\newcommand{\dd}{d}

\usepackage{color}
\definecolor{darkred}{rgb}{0.9,0.1,0.1}
\definecolor{darkblue}{rgb}{0,0,0.7}
\definecolor{darkgreen}{rgb}{0,0.5,0}

\begin{document}

\title[Bias in the Representative Volume Element method]{{\bf Bias in the Representative Volume Element method:\\
periodize the ensemble instead of its realizations}}
\author{Nicolas Clozeau, Marc Josien, Felix Otto and Qiang Xu}
\thanks{Max Planck Institute for Mathematics in the Sciences,
Inselstrasse 22, 04103 Leipzig, Germany}
\maketitle
\begin{abstract}
We study the Representative Volume Element (RVE) method, which is a method to approximately infer the effective behavior $a_{\rm hom}$ of a stationary random medium.
The latter is described by a coefficient field $a(x)$ generated from a given ensemble $\langle\cdot\rangle$ and the corresponding linear elliptic operator $-\nabla\cdot a\nabla$. 
In line with the theory of homogenization, the method proceeds by computing $d=3$ correctors ($d$ denoting the space dimension).
To be numerically tractable, this computation has to be done on a finite domain: the so-called ``representative''  volume element, i.~e.~a large box with, say, 
periodic boundary conditions. 
The main message of this article is:
Periodize the ensemble instead of its realizations.

\smallskip

By this we mean that it is better to sample from a suitably periodized ensemble
than to periodically extend the restriction of a realization $a(x)$ 
from the whole-space ensemble $\langle\cdot\rangle$.
We make this point by investigating the bias (or systematic error), i.~e.~the difference between 
$a_{\rm hom}$ and the expected value of the RVE method, 
in terms of its scaling w.~r.~t.~the lateral size $L$ of the box.
In case of periodizing $a(x)$, 
we heuristically argue that this error is generically $O(L^{-1})$.
In case of a suitable periodization of $\langle\cdot\rangle$, 
we rigorously show that it is $O(L^{-d})$. 
In fact, we give a characterization of the leading-order error term for
both strategies, and argue that even in the isotropic case it is generically non-degenerate.

\smallskip

We carry out the rigorous analysis in the convenient
setting of ensembles $\langle\cdot\rangle$ of Gaussian type,
which allow for a straightforward periodization, 
passing via the (integrable) covariance function.
This setting has also the advantage of making the Price theorem and the
Malliavin calculus available
for optimal stochastic estimates of correctors.
We actually need control of second-order correctors 
to capture the leading-order error term.
This is due to inversion symmetry when applying the two-scale expansion to the Green function.
As a bonus, we present a stream-lined strategy to estimate
the error in a higher-order two-scale expansion of the Green function.
\end{abstract}
\tableofcontents

\section{Introduction and statement of rigorous result}


\subsection{Uniformly elliptic coefficient fields}
The basic objects of this paper are $\lambda$-uniformly elliptic tensor
fields $a=a(x)$ (that are not necessarily symmetric)
in $d$-dimensional space, by which we mean that for all points $x$
\begin{align}\label{Ellipticite}
\xi\cdot a(x)\xi\ge\lambda|\xi|^2\quad\mbox{and}\quad
\xi\cdot a(x)\xi\ge|a(x)\xi|^2\quad\mbox{for all}\;\xi\in\mathbb{R}^d.
\end{align}
\footnote{Note that the second condition implies $|a(x)\xi|$ $\le|\xi|$. 
These conditions are not equivalent, unless $a(x)$ is symmetric, 
however the form of the bounds in (\ref{Ellipticite}) is the one preserved 
under homogenization.} 
Such a tensor field $a$ gives rise to the
heterogeneous elliptic operator $-\nabla\cdot a \nabla$ acting on functions $u$.\footnote{
While we use scalar language and notation, $\mathbb{R}$ as a space for the values of $u$
may be replaced by a finite dimensional vector space.} 

\medskip

Homogenization means assimilating the effective, i.~e.~large-scale, behavior
of a heterogeneous medium 
to a homogeneous one, as described by the constant tensor $a_{\rm hom}$. By this one means that
the difference of the solution operators $(-\nabla\cdot a\nabla)^{-1}$
$-(-\nabla\cdot a_{\rm hom}\nabla)^{-1}$ converges to zero when applied to functions 
$f$ varying only on scales $L\uparrow\infty$. 
Homogenization is known to take place
in a number of situations, see \cite{Tartar} for a general notion,
e.~g.~when the coefficient field $a$ is periodic or when it is
sampled from a stationary and ergodic ensemble $\langle\cdot\rangle$.
While we are interested in the latter, it is convenient to introduce the
Representative Volume Element (RVE) method in the context of the former.


\subsection{The RVE method}
Unless $d=1$, there is no explicit formula that allows to compute in practice $a_{\rm hom}$ 
for a general ensemble $\langle\cdot\rangle$. Early work treated
specific ensembles that admit asymptotic explicit formulas in limiting regimes,
like spherical inclusions covering a low volume fraction
in \cite[p.\ 365]{Maxwell-1873}. Explicit upper and lower bounds on
$a_{\rm hom}$ in terms of features of the ensemble $\langle\cdot\rangle$ 
play a major role in the engineering literature, see for instance \cite{MoriTanaka,Torquato}.
On the contrary, the RVE method is a computational method to obtain
convergent approximations to $a_{\rm hom}$ for a general the ensemble $\langle\cdot\rangle$.
As the name ``volume element'' indicates, it is based on samples $a$ of $\langle\cdot\rangle$
in a (computational) domain, typically a cube of side-length $L$. 
It consists in inverting $-\nabla\cdot a\nabla$ for $d$ (representative) boundary conditions. 
The question of the appropriate size $L$ of the RVE evolved from a philosophical
one in \cite{HillRVE} (large enough to be statistically typical and so that boundary
effects are dominated by bulk effects)
towards a more practical one in \cite{DruganWillis} (just large enough so that the
statistical properties relevant for the physical quantity $a_{\rm hom}$ are captured).
The convergence of the method has been extensively investigated by numerical experiments
in the engineering literature, see some references below.
In this paper, we rigorously analyze some aspects of the convergence
for a certain class of ensembles $\langle\cdot\rangle$.

	
%
\medskip		

We now introduce the RVE method. 
Suppose (momentarily) that the coefficient field $a$ is $L$-periodic, meaning that
$a(x+Lk)=a(x)$ for all $x$ and $k\in\mathbb{Z}^d$. 
Given a Cartesian coordinate direction $i=1,\cdots,d$ and denoting by $e_i$
the unit vector in the $i$-th direction, we define (up to additive constants)
$\phi_i^{(1)}$ as the $L$-periodic solution of
\begin{align}\label{pde:9.1_quad}
-\nabla\cdot a(\nabla\phi_i^{(1)}+e_i)=0.
\end{align}
The function $\phi_i^{(1)}$ is called first-order
corrector, because it additively corrects the affine coordinate
function $x_i$ in such a way that the resulting function
$x\mapsto x_i+\phi_i^{(1)}(x)$ is $a$-harmonic, by
which we understand that it vanishes under application of $-\nabla\cdot a\nabla$.

\medskip

Let us momentarily adopt the language of a conducting medium: On the
microscopic level, multiplication with the tensor field $a$ converts 
the electric field into the electric flux. On the macroscopic level, it is $a_{\rm hom}$
that relates the averaged field to the averaged flux.
In view of (\ref{pde:9.1_quad}),  $\nabla\phi_i^{(1)}+e_i$ can be considered as an electric
field in the absence of charges, 
arising from the electric potential $-(\phi_i^{(1)}+x_i)$.
In view of the periodicity of $\phi_i^{(1)}$, the large-scale average of
$\nabla\phi_i^{(1)}+e_i$ is just $e_i$. Now $a(\nabla\phi_i^{(1)}+e_i)$ is the corresponding
flux. It is periodic, so its large-scale average is given by its average on the periodic cell
\begin{align}\label{Def_Abar_intro}
\bar a e_i:=\fint_{[0,L)^d}a(\nabla\phi_i^{(1)}+e_i).
\end{align}
Observe that the notation $\bar a$ without reference
to the period $L$ is legitimate since (\ref{Def_Abar_intro}) is equivalent to
$\bar a e_i$ $=\lim_{R\uparrow\infty}\fint_{[0,R)^d}a(\nabla\phi_i^{(1)}+e_i)$.
A well-known feature of homogenization is
that $\bar a$ inherits the bounds (\ref{Ellipticite}) from $a$, 
as can be derived with help of the dual problem (\ref{ao01}).
In the periodic case, (\ref{Def_Abar_intro}) in fact coincides with the homogenized coefficient $a_{\rm hom}$.

\medskip

On the contrary, in the random case which we introduce now, (\ref{Def_Abar_intro}) provides 
only a fluctuating approximation to the deterministic $a_{\rm hom}$.
Homogenization is known to take place when $a$ is sampled from a stationary and
ergodic ensemble $\langle\cdot\rangle$, see \cite{JKO,Varadhan_1979}. 
By the latter we mean a probability 
measure\footnote{We are deliberately vague on the $\sigma$-algebra, which could be taken as
generated by the Borel algebra induced by $H$-convergence as in \cite{GloriaOtto_2015},
because we will consider a very explicit class in this paper.} 
on the space of tensor fields $a$ satisfying (\ref{Ellipticite}); we use
the symbol $\langle\cdot\rangle$ to address both the ensemble and to denote 
its expectation operator.
Stationarity is the crucial structural assumption and means that the shifted
random field $x\mapsto a(z+x)$ has the same (joint) distribution
as $a$ for any shift vector $z\in\mathbb{R}^d$. Ergodicity is a qualitative 
assumption\footnote{Again, we are deliberately vague since this qualitative assumption will
be replaced by an explicit quantitative one.}
that encodes the decorrelation of the values of $a$ over large distances.

\subsection{Two strategies of periodizing}

In order to apply the RVE method in form of (\ref{Def_Abar_intro}), considered
as an approximation for $a_{\rm hom}$,
one needs to produce samples $a$ of $L$-periodic coefficient fields
connected to the given ensemble $\langle\cdot\rangle$.
The goal of this paper is to compare two strategies to procure such $L$-periodic
samples.
The first strategy relies on ``periodizing the realizations'' in its most naive form
-- we shall actually consider a seemingly less naive form of it, see Section \ref{S:heur} --
and goes as follows:
Taking a coefficient field $a$ in $\mathbb{R}^d$, we restrict it to the box $[0,L)^d$ 
and then periodically extend it.
This defines a map $a\mapsto a_L$. We then take
$\overline{a_L}$, cf.~(\ref{Def_Abar_intro}), 
as an approximation for $a_{\rm hom}$. One unfavorable aspect of this strategy is obvious: 
The push-forward of $\langle\cdot\rangle$ under this map $a\mapsto a_L$ is no longer stationary 
-- an imagined glance at a typical realization would reveal $d$ 
families of parallel artificial hypersurfaces. 

\medskip

Related variants of this strategy consist
in still restricting $a$ to $[0,L)^d$ but then imposing Dirichlet or Neumann boundary
conditions instead of periodic boundary conditions (Neumann conditions will actually be analyzed in
Section \ref{S:heur}). Both boundary conditions for the random (vector) field
$\nabla\phi^{(1)}=\nabla\phi^{(1)}(a,x)$ destroy its shift-covariance\footnote{also often called stationarity}: It is no longer
true that for any shift-vector $z\in\mathbb{R}^d$ we have $\nabla\phi^{(1)}(a,z+x)$
$=\nabla\phi^{(1)}(a(z+\cdot),x)$.

\medskip

The second strategy relies on ``periodizing the ensemble'' and is more subtle: Given
an ensemble $\langle\cdot\rangle$, 
one constructs a ``related'' stationary ensemble $\langle\cdot\rangle_L$ 
of $L$-periodic fields,
samples $a$ from $\langle\cdot\rangle_L$ and takes $\bar a$ as an approximation.
The quality of this second method was numerically explored in \cite{Gusev} for
random non-overlapping inclusions, and (next to the first strategy) in \cite{KanitRVE} for
random Voronoi tesselations\footnote{the periodization is mentioned in a 
somewhat hidden way on p.3658}; in both cases the periodization is obvious. 
Requirements on the periodization of ensembles were formulated in \cite[Section 4]{SabNedjar}, 
a general construction idea was given in \cite[Remark 5]{Gloria_Otto}.
In this paper, we advocate thinking of the map $\langle\cdot\rangle\leadsto\langle\cdot\rangle_L$
as conditioning on periodicity, and will construct it 
for a specific but relevant class of $\langle\cdot\rangle$ given in Assumption \ref{Ass:1}. 

\medskip

The second strategy obviously capitalizes on the knowledge of the ensemble $\langle\cdot\rangle$ 
and not just of a single realization (a ``snapshot''),
in the sense of ``known unknowns'' as opposed to ``unknown unknowns''.
This is in contrast with the numerical analysis on inferring $a_{\rm hom}$ in \cite{Mourrat_2019}, 
or on constructing effective boundary conditions in \cite{Lu_Otto_2018,LuOttoWang}
from a snapshot.


\subsection{Fluctuations and bias}
In this paper, we are interested in comparing these two strategies 
in terms of their bias (also called systematic error): How much do the two expected values 
$\langle\overline{a_L}\rangle$ and $\langle\bar a\rangle_L$
deviate from $a_{\rm hom}$, which by qualitative theory is their common limit for $L\uparrow\infty$ (see \cite{BourgeatPiatnitski} for $\langle\overline{a_L}\rangle$ and Corollary \ref{Cor:1} $iii)$ for $\langle\overline{a}\rangle_L$).
We shall heuristically argue that
\begin{equation}\label{ao18}
\langle\overline{a_L}\rangle-a_{\rm hom}=O(L^{-1}),
\end{equation}
see Section \ref{S:heur}, while proving 
\begin{align}\label{ao17}
\langle\bar a\rangle_L-a_{\rm hom}=O(L^{-d}),
\end{align}
see Theorems \ref{T:mainpara} and \ref{T:main}.
 Here $L$ should be thought of as the 
(non-dimensional) ratio between
the actual period $L$ and a suitably defined correlation length of 
$\langle\cdot\rangle$ set to unity. 
The quantification of the convergence in $L$ is clearly of practical interest:
After a discretization that resolves the correlation length, the number
of unknowns of the linear algebra problem (\ref{pde:9.1_quad}) scales with $L^d$ for $L\gg 1$.
Numerical experiments confirm the
$O(L^{-d})$ scaling \cite[Figure 5]{KhormoskijOtto} and the substantially worse
behavior \eqref{ao18} for the other strategy \cite{SchneiderJosienOtto}.
In this regard, result \eqref{ao18} is not unexpected either from a theoretical or 
a numerical perspective, \textit{cf.} \cite[(3.4)]{Egloffe2015}, \cite{BourgeatPiatnitski} and \cite{HMMAnalysis}, 
respectively.
Nevertheless, to the best of our knowledge, we provide here the first formal argument in favor of such a behavior.

\medskip

We note that fluctuations (which are at the origin of the random part of the error), 
as for instance measured in terms of the square root of the variance, 
are in many situations proven to be of the order (see, \textit{e.g.}, \cite[Theorem 2]{Gloria_Otto})
\begin{align}\label{ao18-bis}
\langle|\bar a-\langle\bar a\rangle_L|^2\rangle_L^\frac{1}{2}=O(L^{-\frac{d}{2}}),
\end{align}
see \textit{e.g.} \cite[Figure 6]{KhormoskijOtto} for a numerical validation,
and the same is expected to hold for the other strategy 
(\cite[Fig.3]{Yue2007} and \cite[(3.3)]{Egloffe2015})
\begin{align*}
\langle|\overline{a_L}-\langle\overline{a_L}\rangle|^2\rangle^\frac{1}{2}=O(L^{-\frac{d}{2}}).
\end{align*}
Hence the variance scales like the inverse of the volume $L^d$ of the periodic cell $[0,L)^d$,
as if we were averaging over $[0,L)^d$ some field of unit range of dependence instead of the 
long-range correlated $a(\nabla\phi_i+e_i)$.
In view of this identical fluctuation scaling for both strategies, the different bias scaling 
is significant in the most relevant dimension $d=3$, which
we mostly focus in this paper: For the first strategy, the bias dominates, 
so that taking the empirical mean of $\overline{a_L}$ over many realizations $a$
does not substantially reduce the total error. It does so in the second scenario,
which suggests to use variance reduction methods, like analyzed in 
\cite{Minvielle_2016, Fischer_ARMA_2019}.

\medskip

Theoretical results on the random error in RVE, at least for the second strategy like in (\ref{ao18-bis}), are by
now abundant, starting from \cite[Theorem 2.1]{GloriaOtto_2011}
for a discrete medium with i.~i.~d.~coefficient, 
over \cite[Theorem 1]{GloriaOttoESAIMProceedingsandSurveys2015} 
for a class of continuum media based 
on the Poisson point process,
to the leading-order identification of the variance in in \cite[Theorem 2]{DuerinckxGO_2016}.
The last result arises from the characterization of leading-order variances
in stochastic homogenization in general, starting from \cite[Theorem 2.1]{Mourrat_Otto_2016} 
for correctors, and is in the spirit of the general approach laid out in \cite{GuMourrat_2016}. 
These estimates of variances and fluctuations in homogenization rely on a functional
calculus (Malliavin derivatives, Spectral Gap inequalities).
There is an alternative approach based on a finite range assumption
(and its relaxation via mixing conditions) that was shown to yield optimal results
in \cite{Armstrong_Kuusi_Mourrat_2017, GloriaOtto_2015}, and culminated in the monograph
\cite{Armstrong_book_2018}. In this paper, we make use of the first approach.

\medskip

Theoretical results on the systematic error in RVE, again for the second strategy as in (\ref{ao17}), 
seem to have been restricted to the case of a discrete medium with i.~i.~d.~coefficients, 
see \cite[Proposition 3]{Gloria_Otto}, 
where the construction of $\langle\cdot\rangle_L$ is obvious. 
The argument 
for \cite[Proposition 3]{Gloria_Otto} is based on a 
(necessarily non-stationary) coupling
of $\langle\cdot\rangle$ and $\langle\cdot\rangle_L$, and introduces a massive
term into the corrector equation in order to screen the resulting boundary layer, 
which leads to a logarithmically worse estimate than (\ref{ao17}). Our analysis avoids this coupling
and suggests that such a logarithmic correction is artificial.
(Incidentally, the phenomenon that the bias decays to an order that is twice
the order of the fluctuation decay occurs also in the analysis of the homogenization
error $(-\nabla\cdot a\nabla)^{-1}f-(-\nabla\cdot a_{\rm hom}\nabla)^{-1}f$
itself: While the variance can be characterized to order $O(L^{-d})$,
where $L\gg 1$ now is the ratio between the scale of $f$ and the correlation length,
see \cite[Theorem 1]{DuerinckxOtto_2019},
the expectation seems to be characterized to order $O(L^{-2d})$, see \cite{Bourgain18,
LemmKim,DuerinckxGloriaLemm}.)

\medskip

%
The first strategy is appealing since it only requires a snapshot,
which could come from an actual material image, whereas the second one 
requires knowledge of the underlying ensemble, which has to be estimated or imposed as a model.
Several methods to overcome the effect of boundary layers on the first strategy
have been proposed:
%
Motivated by the treatment of periodic coefficient fields of unknown period and
the ensuing resonance error, 
oversampling \cite{HouWu1997} and filtering \cite{Blanc2010} strategies were proposed.
Until recently, in case of random media however, because of the slow decay of the boundary layer,
they were not expected to perform better than $O(L^{-1})$, see \cite[(3.4)]{Egloffe2015}.
This motivated \cite{gloria1} to screen the boundary effects by a massive term 
to the corrector equation \eqref{pde:9.1_quad}, cf.~(\ref{MassiveEquation}). However, results in preparation \cite{BellaFischerJosienRaithel} suggest that, in the case of an isotropic ensemble, oversampling strategies may give rise to an improved rate $O(L^{-d/2})$ for the systematic error. Screening strategies based on semi-group 
\cite{Abdulle_2019,Mourrat_2019} 
or wave-equation \cite{Arjmand_2016} versions of the corrector equation have 
also been analyzed.
Screening by a massive term, 
in conjunction with extrapolation in the massive parameter, has been proven to reduce
the systematic error to $O(L^{-d})$ \cite[Thm. 2]{Gloria_Otto}.
Based on screening and extrapolation,
\cite[Prop. 1.1 \& Th. 1.2]{Mourrat_2019} formulated a numerical algorithm that
extracts $a_{\rm hom}$ from a snapshot $a$ up to the optimal total error $O(L^{-\frac{d}{2}})$ 
with only $O(L^d)$ operations.

\subsection{Assumptions and formulation of rigorous result}\label{SectionDefEnsembleL}

We now introduce a class of ensembles $\langle\cdot\rangle$ of $\lambda$-uniform
coefficient fields $a$ that can be easily periodized. Loosely speaking, the natural way to periodize a general stationary ensemble
$\langle\cdot\rangle$ of coefficient fields $a$ on $\mathbb{R}^d$
is to {\it condition} on $a$ being $L$-periodic. 
Clearly, this conditioning is highly singular, and we thus shall restrict 
ourselves to stationary and centered {\it Gaussian} ensembles $\langle\cdot\rangle$. Since a realization
$g$ of such a Gaussian ensemble is obviously not ($\lambda$-uniformly) elliptic, we will
work with a (nonlinear) map $A$ and consider the pointwise transformation $a(x)=A(g(x))$.
More precisely, we will identify $\langle\cdot\rangle$ with its push-forward under
\begin{align}\label{Def_a}
g\mapsto a:=\big(x\mapsto A(g(x))\big).
\end{align}
Centered Gaussian ensembles on some (infinite-dimensional) Banach space $X$ are 
characterized through their covariance, which is a semi-definite bounded bilinear 
form on $X^*$, defining a Hilbert space\footnote{known as the Cameron-Martin space} 
$H\subset X$. When $H$ is a Hilbert space
of H\"older continuous functions on $\mathbb{R}^d$, this operator is best represented
by its kernel $c(x,y)=\langle g(x) g(y)\rangle$. Stationarity of $\langle\cdot\rangle$ 
then amounts to $c=c(x-y)$; the positive-semidefinite character of the bilinear form
translates into non-negativity of the Fourier transform ${\mathcal F}c(q)\ge 0$
for all wave vectors $q\in\mathbb{R}^d$. We now argue that periodization by conditioning 
can be characterized as follows: $\langle\cdot\rangle_L$ is the stationary centered Gaussian measure
with $L$-periodic covariance $c_L$ with Fourier coefficients given
by restricting the Fourier transform ${\mathcal F}c(q)$ to the dual lattice 
$q\in\frac{2\pi}{L}\mathbb{Z}^d$, defining the $k$-th Fourier mode of $c_L$ as :
\begin{align}\label{RestrictFourierIntro}
\frac{1}{\sqrt{L^d}}\int_{[0,L)^d}dx\, e^{-i\frac{2\pi k}{L}\cdot x}c_L(x):=\frac{1}{\sqrt{L^d}}{\mathcal F}c(\frac{2\pi k}{L})
\;\;\mbox{for all}\;k\in\mathbb{Z}^d.
\end{align}
Since loosely speaking, the contributions to $\langle\cdot\rangle$ from every
wave vector $q\in\mathbb{R}^d$ are independent, this restriction indeed corresponds to
conditioning. This definition also highlights that information is lost when
passing from $\langle\cdot\rangle$ to $\langle\cdot\rangle_L$. In terms of real space,
the passage from $c$ to $c_L$ amounts to periodization of the covariance function:
\begin{align}\label{Def_cL}
c_L(x)=\sum_{k\in\mathbb{Z}^d}c(x+Lk).
\end{align}
%
As for the whole space ensemble, we identify $\langle\cdot\rangle_L$ with its
push-forward under (\ref{Def_a}).

\medskip

We now collect the technical assumptions on $\langle\cdot\rangle$, that is,
on the covariance function $c$ and the map $A$. Loosely speaking, we need that $A$
is regular and that $c$
is regular with integrable decay, both up to second derivatives
\footnote{A subclass of these ensembles, namely those of Mat\'ern form
${\mathcal F}c(q)=(1+|q|^2)^{-\frac{d}{2}-\nu}$, is for instance used as a prior
for elastic microstructures, where the smoothness parameter $\nu$
is estimated from real material images, and the effective modulii $a_{\rm hom}$ are
computed by RVE via periodization of ensembles, see
\cite[(2.10) and Fig.~9]{Wohlmuth_2020}.}.
\begin{assumption}\label{Ass:1}
Let $\langle\cdot\rangle$ be a stationary and centered Gaussian ensemble
of scalar\footnote{For notational simplicity, we consider scalar Gaussian field. 
However, the Gaussian field $g$ may take values in any finite-dimensional linear space,
which gives a high degree of flexibility.} fields $g$ on $\mathbb{R}^d$,
as determined by the covariance function\footnote{of which we implicitly assume that $\lim_{|x|\uparrow\infty}c(x)=0$} $c(x):=\langle g(x) g(0)\rangle$.
We assume that there exists an $\alpha>0$ such that
%
%
%
\begin{align}\label{c:6_bis}
\sup_{x\in\mathbb{R}^d}(1+|x|^2)^{\frac{d+2}{2}+\alpha}|\nabla^2c(x)|<\infty.
\end{align}
We identify $\langle\cdot\rangle$ with its push-forward under the map (\ref{Def_a}),
where $A\colon\mathbb{R}\rightarrow\mathbb{R}^{d\times d}$ is such that
the coefficient field $a$ is $\lambda$-uniformly elliptic, see (\ref{Ellipticite}). 
We assume that
\begin{align}\label{Reg_A}
\sup_{g\in\mathbb{R}}|A'(g)|+|A''(g)|<\infty.
\end{align}
\end{assumption}

On the one hand, (\ref{c:6_bis}) implies by integration
$\sup_{x\in\mathbb{R}^d}(1+|x|^2)^{\frac{d}{2}+\alpha}|c(x)|<\infty$
and thus because of $\alpha>0$
\begin{align}\label{fw32}
\sup_{q\in\mathbb{R}^d}{\mathcal F}c(q)\lesssim\int dx|c(x)|<\infty.
\end{align}
Via (\ref{RestrictFourierIntro}),
(\ref{fw32}) yields that the Cameron-Martin norm of $\langle\cdot\rangle_L$ dominates
the $L^2([0,L)^d)$-norm. This implies that $\langle\cdot\rangle_L$ endowed with the 
Hilbert structure of $L^2([0,L)^d)$ has a uniform spectral gap in $L$, see 
\cite[Poincar\'e inequality (5.5.2)]{Bogachev_Gaussian}.
By (\ref{Reg_A}), this transmits to the ensemble of $a$'s, see (\ref{Def_a}), 
and will be used for the stochastic estimates.

\medskip

On the other hand, (\ref{c:6_bis}) implies $\sup_{x}|\nabla^2 c(x)|<\infty$ and thus
$\sup_{x,y}|y-x|^{-2}\langle(g(y)$ $-g(x))^2\rangle$ 
$=\sup_{z}|z|^{-2}(c(0)$ $-c(z))$ $<\infty$. Since $g$ is Gaussian, this extends
to arbitrary moments: $\sup_{x,y}|y-x|^{-1}$ $\langle|g(y)-g(x)|^p\rangle^{\frac{1}{p}}<\infty$.
Estimating the H\"older semi-norm $[g]_{\alpha,B_1}^p$ by the Besov norm
$\int_{B_1} dz|z|^{-d-p\alpha}\int_{B_1} dx|g(x+z)$ $-g(x)|^p$, 
one derives $\sup_{x}\langle [g]_{\alpha,B_1(x)}^p\rangle<\infty$ for any\footnote{this
$\alpha$ is unrelated to the one in (\ref{c:6_bis}); the argument is a
spatial version of Kolmogorov's criterion} $\alpha<1$ and $p<\infty$.
By (\ref{Reg_A}), this transmits to the coefficient field $a$:
\begin{align}\label{fw33}
\sup_{x}\langle\|a\|_{C^{0,\alpha}(B_1(x))}^p\rangle<\infty
\quad\mbox{for all}\;\alpha<1,\;p<\infty,
\end{align}
%
which will allow us to appeal to Schauder theory for local regularity. Since because of \eqref{c:6_bis} we also have $\sup_{L\geq 1}\vert\nabla^2 c_L(x)\vert<\infty$, \eqref{fw33} extends to $\langle\cdot\rangle_L$ :
\begin{equation}\label{fw33Bis}
\sup_{L\geq 1}\sup_x\langle\|a\|^p_{C^{0,\alpha}(B_1(x))}\rangle_L<\infty\quad \text{for any $\alpha<1$, $p<\infty$.}
\end{equation}
Equipped with the definition of the periodized
ensembles $\langle\cdot\rangle_L$, we can state our main result.

\begin{theorem}\label{T:mainpara}
Let $d>2$ and $A$ be symmetric. 
Under the Assumption~\ref{Ass:1} on $\langle\cdot\rangle$, for all $L$, and with $\langle\cdot\rangle_L$ defined with \eqref{RestrictFourierIntro} we have
for the expectation $\langle\bar a\rangle_L$ of $\bar a$ defined in
\eqref{Def_Abar_intro}
\begin{align}\label{fw37}
\limsup_{L\uparrow\infty}L^{d}|\langle\bar a\rangle_L-a_{\rm hom}|<\infty.
\end{align}
%
\end{theorem}

Let us motivate the scaling (\ref{fw37}). For fixed $i=1,\cdots,d$ we consider the
flux 
\begin{equation}\label{Defq:Intro}
q:=a(\nabla\phi^{(1)}_i+e_i),
\end{equation}
which is a random (vector) field, meaning that $q=q(g,x)$.
We note that by uniqueness for (\ref{pde:9.1_quad}), $q$ is stationary,
or rather ``shift-covariant'', meaning that for every shift vector $z\in\mathbb{R}^d$,
we have $q(g,z+x)=q(g(z+\cdot),x)$ for all points $x$ and (periodic) fields $g$. 
Hence by stationarity of $\langle\cdot\rangle_L$, we may write $\langle\bar a\rangle_L$
$=\langle q(0)\rangle_L$. Clearly, $q(0)$, as arising from the solution of the PDE 
(\ref{pde:9.1_quad}), depends via $a=A(g)$ on the value of $g$ in any point $y$, no
matter how distant from $0$. 

\medskip

Let us assume for a moment that $q$ were more {\it local},
meaning that $q(0)$ depends on $g$ only through its restriction $g_{|B_R}$ 
for some radius $R<\infty$.
Let us also assume for simplicity that $\langle\cdot\rangle$ has {\it unit range}, 
which amounts to assume that $c$ is supported in $B_1$, a sharpening of (\ref{c:6_bis}). 
We then claim that
\begin{align*}
\langle q(0)\rangle_L\quad\mbox{is independent of}\;L\ge 2R+2.
\end{align*}
Indeed, by the locality assumption and (centered) Gaussianity, 
the distribution of the value $q(0)=q(g,0)$ is determined by ${c_L}_{|B_{2R}}$.
In view of (\ref{Def_cL}) and by the finite range assumption, ${c_L}_{|B_{2R}}=c_{|B_{2R}}$
for $L\ge 1+2R+1$.

\medskip

As mentioned, our flux $q(g,0)$ does depend on $g(y)$ even for $R=|y|\gg 1$.
This dependence is described by the mixed derivative $\nabla\nabla G(a,0,y)$
of the Green function $G(a,x,y)$ for $-\nabla\cdot a\nabla$, see Subsection 
\ref{SS:asympt}. 
Stochastic estimates show that, at least on an annealed\footnote{The language of quenched and annealed
arises from metallurgy estimate and made its to model with disorder in statistical
mechanics.} level, the decay
of this variable-coefficient Green's function is no worse than of its
constant-coefficient counterpart so that $R^d|\nabla\nabla G(a,0,y)|\lesssim 1$.
Loosely speaking, it is this exponent $d$ that shows up in (\ref{fw37}).

\medskip

In the following Section \ref{S:struc}, 
we will refine Theorem \ref{T:mainpara} by characterizing the leading-order
error term in Theorem \ref{T:main}.


\section{Theorem \ref{T:mainpara}: refinement and main ideas}\label{S:struc}

The two ingredients for Theorem \ref{T:mainpara} are a suitable representation
formula for $\langle\bar a\rangle_L$, see Subsection \ref{SS:formula}, 
and its asymptotics through stochastic homogenization,
here on the level of the mixed derivatives of the Green function, see Subsection
\ref{SS:asympt}. We need the second-order version of stochastic homogenization
because of an inversion symmetry. 
We refine Theorem \ref{T:mainpara} in Subsection \ref{SS:refine}
by identifying the leading-order error term, see Theorem \ref{T:main}.
In Subsection \ref{SS:small} we will argue that the leading-order error 
typically does not vanish, by exploring the regime of small ellipticity contrast.
In Subsection \ref{SS:symm} we discuss the structure of the leading-order error term
in the case of an isotropic ensemble.

\subsection{Representation formula}\label{SS:formula}

We start with an informal, but detailed, derivation of the representation formula,
see (\ref{ao04}),
which might be the most conceptual piece of our work.

\medskip

Let us fix two vectors $\xi$ and $\xi^*$ and focus on the component
$\xi^*\cdot\bar a\xi$; 
we denote by $\phi^{(1)}$ the solution of (\ref{pde:9.1_quad}) with $e_i$
replaced by $\xi$.\footnote{By linearity and uniqueness (up to additive constants), we have $\phi^{(1)} = \sum_i\xi_i\phi^{(1)}_i$.} 
By shift-covariance\footnote{meaning that $\nabla\phi^{(1)}(a(\cdot+h),x)$
$=\nabla\phi^{(1)}(a,x+h)$ a.~s.} of $\nabla \phi^{(1)}$ and stationarity of $\langle\cdot\rangle_L$, we have
\begin{align}\label{Num::1}
\langle\xi^*\cdot\bar a \xi\rangle_L=\langle F\rangle_L\quad\mbox{where}\quad
F:=(\xi^*\cdot a(\nabla\phi^{(1)}+\xi))(0).
\end{align}
Instead of directly estimating $\langle\xi^*\cdot \bar a\xi\rangle_L
-\xi^*\cdot a_{\rm hom}\xi$, we will
estimate its derivative w.~r.~t.~$L$, that is $\frac{d}{dL}\langle\xi^*\cdot\bar a\xi\rangle_L$. The reason 
is that by general Gaussian calculus (in form of the Price formula) applied to the ensemble
$\langle\cdot\rangle_L$ of (periodic) fields $g$ that depends on a parameter $L$,
we have for any $F=F(g)$ 
%
\begin{equation}
\frac{d}{dL}\langle F\rangle_L
=\frac{1}{2}\int_{\mathbb{R}^d}dx \int_{\mathbb{R}^d}dy\Big\langle\frac{\partial^2 F}
{\partial g(-x)\partial g(-y)}\Big\rangle_L\frac{\partial c_L}{\partial L}(x-y),
\label{Price}
\end{equation}
where the two minus signs in the denominator are for later convenience.
Here $\frac{\partial^2 F}{\partial g(x)\partial g(y)}$ denotes the kernel representing
the second Fr\'echet derivative of $F$, seen as a bilinear form on the space of
functions on $\mathbb{R}^d$.  
As a derivative w.~r.~t.~the noise $g$, it can be seen as a Malliavin derivative. We refer the reader to \cite{PriceFormula} for a rigorous proof of \eqref{Price}.

\medskip

We define $F$ by \eqref{Num::1}.
By the change of variables $z\rightsquigarrow x-y$, which capitalizes on the translation invariance
of the covariance, and (more directly) by the stationarity of $\langle\cdot\rangle_L$ 
in conjunction with the stationarity of $\nabla\phi$ that leads to
$\langle\frac{\partial^2 F}{\partial g(-x)\partial g(z-x)}\rangle_L$ 
$=\langle\xi^*\cdot\frac{\partial^2 a(\nabla\phi^{(1)}+\xi)(x)}{\partial g(0)\partial g(z)}
\rangle_L$,
we obtain
%
\begin{align}\label{ao02}
\frac{d}{dL}\langle\xi^*\cdot \bar a\xi\rangle_L
=\frac{1}{2}\int_{\mathbb{R}^d}dz
\Big\langle\int_{\mathbb{R}^d}\xi^*\cdot\frac{\partial^2  a(\nabla\phi^{(1)}+\xi)}
{\partial g(0)\partial g(z)}\Big\rangle_L \frac{\partial c_L}{\partial L}(z).
\end{align}
With help of the corrector for the (pointwise) dual\footnote{While we work with
the assumption $A^*=A$ and thus have $a^*=a$, keeping primal and dual
medium apart reveals more of the structure.} coefficient field $a^*$
in direction $\xi^*$, i.~e.~the periodic solution ${\phi^*}^{(1)}$ of
\begin{align}\label{ao01}
\nabla\cdot a^*(\nabla{\phi^*}^{(1)}+\xi^*)=0,
\end{align}
the inner integral can be rewritten more symmetrically as
\begin{align*}
\int_{\mathbb{R}^d}\xi^*\cdot\frac{\partial^2 a(\nabla\phi^{(1)}+\xi)}
{\partial g(0)\partial g(z)}&=\int_{\mathbb{R}^d}(\nabla{\phi^*}^{(1)}+\xi^*)\cdot\frac{\partial^2  a(\nabla\phi^{(1)}+\xi)}
{\partial g(0)\partial g(z)}\nonumber\\
&=\int_{\mathbb{R}^d}(\nabla{\phi^*}^{(1)}+\xi^*)\cdot[\frac{\partial^2 }
{\partial g(0)\partial g(z)},a](\nabla\phi^{(1)}+\xi);\nonumber
\end{align*}
indeed, the first identity (formally) follows from applying $\frac{\partial^2}{\partial g(0)\partial g(z)}$ to~(\ref{pde:9.1_quad}) and then testing with ${\phi^*}^{(1)}$,
whereas the second identify follows from testing (\ref{ao01}) with
$\frac{\partial^2\phi^{(1)}}{\partial g(0)\partial g(z)}$. Resolving the commutator $[\frac{\partial^2 }
{\partial g(0)\partial g(z)},a]$ by Leibniz' rule we obtain
%
\begin{equation}\label{Naiveper_FormalCoputeIntro}
\begin{aligned}
\int_{\mathbb{R}^d}\xi^*\cdot\frac{\partial^2 a(\nabla\phi^{(1)}+\xi)}
{\partial g(0)\partial g(z)}=&\int_{\mathbb{R}^d}(\nabla{\phi^*}^{(1)}+\xi^*)\cdot\frac{\partial^2 a }
{\partial g(0)\partial g(z)}(\nabla\phi^{(1)}+\xi)\\
&+\int_{\mathbb{R}^d}(\nabla{\phi^*}^{(1)}+\xi^*)\cdot\frac{\partial a }
{\partial g(0)}\nabla\frac{\partial \phi^{(1)}}{\partial g(z)}
+\mbox{term with $z$, $0$ exchanged}.
\end{aligned}
\end{equation}
Denoting $a':=A'(g)$ and $a'':=A''(g)$, we remark that by (\ref{Def_a}) we have $\frac{\partial a(x)}{\partial g(z)}$ $=a'(z)\delta(x-z)$.
%
%
Applying operator $\frac{\partial}{\partial g(z)}$ on (\ref{pde:9.1_quad}), we thus obtain the representation
\begin{align}\label{Naiveper_Derivative7}
\frac{\partial \nabla \phi^{(1)}(x)}{\partial g(z)}=-\nabla\nabla G(x,z)a'(z)(\nabla\phi+e)(z)
\end{align}
%
in terms of the mixed derivatives of the \textit{non-periodic} Green function\footnote{Since we are only interested
in the mixed gradient of the Green function, the dimension $d=2$ poses no problems here.} $G=G(a,x,y)$ associated with the operator $-\nabla\cdot a\nabla$.
Hence the above turns into
\begin{align*}
\int_{\mathbb{R}^d}\xi^*\cdot\frac{\partial^2 a(\nabla\phi^{(1)}+\xi)}
{\partial g(0)\partial g(z)}
=&\delta(z)\big((\nabla{\phi^*}^{(1)}+\xi^*)\cdot a''(\nabla\phi^{(1)}+\xi)\big)(0)
\nonumber\\
&-\big(a'(\nabla{\phi^*}^{(1)}+\xi^*)\big)(0) \cdot\nabla\nabla G(0,z) \big(a'
(\nabla\phi^{(1)}+\xi)\big)(z)\nonumber\\
&-\big(a'(\nabla{\phi^*}^{(1)}+\xi^*)\big)(z) \cdot\nabla\nabla G(z,0) \big(a'
(\nabla\phi^{(1)}+\xi)\big)(0).
\end{align*}
Applying $\langle\cdot\rangle_L$ we obtain by stationarity
\begin{equation}\label{ao03}
\begin{aligned}
\big\langle\int_{\mathbb{R}^d}\xi^*\cdot\frac{\partial^2 a(\nabla\phi^{(1)}+\xi)}
	{\partial g(0)\partial g(z)}\big\rangle_L=&\delta(z)\big\langle(\nabla{\phi^*}^{(1)}+\xi^*)\cdot a''(\nabla\phi^{(1)}+\xi)\big\rangle_L
\\
&-\big\langle\big(a'(\nabla{\phi^*}^{(1)}+\xi^*)\big)(0) \cdot \nabla\nabla G(0,z) \big(a'
(\nabla\phi^{(1)}+\xi)\big)(z)\big\rangle_L\\
&-\big\langle\big(a'(\nabla{\phi^*}^{(1)}+\xi^*)\big)(0)\cdot\nabla\nabla G(0,-z) \big(a'
(\nabla\phi^{(1)}+\xi)\big)(-z)\big\rangle_L.
\end{aligned}
\end{equation}
Inserting this into (\ref{ao02}), and noting that since $\frac{\partial c_L}{\partial L}$ is even (as derivative of a covariance function), the two last terms
have the same contribution, we obtain
\begin{equation}
\label{ao35}
\begin{aligned}
\frac{d}{dL}\langle\xi^*\cdot \bar a\xi\rangle_L
=&-\int_{\mathbb{R}^d}dz\big\langle \big(a'(\nabla{\phi^*}^{(1)}+\xi^*)\big)(0) \cdot\nabla\nabla G(0,z) \big(a'
(\nabla\phi^{(1)}+\xi)\big)(z)\big\rangle_L\frac{\partial c_L}{\partial L}(z)\\
&+\frac{1}{2}\big\langle(\nabla{\phi^*}^{(1)}+\xi^*)\cdot a''(\nabla\phi^{(1)}+\xi)\big\rangle_L
\frac{\partial c_L}{\partial L}(0).
\end{aligned}
\end{equation}
We now insert (\ref{Def_cL}) in form of
\begin{align}\label{ao37}
\frac{\partial c_L}{\partial L}(z)\stackrel{(\ref{Def_cL})}{=}
\sum_{k\in\mathbb{Z}^d}k\cdot \nabla c(z+Lk).
\end{align}
This relation highlights that the $z$-integral in (\ref{ao35}) is
not absolutely convergent for $|z|\uparrow\infty$, not even borderline:
While $\nabla\nabla G(0,z)$ decays as $|z|^{-d}$, a glance at (\ref{ao37})
reveals that $\frac{\partial c_L}{\partial L}(z)$ grows as $|z|$.
%
%
Part of the
rigorous work is devoted to emulate this formal derivation of 
(\ref{ao35}) by replacing the operator $-\nabla\cdot a\nabla$
by $\frac{1}{T}-\nabla\cdot a\nabla$, see Proposition \ref{P:mass}.

\medskip

In order to access the cancellations, we will perform a re-summation.
Assuming for simplicity for this exposition that $\langle\cdot\rangle$ has
unit range of dependence,  so that $c$ is supported in the unit ball, we have that
$c_L(z=0)$ does not depend on $L\ge 2$. Hence the second r.~h.~s.~term in (\ref{ao35})
does not contribute. 
%
%
By $L$-periodicity of the correctors, (\ref{ao35}) can be re-summed to
\begin{equation}\label{ao04}
\begin{aligned}
\frac{d}{dL}\langle \xi^*\cdot \bar a\xi\rangle_L=
	\int_{\mathbb{R}^d}dz \big\langle
	\big(a'(\nabla{\phi^*}^{(1)}+\xi^*)\big)(0) \cdot\big(\sum_{k\in\mathbb{Z}^d}k_n\nabla\nabla G(0,z+Lk)\big)
\big(a'(\nabla\phi^{(1)}+\xi)\big)(z)\big\rangle_L \partial_nc(z),
\end{aligned}
\end{equation}
where from now on we use Einstein's convention of summation over repeated spatial indices, 
here $n\in\{1,\cdots,d\}$.
Formula (\ref{ao04}) is our final representation. Clearly, the sum over $k$
is still not absolutely convergent. However, as we shall see in the next subsection,
it converges after homogenization.  

\subsection{Approximation by second-order homogenization}\label{SS:asympt}

In this subsection,
we turn to the a\-sym\-pto\-tics of the representation (\ref{ao04}) for $L\uparrow\infty$.
In particular, we shall argue why first-order homogenization is not sufficient
and give an efficient introduction into second-order correctors.

\medskip

As there is no contribution from $k=0$, and since by our finite range assumption
(for the sake of this discussion), $z$ is constrained to the unit ball, the argument
$z+Lk$ of the Green function satisfies $|z+Lk|\gtrsim L$. Hence we may appeal
to homogenization to replace $G(x,y)$ by $\overline{G}(x-y)$, where $\bar G$ denotes
the fundamental solution of $-\nabla\cdot\bar a\nabla$.
This appears like periodic homogenization as long as $L$ is fixed, 
but in fact amounts to stochastic homogenization since we are interested in $L\uparrow\infty$.
Since we are interested in its gradient, we need to replace
$G$ by the two-scale expansion of $\overline{G}$. (See below for more details on the two-scale expansion.)
Since we are interested in the mixed gradient,
the two-scale expansion acts on both variables. Hence in a first Ansatz, we approximate
\begin{align}\label{ao05}
\nabla\nabla G(0,x)\approx -\partial_{ij}\overline{G}(x)
(e_i+\nabla\phi_i^{(1)})(0)\otimes (e_j+\nabla{\phi_j^*}^{(1)})(x),
\end{align}
where ${\phi_j^*}^{(1)}$ denotes the solution of (\ref{ao01}) with $\xi^*$ replaced by $e_j$.
To leading order, this yields by the periodicity of correctors
\begin{align}\label{ao09}
\nabla\nabla G(0,z+Lk)\approx
-\partial_{ij}\overline{G}(Lk)\;(e_i+\nabla\phi_i^{(1)})(0)\otimes (e_j+\nabla{\phi_j^*}^{(1)})(z).
\end{align}
Applying $\sum_{k\in\mathbb{Z}^d}k_n$ to the r.~h.~s., we see that
it vanishes by parity w.~r.~t.~inversion $k\leadsto -k$. 
This is an indication that the {\it first order} two-scale expansion (\ref{ao05})
is not sufficient and that we have to go to a second-order expansion, which we
shall describe now.

\medskip

We need to replace the first-order version of the two-scale expansion of $\overline{G}$
by its second-order version. We recall the two-scale expansion in its first-order version: 
Given an $\bar a$-harmonic function $\bar u$, one considers 
$u=(1+\phi^{(1)}_i\partial_i)\bar u$ as a good approximation to an $a$-harmonic function.
Indeed, it follows from (\ref{pde:9.1_quad}) that when $\bar u$ is a first-order polynomial, 
$u$ is exactly $a$-harmonic.
In fact, this is a characterization of the first-order correctors $\phi_i^{(1)}$.
Second-order correctors $\phi^{(2)}_{ij}$ can be characterized in a similar way: 
For every $\bar a$-harmonic second-order polynomial $\bar u$, we impose that
$u$ $=(1+\phi_i^{(1)}\partial_i+\phi_{ij}^{(2)}\partial_{ij})\bar u$
is $a$-harmonic\footnote{This does not characterize all components $\phi_{ij}^{(2)}$
separately but only the trace-free and symmetric part of this tensor, where the trace is defined
w.~r.~t.~$\bar a$. Since we apply the two scale expansion only to $\bar a$-harmonic
functions like $\overline{G}$, this is not an issue.}. It is clear from this characterization
that $\phi_{ij}^{(2)}$ depends on the choice of the additive constant in $\phi_i^{(1)}$,
which we now fix through
%
\begin{align}\label{ao08}
\fint_{[0,L)^d}\phi_i^{(1)}=0.
\end{align}
Since for our second-order polynomial $\bar u$ we have 
\begin{align}\label{ao07}
\nabla u=\partial_i\bar u(e_i+\nabla\phi_i^{(1)})+
\partial_{ij}\bar u(\phi^{(1)}_ie_j+\nabla\phi_{ij}^{(2)}),
\end{align}
so that $\nabla\cdot a\nabla u=0$ turns into
$\nabla\partial_i\bar u\cdot a(e_i+\nabla\phi_i^{(1)})$
$+\partial_{ij}\bar u\nabla\cdot a(\phi^{(1)}_ie_j$ $+\nabla\phi_{ij}^{(2)})$
$=0$, and using that $\nabla\cdot\bar a\nabla \bar u=0$, we 
obtain the following standard PDE characterization of $\phi_{ij}^{(2)}$:
\begin{align}\label{ao06}
-\nabla\cdot a(\nabla\phi_{ij}^{(2)}+\phi^{(1)}_ie_j)
=e_j\cdot(a(\nabla\phi_i^{(1)}+e_i)-\bar ae_i).
\end{align} 
Note that 
(\ref{ao06}) is uniquely solvable (up to additive constants) for a periodic $\phi_{ij}^{(2)}$ because the
r.~h.~s.~of (\ref{ao06}) has vanishing average in view of (\ref{Def_Abar_intro}). 
The definition of ${\phi_{ij}^*}^{(2)}$ for the dual medium $a^*$ is analogous.

\medskip

In view of (\ref{ao07}), we thus replace (\ref{ao05}) by
\begin{align}\label{ao16}
\begin{aligned}
\nabla\nabla G(0,x)\approx&-\partial_{ij}\overline{G}(x)
(e_i+\nabla\phi_i^{(1)})(0)\otimes (e_j+\nabla{\phi_j^*}^{(1)})(x)\\
&-\partial_{ijm}\overline{G}(x)
(\phi_i^{(1)}e_m+\nabla\phi_{im}^{(2)})(0)\otimes (e_j+\nabla{\phi_j^*}^{(1)})(x)\\
&+\partial_{ijm}\overline{G}(x)
(e_i+\nabla\phi_i^{(1)})(0)\otimes ({\phi_j^*}^{(1)}e_m+\nabla\phi_{jm}^{*(2)})(x).
\end{aligned}
\end{align}
It is here that the assumption of symmetry of $A$ is convenient: Otherwise,
the instance of $\overline{G}$ in the first r.~h.~s.~term of (\ref{ao16}) would
have to be replaced by $\overline{G}+\overline{G}^{(2)}$ where $\overline{G}^{(2)}$
is the $(1-d)$-homogeneous solution of
$\nabla\cdot(\bar a\nabla\overline{G}^{(2)}+\bar a^{(2)}_m\nabla\partial_m\overline{G})=0$,
where $\bar a^{(2)}$ is the second-order homogenized coefficient, see (\ref{ao44}).
Since $\overline{G}^{(2)}$, as a dipole, is odd w.~r.~t.~point inversion, its
contribution does not vanish as for $\overline{G}$, c.~f.~(\ref{ao09}).
For the analogue of (\ref{ao09}) we now turn to the first-order Taylor expansion (recall $k\not=0$)
\begin{align*}
\nabla\nabla G(0,z+Lk)\approx&-\big(\partial_{ij}\overline{G}(Lk)+z_m\partial_{ijm}\overline{G}(Lk)\big)
(e_i+\nabla\phi_i^{(1)})(0)\otimes (e_j+\nabla{\phi_j^*}^{(1)})(z)\nonumber\\
&-\partial_{ijm}\overline{G}(Lk)
(\phi_i^{(1)}e_m+\nabla\phi_{im}^{(2)})(0)\otimes (e_j+\nabla{\phi_j^*}^{(1)})(z)\nonumber\\
&+\partial_{ijm}\overline{G}(Lk)
(e_i+\nabla\phi_i^{(1)})(0)\otimes ({\phi_j^*}^{(1)}e_m+\nabla\phi_{jm}^{*(2)})(z).
\end{align*}
By the inversion symmetry of $\overline{G}$ and the $-d-1$-homogeneity of $\partial_{ijm}\overline{G}$, this implies
\begin{align}\label{ao19}
\begin{aligned}
\sum_{k\in\mathbb{Z}^d}k_n\nabla\nabla G(0,z+Lk)
\approx& L^{-d-1}\sum_{k\in\mathbb{Z}^d}k_n\partial_{ijm}\overline{G}(k)\Big(-z_m (e_i+\nabla\phi_i^{(1)})(0)\otimes (e_j+\nabla{\phi_j^*}^{(1)})(z)\\
&-(\phi_i^{(1)}e_m+\nabla\phi_{im}^{(2)})(0)\otimes (e_j+\nabla{\phi_j^*}^{(1)})(z)\\
&+(e_i+\nabla\phi_i^{(1)})(0)\otimes ({\phi_j^*}^{(1)}e_m+\nabla\phi_{jm}^{*(2)})(z)\Big).
\end{aligned}
\end{align}
In view of $\bar a\approx a_{\rm hom}$ we finally replace $\overline{G}$, which is still random, 
by the deterministic $G_{\rm hom}$ that may be pulled out of $\langle\cdot\rangle_L$ when
inserting (\ref{ao19}) into (\ref{ao04}). Hence we obtain the approximation
\begin{align}\label{ao12}
\frac{d}{dL}\langle \xi^*\cdot \bar a\xi\rangle_L
\approx L^{-d-1}\Gamma_{{\rm hom},ijmn}\int_{\mathbb{R}^d}dz\,\xi^*\cdot{\mathcal Q}_{Lijm}(z)\xi
\,\partial_nc(z),
\end{align}
where the five-tensor field ${\mathcal Q}_L$ is defined through a combination of three covariances
of quadratic expressions in correctors, see Definition \ref{Def:2},
and where the four-tensor $\Gamma_{\rm hom}$ is formally given by the (borderline) divergent lattice
sum $\sum_{k\in\mathbb{Z}^d}k_n\partial_{ijm} G_{T,{\rm hom}}(k)$, which in line with
the remark at the end of Subsection \ref{SS:formula} we replace by
\begin{align}\label{ao11}
\Gamma_{{\rm hom}}=\lim_{T\uparrow\infty}\Gamma_{{\rm hom},T}\;\;\mbox{where}\;\;
\Gamma_{{\rm hom},Tijmn}:=\sum_{k\in\mathbb{Z}^d}k_n\partial_{ijm} G_{T,{\rm hom}}(k),
\end{align}
with $G_{T,{\rm hom}}$ denoting the fundamental solution of
$\frac{1}{T}-\nabla\cdot a_{\rm hom}\nabla$.

\subsection{Refinement of rigorous result}\label{SS:refine}

We start with the full definition of the tensor field ${\mathcal Q}_L$ appearing in (\ref{ao12}).

\begin{definition}\label{Def:2} Recall the definitions (\ref{pde:9.1_quad})
\& (\ref{ao06}) of first and second-order correctors $\phi_i^{(1)}$ and $\phi_{ij}^{(2)}$, 
and their versions ${\phi_i^*}^{(1)}$ and $\phi_{ij}^{*(2)}$ with $a$ replaced by $a^*$. 
For given vectors $\xi$ and $\xi^*$ we continue to write $\phi^{(1)}=\xi_i\phi_i$ 
and ${\phi^*}^{(1)}=\xi_i{\phi^*}^{(1)}_i$. Consider the random tensor fields
\begin{align}\label{ao45bis}
\begin{aligned}
\xi^*\cdot Q^{(1)}_{ij}(z)\xi
:=\big((\xi^*+\nabla{\phi^*}^{(1)})\cdot a'(e_i+\nabla\phi_i^{(1)})\big)(0)\big((e_j+\nabla{\phi_j^*}^{(1)})\cdot a'
(\xi+\nabla\phi^{(1)})\big)(z),
\end{aligned}
\end{align}
\begin{align}\label{ao46}
\begin{aligned}
\xi^*\cdot Q^{(2)}_{ijm}(z)\xi
:=&
-\big((\xi^*+\nabla{\phi^*}^{(1)})\cdot a'
(\phi^{(1)}_ie_m+\nabla\phi_{im}^{(2)})\big)(0)\big((e_j+\nabla{\phi_j^*}^{(1)})\cdot a'
(\xi+\nabla\phi^{(1)})\big)(z)
\\
&+\big((\xi^*+\nabla{\phi^*}^{(1)})\cdot a'(e_i+\nabla\phi_i^{(1)})\big)(0)\big((\phi^{*(1)}_je_m+\nabla\phi_{jm}^{*(2)})
\cdot a'(\xi+\nabla\phi^{(1)})\big)(z).
\end{aligned}
\end{align}
For any $L$ we consider the ensemble
$\langle\cdot\rangle_L$ from Definition \eqref{Def_cL} and define
\begin{align}\label{ao10}
{\mathcal Q}_{Lijm}(z)
:=-z_m\langle Q_{ij}^{(1)}(z)\rangle_L+\langle Q_{ijm}^{(2)}(z)\rangle_L.
\end{align}
\end{definition}

Here comes the more precise version of Theorem \ref{T:mainpara}, which
consists in making (\ref{ao12}) rigorous:

\begin{theorem}\label{T:main}
Let $d>2$ and $A$ be symmetric. Suppose $\langle\cdot\rangle$ satisfies Assumption \ref{Ass:1} 
and let $a_{\rm hom}$ denote the homogenized coefficient . For all $L$, let
$\langle\cdot\rangle_L$ defined with \eqref{RestrictFourierIntro}, $\bar a$ be defined by
(\ref{Def_Abar_intro}), $\Gamma_{{\rm hom},T}$ defined by (\ref{ao11}), 
and ${\mathcal Q}_L$ be as in Definition \ref{Def:2}. Then the following limits exist:
\begin{align}
\Gamma_{{\rm hom},ijmn}&:=\lim_{T\uparrow\infty}\Gamma_{{\rm hom},Tijmn},\label{ao13}\\
{\mathcal Q}_{ijm}(z)&:=\lim_{L\uparrow\infty}{\mathcal Q}_{Lijm}(z)\quad\mbox{pointwise,
uniformly bounded~in}\;z,\label{ao15}
\end{align}
and the latter only depends on $\langle\cdot\rangle$ (and not the lattice).
Moreover, we have
\begin{align}\label{ao14}
\lim_{L\uparrow\infty}L^{d+1}\frac{d\langle \bar a\rangle_L}{dL}
=\Gamma_{{\rm hom},ijmn}\int_{\mathbb{R}^d}dz{\mathcal Q}_{ijm}(z)\partial_nc(z).
\end{align}
%
%
%
%
%
\end{theorem}

With the tools of this paper, the asymptotics of $\frac{d\langle \bar a\rangle_L}{dL}$
could be characterized up to order $O(L^{-d-\frac{d}{2}})$.
Let us comment on the representation of the leading error term arising from 
(\ref{ao14}), namely
\begin{align}\label{ao26}
d \lim_{L\uparrow \infty} L^{d}\big(a_{\rm hom}-\langle \bar a\rangle_L\big)
=\Gamma_{{\rm hom},ijmn}\int_{\mathbb{R}^d}dz{\mathcal Q}_{ijm}(z)\partial_nc(z).
\end{align}
This representation separates a first factor $\Gamma_{\rm hom}$, which only depends on the
type of the periodic lattice (here cubic) and the homogenized coefficient $a_{\rm hom}$,
from a second factor that only depends on the whole-space ensemble $\langle\cdot\rangle$, via its
covariance function $c$ and covariances involving its first- and second-order
correctors.

\medskip

Let us address the coordinate-free interpretation of ${\mathcal Q}_L$ (and its limit~${\mathcal Q}$)
i.~e.~its transformation behavior. We note that $\xi$, and likewise $\xi^*$, should
be seen as a linear form (rather than a vector), since it gives rise to a coordinate function:
namely affine coordinates via $\xi\cdot x$ and harmonic coordinates via $\phi(x)+\xi\cdot x$.
A glance at the first r.~h.~s.~term in (\ref{ao10}) shows that the indices $i$ and $j$
label the first-order correctors and thus take in linear forms; this is even more
obvious for the index $m$ that takes in a linear form in the $z$-variable.
The second, and likewise the third, r.~h.~s.~term in (\ref{ao10}) is of the same nature since
the second-order corrector naturally takes in a (homogeneous)
second-order polynomial, which can be identified
with linear combinations of (symmetric) tensor products of linear coordinates.
Hence in the language of differential geometry
${\mathcal Q}_L(z)$ is a five-contravariant tensor field -- as it takes in the five linear
forms. 

\medskip

The four-tensor $\Gamma_{{\rm hom},T}$ (and its limit $\Gamma_{\rm hom}$)
allows for a coordinate-free interpretation:
$\Gamma_{{\rm hom},T}$ takes in three vectors (namely the directions of the derivatives of $G_{\rm hom}$)
and renders a vector; as a form it is thus three-covariant and one-contravariant,
and in the traditional notation of differential geometry one would write $\Gamma_{{\rm hom},Tijm}^n$,
highlighting that contraction in (\ref{ao14}) with the three-contravariant
tensor field $\xi^*\cdot\mathcal{Q}^{ijm}\xi$ (with $\xi$, $\xi^*$ fixed) is natural. In view of calculus, $\Gamma_{{\rm hom},T}$
is invariant under permutation of the covariant indices.
There is an isomorphic way of seeing $\Gamma_{{\rm hom},T}$ that allows for an electrostatic
interpretation: $\Gamma_{{\rm hom},T}$ in fact takes in an endomorphism\footnote{An endomorphism is a linear combination
of tensor products of a vector (contra-variant) and a linear form (co-variant).}
and renders a (symmetric) bilinear form.
Indeed, for some endomorphism $B$ of $\mathbb{R}^d$ consider the lattice $B\mathbb{Z}^d$,
and the accordingly periodized version of $G_{T,{\rm hom}}$, that is
$G_{T,{\rm hom},B}$ $:=\sum_{k\in\mathbb{Z}^d}G_{T,{\rm hom}}(x+Bk)$.
We then have
\begin{align}\label{ao22}
\Gamma_{{\rm hom},Tijm}^n v^i v^j u^m \xi_n
=\frac{d}{dt}_{|t=0}v\cdot \nabla^2G_{T,{\rm hom},{\rm id}+t u\otimes\xi}(x=0)v.
\end{align}
Hence $\Gamma_{{\rm hom},Tijm}^{n}$ describes, on the level of the second derivatives, how
(the regular part of) the fundamental solution (infinitesimally)
depends on the lattice w.~r.~t. which one periodizes it.

\subsection{Small contrast regime and non-degeneracy}\label{SS:small}

In this subsection, we (formally) identify the leading order (\ref{ao27})
of the r.~h.~s.~of (\ref{ao26})  in the small-contrast regime. 
We then argue that this leading-order error term
typically does not vanish, even in the high-symmetry case of an isotropic ensemble.

\medskip

We start with the derivation of (\ref{ao27}):
To leading order in a small ellipticity contrast $1-\lambda$, the quantity $\nabla\phi_i^{(1)}$ may
be neglected w.~r.~t. $e_i$; likewise $\phi_i^{(1)}e_m+\nabla\phi_{im}^{(2)}$
may be neglected w.~r.~t. $e_i$. Hence to leading order, (\ref{ao10}) reduces to
\begin{align*}
\xi^*\cdot{\mathcal Q}_{ijm}(z)\xi\approx-z_m\big\langle\xi^*\cdot a'(0)e_i\;
e_j\cdot a'(z)\xi\big\rangle.
\end{align*}
Restricting to the case of scalar $A$ for convenience, the expression further simplifies to
\begin{align*}
{\mathcal Q}_{ijm}(z)
\approx-z_m\,\langle a'(0)a'(z)\rangle\,e_i\otimes e_j.
\end{align*}
Restricting ourselves w.~l.~o.~g.~to ensembles $\langle\cdot\rangle$
with $c(0)=\langle g^2(0)\rangle=\langle g^2(z)\rangle=1$, 
we see that $\langle a'(0)a'(z)\rangle$ depends on 
the Gaussian ensemble $\langle\cdot\rangle$ only through $c(z)$.
We thus write $\langle a'(0)a'(z)\rangle$ $={\mathcal A}'(c(z))$ for
some function ${\mathcal A}$, so that by the chain rule
\begin{align*}
{\mathcal Q}_{ijm}(z)\partial_nc(z)\approx-z_m
\,\partial_n{\mathcal A}(c(z))\,e_i\otimes e_j.
\end{align*}
Normalizing ${\mathcal A}$ such that ${\mathcal A}(0)=0$, we obtain by integration
by parts 
\begin{align*}
\int_{\mathbb{R}^d}dz {\mathcal Q}_{ijm}(z)\partial_nc(z)
\approx\delta_{mn}\,
\,\int_{\mathbb{R}^d}dz{\mathcal A}(c(z))\,e_i\otimes e_j.
\end{align*}
%
%
%
Hence the r.~h.~s.~of (\ref{ao26}) is given by
\begin{align}\label{ao27}
\Big(\lim_{T\uparrow\infty}\sum_{k\in\mathbb{Z}^d}k_m\partial_m\nabla^2 G_{T,{\rm hom}}(k)\Big)
\int_{\mathbb{R}^d}dz{\mathcal A}(c(z))
\end{align}
to leading order in the contrast.

\medskip

It remains to argue that the two factors in (\ref{ao27}) typically do not vanish.
The second factor in (\ref{ao27}) does not vanish in the typical case of 
$A'>0$ and $c\ge 0$. Indeed, by definition of ${\mathcal A}$, we then have
${\mathcal A}'> 0$ and thus ${\mathcal A}(c)>0$ for $c>0$, so that
$\int_{\mathbb{R}^d}dz{\mathcal A}(c(z))>0$ because of $c(0)=1$.

\medskip

For the first factor in (\ref{ao27}), we restrict ourselves to an isotropic ensemble,
namely the case where $c$ is radially symmetric, in addition to $A$ being scalar. 
In line with this, we show that the trace of the first factor in (\ref{ao27}) does not vanish:
\begin{align}\label{ao33}
\lim_{T\uparrow\infty}\sum_{k\in\mathbb{Z}^d}k_m\partial_m\Delta G_{T,{\rm hom}}(k)\not=0.
\end{align}
For our isotropic ensemble, the contravariant two-form $a$ is
invariant in law under orthogonal transformations, and so is $a_{\rm hom}$,
which thus is a multiple of the identity, so that $\triangle$ is
a multiple of $\nabla\cdot a_{\rm hom}\nabla$. Hence by definition of
$G_{T,{\rm hom}}$, (\ref{ao33}) follows from
\begin{align}\label{ao34}
\lim_{T\uparrow\infty}\frac{1}{T}\sum_{k\in\mathbb{Z}^d}k_m\partial_m G_{T,{\rm hom}}(k)
\not=0.
\end{align}
By scaling, we have $G_{T,{\rm hom}}(k)$ $=\frac{1}{\sqrt{T}^{d-2}}
G_{1,{\rm hom}}(\frac{k}{\sqrt{T}})$. Hence we see that the sum in (\ref{ao34}) 
can be interpreted as
a Riemann sum that in the limit $T\uparrow\infty$ converges to the integral
\begin{align*}
\int_{\mathbb{R}^d} dk k_m\partial_m G_{1,{\rm hom}}(k)
=-d\int_{\mathbb{R}^d} dk G_{1,{\rm hom}}(k)=-d,
\end{align*}
where the identity follows from integrating the defining equation $G_{1,{\rm hom}}-
\nabla\cdot a_{\rm hom}\nabla G_{1,{\rm hom}}=\delta$ over $\mathbb{R}^d$.
In particular, we find that $\langle\bar a\rangle_L>a_{\rm hom}$ for $L$ large enough,
which is consistent with numerical simulations in \cite[Fig.\ 7 \& 8]{KanitRVE}, \cite[Tab.\ 3]{SchneiderJosienOtto} and \cite[Tab.\ 5.2]{KhormoskijOtto},
where however types of ensembles are considered that are different from our class.

\subsection{Isotropic ensembles}\label{SS:symm}

In this subsection, we address the case of an isotropic ensemble.
The main step is to characterize the structure of $\Gamma_{\rm hom}$, see (\ref{ao28}),
which amounts to an elementary exercise in representation theory.

\medskip

We recall that by an isotropic ensemble we mean that $c$ is radially symmetric and
that $A$ is scalar. As a consequence, the law of the scalar $a$ under
$\langle\cdot\rangle_L$ is invariant under a change of variables by the octahedral group,
and its law under $\langle\cdot\rangle$ is invariant under the full orthogonal group.
As a consequence, both $\langle\bar a\rangle_L$ and $a_{\rm hom}$ are multiples
of the identity. As a consequence $G_{T,{\rm hom}}$ is radially symmetric.
Hence by definition (\ref{ao11}), 
the 3-covariant and 1-contravariant tensor $\Gamma_{{\rm hom},T}$, like its limit $\Gamma_{\rm hom}$, is
invariant under the octahedral group. Furthermore, it is obviously invariant under the
permutation of its first three (covariant) derivatives.

\medskip

We now derive the (quite restricted) form $\Gamma_{\rm hom}$ takes as a consequence of these symmetries.
We recall that the four-linear form
$\Gamma_{\rm hom}=\Gamma_{\rm hom}(v,v',u,\xi)$ takes in three vectors $v$, $v'$, $u$, and the form $\xi$.
Choosing the standard basis $\{e_m\}_m$ 
and its dual basis $\{e^n\}_n$, by linearity and invariance under the octahedral group, 
it is enough to characterize the two bilinear forms $\Gamma_{\rm hom}(v,v',e_1,e^1)$ and 
$\Gamma_{\rm hom}(v,v',e_2,e^1)$. The first form is invariant under the octahedral subgroup
that fixes $e_1$, which contains in particular reflections $x_i\leadsto
-x_i$ for $i\not=1$. Since the form is symmetric and thus diagonalizable, this first implies that
$e_1$ is an eigenvector, and then that $\{e_1\}^\perp$ is an eigenspace. 
Hence the bilinear form can be written as 
\begin{align}\label{ao23}
\Gamma_{\rm hom}(v,v',e_1,e^1)=\mu_{\perp}v\cdot v'+\mu_{||}(v\cdot e_1)(v'\cdot e_1)
\end{align}
for some constants $\mu_{\perp}$ and $\mu_{||}$.
For the second bilinear form $\Gamma_{\rm hom}(v,v',e_2,e^1)$, the same
argument yields that it has block diagonal form w.~r.~t.~ the span of $\{e_1,e_2\}$ and
its orthogonal complement. In particular we have
\begin{align*}
\Gamma_{\rm hom}(v,v',e_2,e^1)=c v\cdot v'\quad\mbox{for}\;v\cdot e_1=v\cdot e_2=0
\end{align*}
for some constant $c$, which we may recover through
$c=\Gamma_{\rm hom}(e_3,e_3,e_2,e^1)$. By invariance under the octahedral
transformation $x_2\leadsto-x_2$, this expression vanishes, so that in fact
\begin{align}\label{ao29}
\Gamma_{\rm hom}(v,v',e_2,e^1)=0\quad\mbox{for}\;v\cdot e_1=v\cdot e_2=0.
\end{align}
For the same reason, we have
\begin{align}\label{ao31}
\Gamma_{\rm hom}(e_2,e_2,e_2,e^1)=0.
\end{align}
By the permutation symmetry in the first three arguments we obtain from (\ref{ao23})
\begin{align}\label{ao30}
\Gamma_{\rm hom}(e_1,e_1,e_2,e^1)=0\;\;\mbox{and}\;\;\Gamma_{\rm hom}(e_1,e_2,e_2,e^1)
=\Gamma_{\rm hom}(e_2,e_1,e_2,e^1)=\mu_{\perp}.
\end{align}
The statements (\ref{ao29}), (\ref{ao30}), and (\ref{ao31}) combine to
\begin{align*}
\Gamma_{\rm hom}(v,v',e_2,e^1)=
\mu_{\perp}\big((v\cdot e_1)(v'\cdot e_2)+(v'\cdot e_1)(v\cdot e_2)\big).
\end{align*}
A short computation shows that the combination of this with (\ref{ao23}) yields
\begin{align}\label{ao25}
\Gamma_{\rm hom}(v,v',u,\xi)&=\xi.\Big(\mu_{\perp}\big((v\cdot v')u+(v\cdot u) v'+(v'\cdot u)v\big)+(\mu_{||}-2\mu_{\perp}) T(v,v',u)\big),
\end{align}
where we have introduced the trilinear map 
\begin{align*}
T_i(v,v',u)=v_i{v'}_iu_i\quad\mbox{(no summation)},
\end{align*}
which is invariant under permutations and octahedral transformations,
but not under all orthogonal transformations.
In terms of indices, we may rewrite (\ref{ao25}) as
\begin{align}\label{ao28}
\Gamma_{{\rm hom},ijmn}=\mu_{\perp}\big(\delta_{ij}\delta_{mn}+\delta_{im}\delta_{jn}
+\delta_{in}\delta_{jm}\big)
+(\mu_{||}-2\mu_{\perp})\delta_{ijmn}.
\end{align}
Hence in the isotropic case, $\Gamma_{\rm hom}$ is determined by just two numbers.

\medskip

We now turn to the second factor on the r.~h.~s.~of (\ref{ao26}).
As discussed after Definition \ref{Def:2}, 
$\xi^*\cdot{\mathcal Q}_{ijm}(z)\xi$ is a five-covariant tensor field, so that
$\int_{\mathbb{R}^d}dz\xi^*\cdot {\mathcal Q}_{ijm}(z)\xi\partial_nc(z)$ is a five-covariant
and one-contravariant tensor. In our case of an isotropic ensemble, 
${\mathcal Q}$ is invariant under the entire orthogonal group
(not just the discrete octahedral group) as a consequence of $L\uparrow\infty$.
Since the l.~h.~s.~of (\ref{ao26}) is a multiple of the identity,
it is enough to consider the trace of $\int_{\mathbb{R}^d}dz\xi^*\cdot 
{\mathcal Q}_{ijm}(z)\xi\partial_nc(z)$ in $\xi,\xi^*$:
\begin{align*}
Q_{ijmn}:=\int_{\mathbb{R}^d}dz\big(e_1\cdot {\mathcal Q}_{ijm}(z)e_1+\cdots
+e_d\cdot {\mathcal Q}_{ijm}(z)e_d\big)\partial_nc(z),
\end{align*}
which is a three-covariant and one-contravariant tensor, 
still invariant under the (full) orthogonal group.
Since in (\ref{ao26}), it is contracted with a tensor, namely $\Gamma_{\rm hom}$, that is symmetric under
permutation of $i,j,m$, we may pass to the orthogonal projection $Q^{sym}$ of 
$Q$ onto this subspace, which preserves invariance under the orthogonal group. 
Hence as for $\Gamma_{\rm hom}$, we obtain that
$Q^{sym}$ must be of the form (\ref{ao28}). However, while
the first three terms in (\ref{ao28}) are invariant under the entire orthogonal group,
the last is not. Hence $Q^{sym}$ must be of the more restricted form
\begin{align*}
Q_{ijmn}^{sym}=\nu_{\perp}
\big(\delta_{ij}\delta_{mn}+\delta_{im}\delta_{jn}+\delta_{in}\delta_{jm}\big)
\end{align*}
for some constant $\nu_{\perp}$. Hence for an isotropic ensemble, 
the relevant information of the entire six-tensor
$\int_{\mathbb{R}^d}dz\xi^*\cdot {\mathcal Q}_{ijm}(z)\xi\partial_nc(z)$
is the single number $\nu_{\perp}$.

\section{Structure of the proof of Theorem \ref{T:main}}

In this section, we formulate the main intermediate results that lead
to Theorem \ref{T:main}: In Subsection \ref{SS:mass}, we introduce the
massive approximation in order to rigorously derive the analogue of the representation
formula (\ref{ao34}) from Subsection \ref{SS:formula}, see Proposition \ref{P:mass}.
In Subsection \ref{SS:resum} we argue, following Subsection \ref{SS:formula},
that a re-summation allows for removing the massive approximation
in the representation formula, see Proposition \ref{P:2}.
It relies on second-order homogenization, as introduced in Subsection \ref{SS:asympt}.
In Subsection \ref{SS:fromreptoas} we sketch how to pass from the representation
given by Proposition \ref{P:2} to the asymptotics stated in Theorem \ref{T:main}.
This essentially relies on corrector estimates and the estimate of the homogenization error,
see Subsections \ref{SS:stochest} and \ref{SS:twoscale}.
In Subsection \ref{SS:stochest}, we formulate the uniform stochastic estimates on first
and second-order correctors
needed to capture the asymptotics $L\uparrow\infty$, see Proposition \ref{P:3}.
In Subsection \ref{SS:twoscale}, we formulate the stochastic second-order estimate
of the homogenization error, applied to the Green function, see Proposition \ref{P:4}.

\subsection{Massive approximation}\label{SS:mass}

As became apparent in Subsection \ref{SS:formula},
there is divergence in the sum over the periodic cells, see (\ref{ao35}).
We avoid it by replacing the
operator $-\nabla\cdot a\nabla$ by $\frac{1}{T}-\nabla\cdot a\nabla$
where $T<\infty$ will eventually tend to infinity.
This has the desired effect that the corresponding Green's function $G_T(a,x,y)$ and its derivatives
now decay exponentially in $\frac{|y-x|}{\sqrt{T}}$, which can be seen for instance from the homogenization result in Proposition \ref{Prop4Massive:Statement}. The language of ``massive'' approximation arises from field theory where such a zero-order term is
often introduced to suppress an infrared divergence, like here. Assimilating
$m^2$ to the inverse of a time scale $T$ however makes the connection
to stochastic processes, since $\frac{1}{T}-\nabla\cdot a\nabla$ is the generator
of a diffusion-desorption process where $T$ is the time scale of desorption,
and ultimately to parabolic intuition.
As a collateral of the massive approximation,
we have to replace the definitions (\ref{pde:9.1_quad}) and (\ref{Def_Abar_intro}) by
\begin{align}
\frac{1}{T}\phi_{Ti}^{(1)}-\nabla\cdot a(\nabla\phi_{Ti}^{(1)}+e_i)=0,
\quad\bar a_Te_i:=\fint_{[0,L)^d}a(\nabla\phi^{(1)}_{Ti}+e_i);
\label{MassiveEquation}
\end{align}
with analogous definitions for the transposed medium $a^*$. 

\medskip

We collect in the following some estimates on the massive quantities that are useful in the proofs of this section. For notational convenience, in our forecoming estimates, we do not make explicit the dependence of the constants in $d$ and $\lambda$.

\medskip

From Schauder's theory, $\phi_T^{(1)}$ belongs to $C^{1,\alpha}_{loc}(\mathbb{R}^d)$ and 
$$\|(\phi_T^{(1)},\nabla\phi_T^{(1)})\|_{C^{0,\alpha}([0,L)^d)}\leq C(L^{\alpha}[a]_{\alpha})\quad\text{and}\quad \|(\phi_T^{(1)}-\phi^{(1)},\nabla\phi_T^{(1)}-\nabla\phi^{(1)})\|_{C^{0,\alpha}([0,L)^d)}\leq C(L^{\alpha}[a]_{\alpha})T^{-1},$$
where we recall that $[a]_\alpha$ denotes the H\"older semi-norm of $a$. Knowing that $C$ grows at most polynomially in its argument $[a]_{\alpha}$, we deduce from \eqref{fw33Bis} that the estimates above can be converted into, for any $p<\infty$
\begin{equation}\label{BoundCorDependL}
\langle\|(\phi_T^{(1)},\nabla\phi_T^{(1)})\|^p_{C^{0,\alpha}([0,L)^d)}\rangle_L\lesssim_{p,L} 1\quad\text{and}\quad \langle\|(\phi_T^{(1)}-\phi^{(1)},\nabla\phi_T^{(1)}-\nabla\phi^{(1)})\|^p_{C^{0,\alpha}([0,L)^d)}\rangle_L\lesssim_{p,L}T^{-1}.
\end{equation}
Analogously, we obtain at the level of the massive Green functions :
\begin{equation}\label{ConvergenceMassiveQuantities}
\langle\vert\nabla\nabla G_T(x,y)-G(x,y)\vert^p\rangle_L\underset{T\uparrow\infty}{\rightarrow}0\quad\text{for any $x\neq y$,}
\end{equation}
as well as 
\begin{equation}\label{ConvergenceMassiveQuantitiesBis}
\langle\vert(\nabla^3\bar G_T(x),\nabla^2\bar G_T(x))-(\nabla^3\bar G(x),\nabla^2\bar G(x))\vert^p_L\rangle\underset{T\uparrow\infty}{\rightarrow} 0 \quad\text{ for any $x\neq 0$.}
\end{equation}
Finally, we have the following moment bounds on the massive Green function $G_T$
\begin{equation}\label{Eq:MomentGreenAppendix}
\langle\vert\nabla\nabla G_T(x,y)\vert^p\rangle^{\frac{1}{p}}\lesssim_{p,L}(\ln\vert x-y\vert)\vert x-y\vert^{-d-2},
\end{equation}
that we deduce from Proposition \ref{Prop4Massive:Statement} and the bound on the constant-coefficient Green function $\bar{G}_T$ and its derivatives\footnote{the estimates hold in a pathwise way with a constant $C(\|\bar{a}_T\|)$ depending polynomially on its argument and well controlled in moments thanks to the first item of \eqref{BoundCorDependL}} 
\begin{equation}\label{Eq:BoundHomogGT}
\langle\vert\nabla^2 \bar{G}_T(x)\vert^p\rangle^{\frac{1}{p}}_L+\vert x\vert\langle\vert\nabla^3\bar{G}_T(x)\vert^p\rangle^{\frac{1}{p}}_L\lesssim \vert x\vert^{-d-2}\quad\text{for any $x\neq 0$},
\end{equation}
that are uniform in $T\uparrow\infty$.

\medskip

We now can state the massive version of formula (\ref{ao35}). Its rigorous proof will be established in \cite{PriceFormula}.

\begin{proposition}\label{P:mass} It holds
\begin{equation}\label{ao36}
\begin{aligned}
\frac{d}{dL} \langle \xi^*\cdot\bar a_T\xi\rangle_L=&
-\int_{\mathbb{R}^d}dz \big\langle\big(a'
(\nabla\phi_T^{*(1)}+\xi^*)\big)(0)\nabla\nabla G_T(0,z)
\big(a'(\nabla\phi^{(1)}_T+\xi)\big)(z)\big\rangle_L \frac{\partial c_L}{\partial L}(z)\\
&+\frac{1}{2}\big\langle(\nabla\phi_T^{*(1)}+\xi^*)\cdot a''
(\nabla\phi^{(1)}_T+\xi)\big\rangle_L\frac{\partial c_L}{\partial L}(0),
\end{aligned}
\end{equation}
where we recall that $\phi^{(1)}_T=\sum_i \xi_i \phi^{(1)}_{Ti}$.
\end{proposition}
\medskip

The $z$-integral on the r.~h.~s.~of (\ref{ao36}) converges absolutely for $|z|\uparrow\infty$
since the exponential decay of $\nabla\nabla G_T(0,z)$ dominates the
linear growth of $\frac{\partial c_L}{\partial L}(z)$, cf.~(\ref{ao37}). The singularity at
$z=0$ is to be interpreted by duality, using that the other factors are locally smooth in $z$.
%

\subsection{Re-summation}\label{SS:resum}

Following Subsection \ref{SS:asympt}, we now appeal to se\-cond-order homogenization,
which allows for a re-summation. As a by-product of the re-summation,
we may pass to the limit $T\uparrow\infty$ in (\ref{ao36}).
The difficulty with passing to the limit $T\uparrow\infty$ lies in the $\{|z|\ge L\}$-part
of the integral in (\ref{ao36}).
We thus fix a smooth cut-off function $\eta$ for $B_\frac{1}{2}$
in $B_1$, rescale according to
\begin{align*}
\eta_L(z)=\eta({\textstyle\frac{z}{L}}),
\end{align*}
and to split the $z$-integral into the benign near-field part
$\int_{\mathbb{R}^d}dz \eta_L(z)$ and the delicate far-field part
$\int_{\mathbb{R}^d}dz(1-\eta_L)(z)$. On the far-field part, we
appeal to the two-scale expansion (\ref{ao16}).
Hence we have to monitor the homogenization error
\begin{equation}\label{ao47}
\begin{aligned}
\lefteqn{{\mathcal E}(x,y)}\\
&:=
\nabla\nabla G(x,y)+\partial_{ij}\overline{G}(x-y)
(e_i+\nabla\phi_i^{(1)})(x)\otimes (e_j+\nabla{\phi_j^*}^{(1)})(y)\\
&+\partial_{ijm}\overline{G}(x-y)
(\phi_i^{(1)}e_m+\nabla\phi_{im}^{(2)})(x)\otimes (e_j+\nabla{\phi_j^*}^{(1)})(y)\\
&-\partial_{ijm}\overline{G}(x-y)
(e_i+\nabla\phi_i^{(1)})(x)\otimes ({\phi_j^*}^{(1)}e_m+\nabla\phi_{jm}^{*(2)})(y),
\end{aligned}
\end{equation}
where as before $\overline{G}$ denotes the fundamental solution for the
constant-coefficient operator $-\nabla\cdot\bar a\nabla$.
%
%

\medskip

The translation invariance of $\overline{G}$ together with the periodicity of
$\phi^{(1)}$ and $\phi^{(2)}$ allows for a re-summation.
As in Subsection \ref{SS:asympt}, we feed in a zeroth- and first-order Taylor
expansion of $\overline{G}$. This gives rise to the analogue of (\ref{ao11}), namely
\begin{align}\label{ao38}
\overline{\Gamma}_{ijmn}=\lim_{T\uparrow\infty}\overline{\Gamma}_{Tijmn}\;\;\mbox{where}\;\;
\overline{\Gamma}_{Tijmn}:=\sum_{k\in\mathbb{Z}^d}k_n\partial_{ijm}\overline{G}_{T}(k),
\end{align}
where $\overline{G}_T$ denotes the fundamental solution of $\frac{1}{T}-\nabla\cdot\bar a\nabla$.
The existence of this limit follows by the same arguments given for (\ref{ao33}).
The Taylor expansion generates the additional error terms
\begin{align}
\epsilon_{Lijn}^{(1)}(z)&:=\sum_{k\in\mathbb{Z}^d}k_n\big(((1-\eta_L)\partial_{ij}\overline{G})(z+Lk)
-\partial_{ij}\overline{G}(Lk)-z_m\partial_{ijm}\overline{G}(Lk)\big),\label{ao39}\\
\epsilon_{Lijmn}^{(2)}(z)&:
=\sum_{k\in\mathbb{Z}^d}k_n\big(((1-\eta_L)\partial_{ijm}\overline{G})(z+Lk)
-\partial_{ijm}\overline{G}(Lk)\big).\label{ao40}
\end{align}
Thanks to this re-summation, the subtlety of the $T\uparrow\infty$ is limited to
the not absolutely convergent sum in (\ref{ao38}). The
sums in (\ref{ao39}) and (\ref{ao40}) are absolutely convergent since both
summands decay as $|k|^{-(d+1)}$ for $|k|\gg\frac{|z|}{L}$, see (\ref{ao49})
and (\ref{ao50}) for a more quantitative discussion.
Equipped with these definitions, we are now able to express the limit $T\uparrow\infty$
of (\ref{ao36}):

\begin{proposition}\label{P:2} Let $\bar\Gamma$ be as in (\ref{ao38}),
$\epsilon^{(1)}$ and $\epsilon^{(2)}$ as in (\ref{ao39}) \& (\ref{ao40}),
and ${\mathcal E}$ as in (\ref{ao47}).
Let $Q^{(1)}$ and $Q^{(2)}$ be defined
as in (\ref{ao45bis}) and (\ref{ao46}). Then we have
\begin{equation}\label{ao41}
\begin{aligned}
\lefteqn{\frac{d}{dL}\langle\xi^*\cdot\bar a\xi\rangle_L}\\
&=L^{-(d+1)}\int_{\mathbb{R}^d}dz
\big\langle\overline{\Gamma}_{ijmn}\big(\xi^*\cdot Q^{(2)}_{ijm}(z)\xi
-z_m \xi^*\cdot Q^{(1)}_{ij}(z)\xi\big)\big\rangle_L
\partial_nc(z)\\
&+\int_{\mathbb{R}^d}dz
\big\langle\epsilon^{(2)}_{Lijmn}(z)\xi^*\cdot Q^{(2)}_{ijm}(z)\xi
+\epsilon^{(1)}_{Lijn}(z)\xi^*\cdot Q^{(1)}_{ij}(z)\xi\big)\big\rangle_L
\partial_nc(z)\\
&-\int_{\mathbb{R}^d}dz(1-\eta_L)(z)
\big\langle\big(a'(\nabla\phi^{*(1)}+\xi^*)\big)(0) {\mathcal E}(0,z)
\big(a'(\nabla\phi^{(1)}+\xi)\big)(z)\big\rangle_L
\frac{\partial c_L}{\partial L}(z)\\
&-\int_{\mathbb{R}^d}dz\eta_L(z)
\big\langle\big(a'(\nabla\phi^{*(1)}+\xi^*)\big)(0)\nabla\nabla G(0,z)
\big(a'(\nabla\phi^{(1)}+\xi)\big)(z)\big\rangle_L
\frac{\partial c_L}{\partial L}(z)\\
&+\frac{1}{2}\big\langle(\nabla\phi^{*(1)}+\xi^*)\cdot a''
(\nabla\phi^{(1)}+\xi)\big\rangle_L\frac{\partial c_L}{\partial L}(0).
\end{aligned}
\end{equation}
\end{proposition}

Periodic homogenization theory suffices to establish
Proposition \ref{P:2} and in particular to ensure that all five expressions
on the r.~h.~s.~of (\ref{ao41}) are well-defined, including the third one.
Indeed, 
it helps to momentarily think of having rescaled length by the fixed $L$.
This puts us into the context of a $1$-periodic coefficient field $a$,
which in addition is H\"older continuous. By periodic homogenization, we prove in Proposition \ref{Prop4Massive:Statement} (in a more general case for the massive quantity $\mathcal{E}_T$)
\begin{align}\label{ao42}
\sup_{x,y}|y-x|^{d+2}\langle|{\mathcal E}(x,y)|^p\rangle^{\frac{1}{p}}<\infty\quad\text{for any $p<\infty$.}
\end{align}
This estimate yields the absolute convergence of the third term on the r.~h.~s.~of (\ref{ao41}),
since the decay (\ref{ao42}) over-compensates the linear growth of
$\frac{\partial c_L}{\partial L}$. 

\subsection{From representation to asymptotics}\label{SS:fromreptoas}

In order to pass from the representation in Proposition \ref{P:2} to
the asymptotics in Theorem \ref{T:main},
we have to show that the first r.~h.~s.~term of (\ref{ao41}), up to the factor $L^{d+1}$,
converges to the r.~h.~s.~term of (\ref{ao14}), and that the remaining terms are $o(L^{-(d+1)})$. Note that by integration, \eqref{c:6_bis} implies
\begin{equation}\label{FromRtoABoundNabla}
\sup_{x}(1+\vert x\vert^2)^{\frac{d+1}{2}+\alpha}\vert\nabla c(x)\vert<\infty,
\end{equation}
so that from \eqref{ao37} and Proposition \ref{P:3} $i)$, the fifth term is directly of order $L^{-(d+1+2\alpha)}$ which as desired is $o(L^{-(d+1)})$. We now discuss the first four terms. Without loss of generality, we henceforth assume that the exponent $\alpha>0$ is (sufficiently) small.

\medskip

We start with the second term and estimate $\epsilon^{(1)}$,
see (\ref{ao39}): In the range $|k|\ge\frac{|z|}{L}$,
we obtain from Taylor applied to $(1-\eta_L)\partial_{ij}\bar G$
that the summand is estimated by $\vert k\vert|z|^2(L|k|)^{-(d+2)}$
$\le|z|(L|k|)^{-(d+1)}$.
Hence the contribution to the sum from this range is dominated by $\min\{|z|^2L^{-(d+2)},$
$|z|L^{-(d+1)}\}$.
In the other range $|k|\le\frac{|z|}{L}$, the contribution from the middle term
vanishes by parity, the contribution from the last term is estimated by $|z|L^{-(d+1)}$
(by a similar argument to the one that shows that the limit (\ref{ao38}) exists), and the first
term in the summand is estimated by $|k|(\vert k\vert L)^{-d}$ so that its contribution to the sum is also
dominated by $|z|L^{-(d+1)}$. Since this second range is only present for $|z|\ge L$,
we obtain in conclusion
\begin{align}\label{ao49}
|\epsilon^{(1)}_{Lijn}(z)|&\lesssim\min\{|z|^2L^{-(d+2)},|z|L^{-(d+1)}\}.
\end{align}
For the estimate of $\epsilon^{(2)}$, see (\ref{ao40}),
we proceed in a similar way and obtain the stronger estimate
\begin{align}\label{ao50}
|\epsilon^{(2)}_{Lijmn}(z)|&\lesssim |z|L^{-(d+2)}.
\end{align}
We combine the estimates (\ref{ao49}) and (\ref{ao50}) with the corrector estimates of Proposition \ref{P:3} $i)$,
which by definitions (\ref{ao45bis}) and (\ref{ao46}) yield for all $p<\infty$
\begin{align}\label{ao55bis}
\langle|Q^{(1)}_{ij}(z)|^p\rangle^{\frac{1}{p}}+\langle|Q^{(2)}_{ijm}(z)|^p\rangle^{\frac{1}{p}}\lesssim 1.
\end{align}
We now see that Assumption \ref{Ass:1} is just what we need:
By (\ref{FromRtoABoundNabla})
we obtain for the second term in (\ref{ao41})
\begin{align*}
\big|\int_{\mathbb{R}^d}dz
\big\langle\epsilon^{(2)}_{Lijmn}(z)Q^{(2)}_{ijm}(z)
+\epsilon^{(1)}_{Lijn}(z)Q^{(1)}_{ij}(z)\big)\big\rangle_L
\partial_nc(z)\big|\lesssim L^{-(d+2)}+L^{-(d+1+2\alpha)},
\end{align*}
which as desired is $o(L^{-(d+1)})$. In this subsection $\lesssim$ means $\le$ up to a
multiplicative constant that only depends on $d$, $\lambda$,
and the constants implicit in (\ref{c:6_bis}) and (\ref{Reg_A}) of Assumption \ref{Ass:1}.

\medskip

We now turn to the third term on the r.~h.~s.~of (\ref{ao41}).
It follows from Proposition \ref{P:3} $i)$
and Proposition \ref{P:4}, together with (\ref{Reg_A}) in Assumption \ref{Ass:1}, that
\begin{align*}
\big|\big\langle(\nabla\phi^{*(1)}+\xi^*)(0)\cdot a'(0){\mathcal E}(0,z)
a'(z)(\nabla\phi^{(1)}+\xi)(z)\big\rangle_L\big|\lesssim \max\{\mu^{(2)}_d(\vert z\vert),\ln \vert z\vert\}\vert z\vert^{-(d+2)}\lesssim |z|^{-(d+\frac{3}{2})}.
\end{align*}
Inserting (\ref{ao37}) we obtain the following estimate
\begin{align*}
&\big|\int_{\mathbb{R}^d}dz(1-\eta_L(z))\big\langle(\nabla\phi^{*(1)}+\xi^*)(0)\cdot a'(0){\mathcal E}(0,z)
a'(z)(\nabla\phi^{(1)}+\xi)(z)\big\rangle_L\frac{\partial c_L}{\partial L}(z)\big|\nonumber\\
&\lesssim\sum_{k}|k|\int_{\mathbb{R}^d}dz(1-\eta_L)(z)|z|^{-(d+\frac{3}{2})}|\nabla c(z+Lk)|.
\end{align*}
Using \eqref{FromRtoABoundNabla} and splitting the integral into $\{\vert z\vert\leq \frac{1}{2}L\vert k\vert\}$ and its complement we obtain that the $z$-integral is estimated by $(\vert k\vert L)^{-(d+1+2\alpha)}$, which implies that the sum converges and is estimated
by $L^{-(d+1+2\alpha)}$, which as desired is $o(L^{-(d+1)})$.

\medskip

We now address to the first term in (\ref{ao41}). The argument is based on the qualitative result of Corollary \ref{Cor:1} in the following subsection. By the first item in \eqref{ao53} we obtain, by the explicit dependence of $\bar G_T$ and thus $\bar\Gamma$
on $\bar a$,
\begin{align*}
\lim_{L\uparrow\infty}\langle|\bar\Gamma-\Gamma_{\rm hom}|\rangle_L=0.
\end{align*}
Since on the other hand, $\bar\Gamma$ is uniformly bounded (recall that $\bar a$
is confined to the set (\ref{Ellipticite})), and by (\ref{ao55bis}), the convergence of the first term in \eqref{ao41} to the r.~h.~s of \eqref{ao14} follows from the two last items in \eqref{ao53} and the definition \eqref{ao10}.

\medskip

We finally turn to the fourth r.~h.~s term of \eqref{ao41}. We first reinterpret and bound this term using the solution of a PDE : considering $u$ the decaying solution of
$$-\nabla\cdot a\nabla u=\nabla\cdot \big(\eta_L a'(\nabla \phi^{(1)}+\xi)\frac{\partial c_L}{\partial L}\big),$$
we have from Proposition \ref{P:3} $i)$ and \eqref{Reg_A}
\begin{align*}
&\big\vert\int_{\mathbb{R}^d}dz\eta_L(z)
\big\langle(\nabla\phi^{*(1)}+\xi^*)(0)\cdot a'(0)\nabla\nabla G(0,z)
a'(z)(\nabla\phi^{(1)}+\xi)(z)\big\rangle_L
\frac{\partial c_L}{\partial L}(z)\big\vert\\
&=\vert\langle \nabla\phi^{*(1)}+\xi^*)(0)\cdot a'(0)\nabla u(0)\rangle_L\vert\lesssim \langle\vert\nabla u(0)\vert^2\rangle_L^{\frac{1}{2}}.
\end{align*}
We split $u$ into the near-origin and the far-origin contribution $u=u_N+\sum_{1\leq 2^k\leq L}u_{kF}$ with
$$-\nabla\cdot a\nabla u_N=\nabla\cdot \big(\eta_1 a'(\nabla\phi^{(1)}+\xi)\frac{\partial c_L}{\partial L}\big),\quad -\nabla\cdot a\nabla u_{kF}=\nabla\cdot \big((\eta_{2^k}-\eta_{2^{k-1}}) a'(\nabla\phi^{(1)}+\xi)\frac{\partial c_L}{\partial L}\big).$$
The near-origin contribution is directly estimated using Schauder's theory. Indeed, making use of the $\alpha$-Hölder regularity \eqref{fw33Bis} and $\nabla\phi^{(1)}$ (itself a consequence of Schauder's theory applied to the equation \eqref{pde:9.1_quad}), the moment bounds Proposition \ref{P:3} $i)$ as well as \eqref{ao37}, \eqref{c:6_bis} and \eqref{FromRtoABoundNabla} imply :
$$\langle\|\eta_1 a'(\nabla\phi^{(1)}+\xi)\frac{\partial c_L}{\partial L}\|^p_{C^{0,\alpha}(B_1)}\rangle_L^{\frac{1}{p}}\lesssim \sup_{B_1}\big\vert\frac{\partial c_L}{\partial L}\big\vert+ \sup_{B_1}\big\vert\nabla\frac{\partial c_L}{\partial L}\big\vert\lesssim L^{-(d+1+2\alpha)}.$$
Therefore, from Schauder's theory and the energy estimate we deduce 
\begin{align*}
\langle\vert\nabla u_N(0)\vert^2\rangle_L^{\frac{1}{2}}\lesssim \big\langle\big(\int_{B_1}\vert \nabla u_N\vert^2\big)^2\big\rangle_L^{\frac{1}{4}}+\langle\|\eta_1 a'(\nabla\phi^{(1)}+\xi)\frac{\partial c_L}{\partial L}\|^4_{C^{0,\alpha}(B_1)}\rangle_L^{\frac{1}{4}} &\lesssim \big(\int \eta^2_1 \big\vert\frac{\partial c_L}{\partial L}\big\vert^2\big)^{\frac{1}{2}}+L^{-(d+1+2\alpha)}\\
&\lesssim L^{-(d+1+2\alpha)}.
\end{align*}
We now turn to the far-field contribution. Using the Lipschitz estimate of Lemma \ref{L:2} together with an energy estimate and Proposition \ref{P:3} $i)$ as well as \eqref{ao37}, and \eqref{FromRtoABoundNabla}, we derive 
\begin{align*}
\langle \vert \nabla u_F(0)\vert^2\rangle_L^{\frac{1}{2}}\leq \sum_{1\leq 2^k\leq L}\langle\vert\nabla u_{kF}(0)\vert^2\rangle_L^{\frac{1}{2}}\lesssim &\sum_{1\leq 2^k\leq L}\big\langle\big(\fint_{B_{2^{k-2}}}\vert \nabla u_{kF}\vert^2\big)^{2}\big\rangle_L^{\frac{1}{4}}\\
\lesssim & \sum_{1\leq 2^k\leq L} 2^{-\frac{k d}{2}}\big(\int (\eta_{2^k}-\eta_{2^{k-1}})\big\vert\frac{\partial c_L}{\partial L}\big\vert^2\big)^{\frac{1}{2}}\\
\lesssim & (\ln L)L^{-(d+1+2\alpha)}.
\end{align*}
This shows that the fourth r.~h.~s term is $o(L^{-(d+1)})$.

\subsection{Stochastic corrector estimates up to second order}\label{SS:stochest}

As just discussed, the proof of Theorem \ref{T:main} will rely on
estimates of not only the first-order
corrector $\phi^{(1)}_i$, but also its second-order version $\phi^{(2)}_{ij}$,
see part i) of Proposition \ref{P:3}.
Since the period $L$ of the ensemble $\langle\cdot\rangle_L$
tends to infinity, these have to rely on stochastic (and not periodic) homogenization. This is the reason for the restriction to $d>2$ (which is just a more telling way
of saying $d\ge 3$ since it is rather $d=2$ that is borderline): For $d=2$,
the first-order corrector in the whole-space ensemble $\langle\cdot\rangle$ is not stationary,
so that one looses (pointwise) control even of a centered second-order corrector.
Only for $d>2$ one has the middle item in (\ref{ao43}), see for instance \cite{GNO_2019_Correlatedfields}.
For the (limiting) whole-space ensemble $\langle\cdot\rangle$, such higher-order
corrector estimates have
first been established in \cite{Gu} (however sub-optimal in odd dimensions)
and \cite[Theorem 3.1]{BellaFehrmanFischerOtto} (see \cite[Proposition 2.2]{DuerinckxOtto_2019}
for a treatment of any order). These works, like ours, rely on Malliavin calculus and a suitable
spectral gap estimate, as is available under Assumption \ref{Ass:1}. (Incidentally,
the quantitative theory based on finite-range assumptions as started in
\cite{Armstrong} has also been extended to get stochastic estimates on $\phi^{(2)}$
in \cite{LuOttoWang}.) Unfortunately, we cannot simply quote \cite{BellaFehrmanFischerOtto}
since we need the estimate for the periodized ensembles $\langle\cdot\rangle_L$ (uniform
for $L\uparrow\infty$, of course).

\medskip

For Proposition \ref{P:4}, we need to also estimate the flux correctors, both
first and second-order, which we shall recall now. (We also refer to
\cite[Section 2]{DuerinckxOtto_2019} for a compact introduction into all higher-order correctors).
It follows from (\ref{pde:9.1_quad}) and (\ref{Def_Abar_intro}) that
$a(\nabla\phi^{(1)}_i+e_i)-\bar a e_i$ is divergence-free, periodic, and
of zero average. Hence it allows for, in the language of $d=2$, a periodic stream function,
or in the language of $d=3$, a periodic vector potential. For general $d$,
it can be represented in terms of a periodic tensor field $\sigma_i$ with
\begin{align}
a(\nabla\phi^{(1)}_i+e_i)=\bar a e_i+\nabla\cdot\sigma_i^{(1)}\quad
\mbox{and}\quad\sigma_{imn}^{(1)}=-\sigma_{inm}^{(1)},\label{ao55}
\end{align}
where for a (skew symmetric) tensor field $\sigma$,
we write $(\nabla\cdot\sigma)_m$ $:=\partial_n\sigma_{mn}$, as an instance of an
exterior derivative.
Observe that (\ref{ao55}) does not determine $\sigma_{i}^{(1)}$.
Indeed, $\sigma_i^{(1)}$, which can be interpreted
as an alternating $(d-2)$-form, is only determined up to a $(d-3)$-form.
For estimates like in Proposition \ref{P:3}, we choose a suitable (and simple) gauge, that is
$$-\Delta\sigma^{(1)}_{imn}=\partial_m(e_n\cdot a(e_i+\nabla\phi^{(1)}_i))-\partial_n(e_m\cdot a(e_i+\nabla\phi^{(1)}_i)).$$
%
Note also that (\ref{ao06}) can be reformulated in divergence form
\begin{align}\label{ao61}
\nabla\cdot a(\nabla\phi_{ij}^{(2)}+\phi^{(1)}_ie_j)
=(\nabla\cdot \sigma^{(1)}_i)e_j.
\end{align}
This shows that there is a second-order analogue of (\ref{ao55}):
For every coordinate direction $i$,
let the matrix $\bar a^{(2)}$ be defined through
\begin{align}\label{ao44}
\bar a^{(2)}_ie_j:=\fint_{[0,L)^d}a(\nabla\phi^{(2)}_{ij}+\phi^{(1)}_ie_j)
\end{align}
for any $j=1,\cdots,d$, and the periodic tensor field $\sigma^{(2)}_{ij}$ through
\begin{align}\label{ao63}
\begin{array}{c}
a(\nabla\phi^{(2)}_{ij}+\phi^{(1)}_ie_j)=\bar a_i^{(2)} e_j+\sigma_i^{(1)}e_j
+\nabla\cdot\sigma_{ij}^{(2)}\quad\mbox{and}\quad\sigma_{ijmn}^{(2)}=-\sigma_{ijnm}^{(2)}.
\end{array}
\end{align}
The merits of the flux correctors $\sigma_i^{(1)}$ and $\sigma_{ij}^{(2)}$ will become clear in
Subsection \ref{SS:twoscale}. In fact, in that context it will be convenient
to have yet one more object, namely the periodic solution $\omega_i$ of
\begin{align}\label{Eq:RewriteDivergence}
-\triangle\omega_i=\phi_i^{(1)}.
\end{align}

\begin{proposition}\label{P:3} Let $d>2$ and $\langle\cdot\rangle$ satisfy Assumptions \ref{Ass:1};
let $\langle\cdot\rangle_L$ be defined with \eqref{RestrictFourierIntro}. Let $p<\infty$
be arbitrary.

i) We have
\begin{align}\label{ao43}
\langle|\nabla\phi^{(1)}_i|^p\rangle_L^\frac{1}{p}
+\langle|\phi^{(1)}_i|^p\rangle_L^\frac{1}{p}
+\langle|\nabla\phi^{(2)}_{ij}|^p\rangle_L^\frac{1}{p}
\lesssim_p 1.
\end{align}
ii) The random tensor fields $\sigma^{(1)}_i$ and $\sigma^{(2)}_{ij}$
can be constructed such that
\begin{align*}
\langle|\sigma^{(1)}_i|^p\rangle_L^\frac{1}{p}
+\langle|\nabla\sigma^{(2)}_{ij}|^p\rangle_L^\frac{1}{p}
\lesssim_p 1.
\end{align*}
iii) We have for any deterministic periodic vector field $h$ and function $\eta$
\begin{align}\label{ao43Bis}
\langle|\int_{[0,L)^d} h\cdot(q_i-\langle q_i\rangle,\nabla\phi^{(1)}_i)|^p\rangle_L^\frac{1}{p}
\lesssim_p\big(\int_{[0,L)^d}|h|^2\big)^\frac{1}{2} \quad\mbox{and}\quad \langle\vert \int_{[0,L)^d} \eta\phi^{(1)}_i\vert^p\rangle_L^{\frac{1}{p}}\lesssim_p (\int_{[0,L)^d} \vert \eta\vert^{\frac{2d}{d+2}})^{\frac{d+2}{2d}},
\end{align}
where we recall the definition of the flux $q_i:=a(\nabla\phi^{(1)}_i+e_i)$.

\smallskip

iv) We have for all $z$
\begin{align}\label{ao69}
\langle|\phi_{ij}^{(2)}(z)-\phi_{ij}^{(2)}(0)|^p\rangle_L^\frac{1}{p}
+\langle|\sigma_{ij}^{(2)}(z)-\sigma_{ij}^{(2)}(0)|^p\rangle_L^\frac{1}{p}
\lesssim_p\mu_d^{(2)}(|z|),
\end{align}
where
\begin{align}\label{ao71}
\mu_d^{(2)}(r):=\left\{\begin{array}{ccccc}
r&\mbox{for}&2> r&&\\
r^\frac{1}{2}&\mbox{for}&2\leq r&\mbox{and}&d=3\\
\ln^\frac{1}{2}r&\mbox{for}&2\leq r&\mbox{and}&d=4\\
1&\mbox{for}&2\leq r&\mbox{and}&d>4
\end{array}\right\}.
\end{align}

\smallskip

v) We have for all $z$
\begin{align}\label{Eq:RewriteDivergence2}
\langle|\nabla\omega_{i}(z)-\nabla\omega_i(0)|^p\rangle_L^\frac{1}{p}\lesssim_p\mu_d^{(2)}(|z|).
\end{align}

\smallskip

Here $\lesssim$ means $\le$ up to a multiplicative constant that only depends
on $d$, $\lambda$, and the constants implicit in (\ref{c:6_bis}) and (\ref{Reg_A})
of Assumption \ref{Ass:1}. The subscript $p$ indicates an additional dependance.
\end{proposition}
While part i) of Proposition \ref{P:3} is explicitly used in Section \ref{SS:fromreptoas},
the usage of the other parts is more indirect: Part ii) is used in
Corollary \ref{Cor:1}, part iii) is used
to estimate the second-order homogenization error in Lemma \ref{L:3};
and part iv) and v) are used to apply this to the Green function, see Proposition \ref{P:4}.

\medskip

The proof of Proposition \ref{P:3} essentially follows the strategy of
\cite[Section 4]{JosienOtto_2019}\footnote{where we deduce the second second estimate of \eqref{ao43Bis} as follows: Given a deterministic periodic function $\eta$ (where w.~l.~o.~g we may assume that $\int_{[0,L)^d}\eta=0$ since $\int_{[0,L)^d}\phi^{(1)}=0$), we consider
the solution of $-\triangle\zeta=\eta$,
and set $h=\nabla\zeta$. By Sobolev's embedding and 
maximal regularity for the Laplacian in $L^\frac{2d}{d+2}([0,L)^d)$, this follows from the first estimate of \eqref{ao43Bis}} and extends it from first-order to second-order correctors;
the passage from $\langle\cdot\rangle$ to $\langle\cdot\rangle_L$ is only a minor change.
In this paper, we will only establish the most important ingredient for Proposition \ref{P:3},
namely the characterization of stochastic cancellations of the gradient of the
correctors in Lemma \ref{L:5} below. While (\ref{ao82}) reproduces \cite[Proposition 4.1]{JosienOtto_2019},
the new element is its second-order counterpart (\ref{ao83}). The first item of (\ref{ao69})
is a consequence of (\ref{ao83}), adapting
\cite[Proposition 4.1, Part 1, Step 5]{JosienOtto_2019}. The second item of (\ref{ao69})
follows from the analogue of (\ref{ao83}) on the level of the second-order flux (\ref{ao63}),
adapting \cite[Proposition 4.1, Part 2]{JosienOtto_2019}.
\begin{lemma}\label{L:5} Let $d>2$ and $\langle\cdot\rangle$ satisfy Assumptions \ref{Ass:1};
let $\langle\cdot\rangle_L$ be defined with \eqref{RestrictFourierIntro}. For any deterministic
periodic vector field $h$ and any $p<\infty$ we have
\begin{align}
\big\langle\big|\int_{[0,L)^d}h\cdot\nabla\phi^{(1)}_{i}\big|^p\big\rangle_L^\frac{1}{p}
&\lesssim_p\big(\int_{[0,L)^d}|h|^2\big)^\frac{1}{2},\label{ao82}\\
\big\langle\big|\int_{[0,L)^d}h\cdot\nabla\phi^{(2)}_{ij}\big|^p\big\rangle_L^\frac{1}{p}
&\lesssim_p\big(\int_{[0,L)^d}|x|^2_L|h|^2\big)^\frac{1}{2}\label{ao83},
\end{align}
where $\vert x\vert_L:=\inf_{k\in\mathbb{Z}^d}\vert x+kL\vert$.
\end{lemma}
The choice of the origin of the weight in (\ref{ao83}) is of course arbitrary.
We note that (\ref{ao83}) also holds with the weighted $L^2$-norm
$\big(\int_{[0,L)^d}|x|^2_L|h|^2\big)^\frac{1}{2}$ replaced by the $L^q$-norm of the
same scaling, namely $\big(\int_{[0,L)^d}|h|^q\big)^\frac{1}{q}$ with
$q=\frac{2d}{d+2}$. However when passing from (\ref{ao83}) to (\ref{ao69}),
we essentially choose $h=\nabla\bar G(\cdot-z)-\nabla\bar G$,
and in the critical dimension $d=4$, we thus would have
$|f|^q=O(|x|^{-4})$ for $1\ll|x|\ll|z|$ and thus would obtain a power $\frac{3}{4}$ on the
logarithm $\ln|z|$ instead of the optimal power $\frac{1}{2}$. 

\medskip

In establishing (\ref{ao83}), we use the same approach as \cite[Proposition 4.1]{JosienOtto_2019}
for (\ref{ao82}), namely we identify and estimate the Malliavin derivative of
the l.~h.~s.~and then appeal to the spectral gap estimate. However, while
for the first-order result (\ref{ao82}), a buckling is required, it is not necessary
for its second-order counterpart (\ref{ao83}). One can avoid it by appealing to
the quenched Calder\'on-Zygmund estimate, see \cite{DuerinckxOtto_2019} and
\cite[Proposition 7.1 ii)]{JosienOtto_2019}, albeit in the weighted form of Lemma \ref{L:4}:
\begin{lemma}\label{L:4}
Let $d>2$ and let
$\langle\cdot\rangle_L$ be an ensemble of $\lambda$-uniformly elliptic coefficient fields
that are $L$-periodic. Let the random periodic fields $f$ and $u$ be related by
\begin{align*}
\nabla\cdot(a\nabla u+f)=0.
\end{align*}
Let $1<p<p'<\infty$ and $1<q<\infty$. Suppose that
$w$ is arbitrary $L$-periodic function in Muckenhoupt class $A_q$, \footnote{We say that $w$ is in Muckenhoupt class $A_q$, if it satisfies
\begin{equation*}
\big(\fint_{Q} w\big)
\big(\fint_{Q}w^{\frac{-1}{q-1}}\big)^{q-1}
\leq C_1
\end{equation*}
for any cubes $Q$ in $\mathbb{R}^d$,
where $C_1$ is independent of $Q$.}
then the weighted annealed Calder\'on-Zygmund estimates hold, i.~e
\begin{align}\label{ao90-1}
\bigg(\int_{[-\small\frac{L}{2},\small\frac{L}{2})^d}
\big<|\nabla u|^{p}\big>^{\frac{q}{p}}_L
w dx\bigg)^{1/q}
\lesssim \bigg(\int_{[-\small\frac{L}{2},\small\frac{L}{2})^d}
\big<|f|^{p'}\big>^{\frac{q}{p'}}_L
w dx\bigg)^{1/q},
\end{align}
where the implicit multiplicative constant depends on $d,\lambda,p,p^\prime,q$
and the Muckenhoupt norm of $w$.
In particular, for $w=|x|_L^2$, we obtain
\begin{align}\label{ao90}
\bigg(\int_{[-\small\frac{L}{2},\small\frac{L}{2})^d}
\big<|\nabla u|^{p}\big>^{\frac{2}{p}}_L
|x|_L^2 dx\bigg)^{1/2}
\lesssim \bigg(\int_{[-\small\frac{L}{2},\small\frac{L}{2})^d}
\big<|f|^{p'}\big>^{\frac{2}{p'}}_L
|x|_L^2 dx\bigg)^{1/2}.
\end{align}
\end{lemma}
An inspection of the proof of \cite[Proposition 7.1 ii)]{JosienOtto_2019} shows that the
argument extends to the case with a weight in the corresponding
Muckenhoupt class. Indeed, the only essential new ingredient is that this weighted annealed estimate
holds for the constant coefficient operator, i.~e.~the analogue of \cite[Lemma 7.4]{JosienOtto_2019}.
This in turn follows from \cite[Theorem 5, p.219]{Rubio_de_Francia}
or \cite[Theorem 7.1]{EmielLorist2020}. Alternatively, one can derive the weighted
estimate from the unweighted one and the dualized Lipschitz estimate Lemma \ref{L:2},
following the strategy of \cite[Corollary 5]{Gloria_Neukamm_Otto_2019}. We finally mention \cite[Theorem 4.4]{duerinckx2022sedimentation} where such annealed regularity estimates are stated and where the proof will be displayed in \cite{duerinckx2022sedimentationQuantitative}.

\medskip

The limit $L\uparrow\infty$ for the first r.~h.~s.~term in (\ref{ao41}) relies
on the following purely qualitative consequence of Proposition \ref{P:3}.


\begin{corollary}\label{Cor:1}
Let $d>2$ and $\langle\cdot\rangle$ satisfy Assumptions \ref{Ass:1};
let $\langle\cdot\rangle_L$ be defined with \eqref{RestrictFourierIntro}.

i) For $i=1,\cdots,d$ there exists a unique stationary\footnote{By saying that a random field 
$\phi$, i.~e.~a function $\phi=\phi(g,x)$, is stationary
we mean that it is shift-covariant in the sense 
that for all shift vectors $z\in\mathbb{R}^d$, we have
$\phi(g,z+x)=\phi(g(z+\cdot),x)$ for all points $x$ and 
$\langle\cdot\rangle$-almost all fields $g$.} 
random field $\phi_i^{(1)}$ with\footnote{always with the understanding that
the statement at hand holds for all $p<\infty$}
$\langle|\phi_i^{(1)}|^p+|\nabla\phi_i^{(1)}|^p\rangle\lessim_p 1$,
which decays\footnote{which because of $\frac{2d}{d+2}>1$ implies
$\langle\phi_i^{(1)}\rangle=0$} in the sense of
\begin{align}\label{fw30}
\langle|\int\eta\phi_i^{(1)}|^p\rangle^\frac{1}{p}\lesssim_p
\big(\int|\eta|^\frac{2d}{d+2}\big)^\frac{d+2}{2d}
\quad\mbox{for all deterministic functions}\;\eta,
\end{align}
and which satisfies
%
\begin{equation}\label{EquationCor1Corrector}
\nabla\cdot a(\nabla\phi_i^{(1)}+e_i)=0\quad\mbox{a.~s.}.
\end{equation}
ii) For $i,j=1,\cdots,d$ there exists a unique random field $\phi^{(2)}_{ij}$ such that 
$\nabla\phi_{ij}^{(2)}$ is stationary and satisfies 
$\langle|\nabla\phi_{ij}^{(2)}|^p\rangle\lesssim_p 1$,
such that $\phi^{(2)}_{ij}$ has moderate growth\footnote{which implies that 
$\langle\nabla\phi_{ij}^{(2)}\rangle=0$} in the sense\footnote{which by definition
(\ref{ao71}) implies that $\phi^{(2)}_{ij}(0)=0$} of
\begin{align}\label{fw35}
\langle\vert\phi^{(2)}_{ij}(z)\vert^p\rangle^{\frac{1}{p}}\lesssim_p\mu^{(2)}_d(\vert z\vert)
\quad\mbox{for all}\;z,
\end{align}
and which satisfies
\begin{align}\label{ao54}
-\nabla\cdot a(\nabla\phi_{ij}^{(2)}+\phi_i^{(1)}e_j)
=e_j\cdot(a(\nabla\phi^{(1)}_i+e_i)-a_{\rm hom}e_i)
\quad\mbox{a.~s.}.
\end{align}
iii) We have
\begin{align}\label{ao53}
\lim_{L\uparrow \infty}\langle\vert \overline{a}-a_{\text{hom}}\vert\rangle_L=0, \quad\lim_{L\uparrow\infty}\langle Q^{(1)}_{ij}(z)\rangle_L=\langle Q^{(1)}_{ij}(z)\rangle
\quad\mbox{and}\quad
\lim_{L\uparrow\infty}\langle Q^{(2)}_{ijm}(z)\rangle_L=\langle Q^{(2)}_{ijm}(z)\rangle
\quad\mbox{for all}\;z,
\end{align}
where also the r.~h.~s.~integrands are defined by the formulas (\ref{ao45bis}) and (\ref{ao46}).
\end{corollary}

The important element of part i) of Corollary \ref{Cor:1} is the stationarity
of $\phi_i^{(1)}$ itself, not just of $\nabla\phi_i^{(1)}$. Such a result
was first established in \cite[Proposition 2.1]{GloriaOtto_2011} in the case of a discrete medium,
see \cite[Proposition 1]{GloriaOttoESAIMProceedingsandSurveys2015}
for the first result for a continuum medium. For part ii) we note that we cannot
expect $\phi^{(2)}_{ij}$ to be stationary unless $d>4$.
Part iii) is new and relies on a soft
argument based on the uniform bounds of Proposition \ref{P:3}.

\subsection{Estimate of homogenization error to second order,
application to the Green function}\label{SS:twoscale}

A second main role of the corrector estimates of Proposition \ref{P:3}, in particular
the estimate of the flux correctors, is to provide an estimate of the homogenization
error. On our second-order level, this connection relies on identity (\ref{ao62})
involving the two-scale expansion (\ref{ao65}), which we recall now.
Suppose that $u$ and $\bar u$ are related via
\begin{align*}
\nabla\cdot a\nabla u=\nabla\cdot \bar a\nabla\bar u
\end{align*}
and that $\bar u^{(2)}$ is related to $\bar u$ via
\begin{align}\label{ao60}
\nabla\cdot(\bar a\nabla \bar u^{(2)}+\overline{a}^{(2)}_{i}\nabla\partial_{i}\bar u)=0.
\end{align}
Consider the error in the second-order two-scale expansion
\begin{align}\label{ao65}
w:=u-(1+\phi^{(1)}_i\partial_i+\phi^{(2)}_{ij}\partial_{ij})(\bar u+\bar u^{(2)}).
\end{align}
Then $\sigma^{(2)}_{ij}$ allows to write the residuum in divergence form:
\begin{align}\label{ao62}
-\nabla\cdot a\nabla w=\nabla\cdot\big((\phi_{ij}^{(2)}a-\sigma_{ij}^{(2)})
\nabla\partial_{ij}(\bar u+\bar u^{(2)})+\bar a_{i}^{(2)}\nabla\partial_i\bar u^{(2)}\big).
\end{align}
Now the advantage of $A$ and thus $a$ being symmetric becomes apparent:
It implies that the symmetric part of the three-tensor with entries $\bar a^{(2)}_{imn}$
vanishes (see e.~g.~\cite[Lemma 2.4]{DuerinckxOtto_2019}). Since (\ref{ao60})
may be rewritten as $-\nabla\cdot\bar a\nabla \bar u^{(2)}=\bar{a}^{(2)}_{imn}\partial_{imn}\bar u$,
we may assume $\bar u^{(2)}=0$ under our symmetry assumption. Hence (\ref{ao65}) simplifies to
\begin{align}\label{ao68}
w:=u-\big(1+\phi^{(1)}_i\partial_i+(\phi^{(2)}_{ij}-\phi^{(2)}_{ij}(0))\partial_{ij}\big)\bar u
\end{align}
and (\ref{ao62}) may be rewritten as
\begin{align}\label{ao66}
-\nabla\cdot a\nabla w=\nabla\cdot\big((\phi_{ij}^{(2)}-\phi_{ij}^{(2)}(0))a
-(\sigma_{ij}^{(2)}-\sigma_{ij}^{(2)}(0))\big)\nabla\partial_{ij}\bar u.
\end{align}
We are allowed to pass to the centered versions of the second order
(flux) corrector, by which we mean that
$(\phi_{ij}^{(2)},\sigma_{ij}^{(2)})$ is replaced by
$(\phi_{ij}^{(2)}-\phi_{ij}^{(2)}(0),\sigma_{ij}^{(2)}-\sigma_{ij}^{(2)}(0))$,
which we do with (\ref{ao69}) in mind,
since a change by an additive constant does not affect anything
stated so far, and in particular not formula (\ref{ao63}), on which (\ref{ao62})
solely relies.
The upcoming lemma provides an estimate of the second-order
stochastic homogenization error $w$; (\ref{ao67}) is optimal since the rate is governed
by the dimension-dependent expression $\mu_{d}^{(2)}$ with its argument
given by the scale of the r.~h.~s.~$h$, which here is expressed by the diameter $2R$
of its support. Since the estimate is pointwise in the gradient, a logarithm is 
unavoidable\footnote{which however is over-shadowed by $\mu_{d}^{(2)}$ for $d=3$},
and a r.~h.~s.~norm marginally stronger than $\sup|\nabla^2h|$ has to be used\footnote{we pass
to the scalingwise identical norm $R\sup|\nabla^3h|$ for convenience}.
\begin{lemma}\label{L:3}
Let $d>2$ and $\langle\cdot\rangle$ satisfy Assumptions \ref{Ass:1} with symmetric $A$;
let $\langle\cdot\rangle_L$ be defined as in Section \ref{SectionDefEnsembleL}.
Given a deterministic and smooth function $f$ supported in $B_R(y)$ with $y\in\mathbb{R}^d$ and some $R<\infty$,
let $u$ and $\bar u$ be the decaying solutions of
\begin{align}\label{ao70}
-\nabla\cdot a\nabla u=f=-\nabla\cdot \bar a\nabla\bar u.
\end{align}
Then $w$ defined in (\ref{ao68}) satisfies for all $p<\infty$
\begin{align}\label{ao67}
\langle|\nabla w(0)|^p\rangle_L^\frac{1}{p}
\lesssim_p \max\{\mu_{d}^{(2)}(R),\ln(R)\}R\sup|\nabla^2 f|.
\end{align}
Here $\lesssim_{p}$ has the same meaning as in Proposition \ref{P:3}.
\end{lemma}
This pointwise estimate (\ref{ao67}) relies on
a decomposition of the r.~h.~s.~of (\ref{ao66}) into pieces supported on dyadic annuli.
For each piece, we first apply Lemma \ref{L:2} combined with the energy estimate, into which we feed (\ref{ao69}) and
a pointwise bound on $\nabla^3\bar u$ relying on the bounds on the Green function of the constant coefficient operator $\nabla\cdot\overline{a}\nabla$, see (\ref{ao70}).

\medskip

The main goal of this subsection is to estimate the homogenization error
on the level of the Green function, see Proposition \ref{P:4}.
This type of homogenization result with singular r.~h.~s.~has been worked out on the
level of the first-order approximation in \cite[Corollary 3]{Bella_Giunti_Otto_2017}
and extended to second-order in \cite[Theorem 1]{BGO_Multipoles_2017},
where these estimates are derived from estimates on
$(\phi^{(1)}_i,\sigma_i^{(1)})$ and $(\phi^{(2)}_{ij},\sigma^{(2)}_{ij})$
of the type of Proposition \ref{P:3}, however in a pathwise way,
see \cite[Proposition 1]{BGO_Multipoles_2017}.
While equipped with Proposition \ref{P:3}, we could post-process
\cite[Theorem 1]{BGO_Multipoles_2017} to obtain Proposition \ref{P:4},
we take a different, and shorter, route in this paper.
Note that \cite[Theorem 1]{BGO_Multipoles_2017} is not formulated in terms
of the Green function $G$, but in terms of decaying $a$-harmonic functions
in exterior domains. Recovering a statement on the Green function would require
\cite[Lemma 4]{Bella_Giunti_Otto_2017}, which we restate as Lemma \ref{L:1}
below for the convenience of the reader.

\medskip

\begin{proposition}\label{P:4}. Let $d>2$ and $\langle\cdot\rangle$
satisfy Assumptions \ref{Ass:1} with symmetric $A$;
let $\langle\cdot\rangle_L$ be defined as in Section \ref{SectionDefEnsembleL}.
Then we have for ${\mathcal E}$ defined in (\ref{ao47})
\begin{align}\label{ao73}
|y-x|^{d+2}\langle|{\mathcal E}(x,y)|^p\rangle_L^\frac{1}{p}
\lesssim_p\max\{\mu_d^{(2)}(|y-x|),\ln \vert x-y\vert\}
\end{align}
provided $|y-x|\geq 2$ and for all $p<\infty$.
Here $\lesssim_{p}$ has the same meaning as in Proposition \ref{P:3}.

\end{proposition}

Here comes the crucial Lemma that converts weak into strong control.

\begin{lemma}\label{L:1}
Let $\langle\cdot\rangle$ be an ensemble of $\lambda$-uniformly elliptic coefficient
fields\footnote{We will apply it to $\langle\cdot\rangle_L$}.
Let the random function $u$ be $a$-harmonic in the ball $B_R$
of radius $R$. Then we have for all $p<\infty$
\begin{align}\label{ao57}
\big\langle\big(\fint_{B_\frac{R}{2}}|\nabla u|^2\big)^\frac{p}{2}\big\rangle^\frac{1}{p}
\lesssim_{p}\sup_{h\in C^\infty_0(B_R)}\frac{\big\langle|\fint_{B_R}h\cdot\nabla u|^p
\big\rangle^\frac{1}{p}}{R^3\sup|\nabla^3 h|}.
\end{align}
Here $\lesssim_{p}$ has the same meaning as in Proposition \ref{P:3}.
\end{lemma}

Lemma \ref{L:1} amounts to an inner regularity estimate for $a$-harmonic functions $u$,
in terms of the norms $L^p_{\langle\cdot\rangle}L^2_x$ and $W^{-2,1}_xL^p_{\langle\cdot\rangle}$
on the level of the gradient $\nabla u$.
As \cite[Lemma 4]{Bella_Giunti_Otto_2017}, estimate \eqref{ao57} is a consequence of an inner regularity estimate, uniform in $a$, with respect to norms $L^2_x$ and $H^{-n}_x$ (the case $W^{-2,1}_x$ of \eqref{ao57} is obtained for $n>\frac{d}{2}+2$).
However, it strengthens \cite[Lemma 4]{Bella_Giunti_Otto_2017} by restricting the r.~h.~s.\ functional to smooth functions $g$ with \textit{compact} support, \textit{i.e.}, functions that vanish to appropriate order at the boundary.
\\
Nevertheless, it requires only a minor modification of the proof.
It is obtained as a combination of two ingredients.
First, by the Caccioppoli estimate and by an $L^2_x$ interpolation estimate, we may estimate the l.~h.~s.\ of \eqref{ao57} by the $L^2_x$ norm of $w$ for $\Delta^{2n} w = u$.
Second, appealing to the fact that the Dirichlet operator $\Delta^{2n}$ has finite trace for $2n>d$, we may obtain \eqref{ao57}.
This second step differs from \cite[Lemma 4]{Bella_Giunti_Otto_2017}, where the Fourier decomposition was explicitly used to solve $\Delta^{2n} w = u$ (thus, losing the property of compact support).
This argument also shows that the second derivative on $g$, that we need here for our second-order homogenization, could be replaced by any order (properly non-dimensionalized).
\medskip

\medskip

We use Lemma \ref{L:1} only in combination with a second inner regularity estimate,
Lemma \ref{L:2},
which amounts to a Lipschitz estimate. Lipschitz estimates are central in the large-scale
regularity theory in homogenization as initiated by Avellaneda and Lin in the
periodic context, and as introduced by Armstrong and Smart \cite{Armstrong}
to the random context.
\begin{lemma}\label{L:2}
Let $d>2$ and $\langle\cdot\rangle$ satisfy Assumptions \ref{Ass:1};
let $\langle\cdot\rangle_L$ be defined as in Section \ref{SectionDefEnsembleL}.
Let the random function $u$ be $a$-harmonic in the ball $B_R$
of radius $R$. Then we have for all $p',p<\infty$
\begin{align}\label{ao59}
\langle|\nabla u(0)|^{p'}\rangle_L^\frac{1}{p'}
\lesssim_{p',p}\big\langle\big(\fint_{B_R}|\nabla u|^2
\big)^\frac{p}{2}\big\rangle_L^\frac{1}{p}\quad\mbox{provided}\;p'<p.
\end{align}
Here $\lesssim_{p,p'}$ has the same meaning as in Proposition \ref{P:3}.
\end{lemma}
Lemma \ref{L:2} is an easy consequence of the pathwise Lipschitz estimate
\cite[Theorem 1]{Gloria_Neukamm_Otto_2019}. More precisely, we refer to
\cite[(16)]{Gloria_Neukamm_Otto_2019}, which takes the form of
\begin{align*}
\big(\fint_{B_1}|\nabla u|^2\big)^\frac{1}{2}
\lesssim_{d,\lambda}r_*^\frac{d}{2}\big(\fint_{B_R}|\nabla u|^2\big)^\frac{1}{2},
\end{align*}
with the random radius $r_*$ defined in \cite[(12)]{Gloria_Neukamm_Otto_2019}.
It easily follows from the estimates on $(\phi_i^{(1)},\sigma_i^{(1)})$ in Proposition
\ref{P:3} that $\langle r_*^p\rangle_L^\frac{1}{p}$ $\lesssim_{p}1$ for all $p<\infty$.
On the other hand, by standard Schauder theory in $C^{\alpha'}$ we have
\begin{align*}
|\nabla u(0)|\le C(a) \big(\fint_{B_1}|\nabla u|^2\big)^\frac{1}{2},
\end{align*}
where $C$ depends at most polynomially on the local H\"older norm $[a]_{\alpha,B_1}$ (recalling that it satisfies \eqref{fw33Bis}).
Now (\ref{ao59}) follows from combining both estimates;
note that the loss in stochastic integrability
is unavoidable, since it compensates the fact that both $r_*$ and $[a]_{\alpha,B_1}$
are not uniformly bounded.

\medskip

As mentioned, we use Lemma \ref{L:1} only in its form combined with Lemma \ref{L:2}
\begin{corollary}\label{Cor:2}
Let $d>2$ and $\langle\cdot\rangle$ satisfy Assumptions \ref{Ass:1};
let $\langle\cdot\rangle_L$ be defined as in Section \ref{SectionDefEnsembleL}.
Let the random function $u$ be $a$-harmonic in the ball $B_R$
of radius $R$. Then we have for all $p',p<\infty$
\begin{align}\label{ao58}
\langle|\nabla u(0)|^{p'}\rangle_L^\frac{1}{p'}
\lesssim_{p',p}\sup_{h\in C^\infty_0(B_R)}\frac{\big\langle|\fint_{B_R}h\cdot\nabla u|^p
\big\rangle_L^\frac{1}{p}}{R^3\sup|\nabla^3 h|}\quad\mbox{provided}\;p'<p.
\end{align}
Here $\lesssim_{p,p'}$ has the same meaning as in Proposition \ref{P:3}.
\end{corollary}
Corollary \ref{Cor:2} amounts to an inner regularity estimate for $a$-harmonic functions $u$,
in terms of the norms $L^\infty_xL^p_{\langle\cdot\rangle}$
and $W^{-2,1}_xL^p_{\langle\cdot\rangle}$ on the level of the gradient $\nabla u$.
We call this estimate an annealed estimate, since now on both sides of (\ref{ao58}), the probabilistic norm is inside.

\medskip

In this paper, we use Lemma \ref{L:1}, or rather Corollary \ref{Cor:2},
in a more substantial way than it is used in \cite[Corollary 3]{Bella_Giunti_Otto_2017}.
Here comes an outline of the argument for Proposition \ref{P:4}:
We apply Lemma \ref{L:3} with the origin
replaced by a general point $x_0$. Writing
$u(x)=\int_{\mathbb{R}^d}dy\,h(y)\cdot \nabla_yG(x,y)$
and $\bar u(x)=\int_{\mathbb{R}^d}dy\,h(y)\cdot\nabla_y\overline{G}(x,y)$,
this provides control of
\begin{align*}
\int_{\mathbb{R}^d}\,dy (\nabla\cdot h)(y)\big(\nabla_xG(x_0,y)
&-\partial_i\overline{G}(x_0-y)(e_i+\nabla\phi^{(1)}_i(x_0))\nonumber\\
&-\partial_{ij}\overline{G}(x_0-y)(\phi_i^{(1)}(x_0)e_j+\nabla\phi^{(2)}_{ij}(x_0))\big),
\end{align*}
in terms of $\mu_3^{(2)}(R)\sup|\nabla^2h|$ with $2R$ the diameter of $\text{supp }g$;
here we used the centering of $\phi^{(2)}_{ij}$ in $x_0$.
We now fix a point $y_0$ with $|y_0-x_0|\ge 4$ and replace both instances of $\overline{G}(x_0-y)$
by what we obtain from applying the two-scale expansion operator in the $y$-variable
\begin{align*}
1+\phi_m^{*(1)}(y)\frac{\partial}{\partial y_m}
+(\phi^{*(2)}_{mn}(y)-\phi^{*(2)}_{mn}(y_0))\frac{\partial^2}{\partial y_m\partial y_n}.
\end{align*}
Provided $h$ is supported in $B_R(y_0)$ with
$R:=\frac{1}{2}|y_0-x_0|\ge 2$, this preserves the estimate:
While for three out of the four extra terms,
this follows directly from parts i) through iii) of Proposition \ref{P:3},
we need part iv) and an integration by parts in $y$
for the contribution coming from $\phi_i^{*(1)}(y)\partial_{im}\overline{G}(x_0-y)$
$(e_i+\nabla\phi_i^{(1)}(x_0))$.
Keeping only first and second-order terms and recalling
the definition (\ref{ao47}), this yields
\begin{align}\label{ao72}
\big\langle\big|\int_{\mathbb{R}^d}dy\,h(y)
\cdot{\mathcal E}_m(x_0,y)\big|^p\big\rangle_L^\frac{1}{p}\lesssim
\mu_3^{(2)}(R)\sup|\nabla^2 h|
\end{align}
for any $h$ supported in $B_R(y_0)$.
By construction, up to third-order terms, $\mathbb{R}^d-\{x_0\}\ni y\mapsto{\mathcal E}_m(x_0,y)$
is a linear combination of a gradient of an $a^*$-harmonic function,
namely $\frac{\partial G}{\partial x_m}(x_0,y)$,
and gradients of two-scale expansions of $\bar a^*$-harmonic functions,
namely of $\bar u(y)$ $=\partial_i\overline{G}(x_0-y)(\delta_{im}+\partial_m\phi^{(1)}_i(x_0))$
and of $\bar u(y)$ $=\partial_{ij}\overline{G}(x_0-y)$
$(\phi_i^{(1)}(x_0)\delta_{jm}$ $+\partial_m\phi^{(2)}_{ij}(x_0))$.
Hence we may appeal once more to (\ref{ao66}), this time in the $y$-variable
and thus for the dual medium, and with the origin replaced by $y_0$.
We decompose the r.~h.~s.~of (\ref{ao66}) into a far-field
supported on $\mathbb{R}^d-B_R(y_0)$ and a near-field supported
on dyadic annuli centered at $y_0$ of radii $R,\frac{R}{2},\frac{R}{4},\cdots$.
For the near-field contributions, we appeal to the energy estimate followed by Lemma \ref{L:2}.
For the far-field contribution, we use (\ref{ao72}) (in conjunction with the estimate
of the near-field part) by appealing to Corollary \ref{Cor:2},
both with the origin replaced by $y_0$.
%
It is thus Corollary \ref{Cor:2} that
converts the weak control (\ref{ao72}) into pointwise control (\ref{ao73}).


\section{Heuristic result}\label{S:heur}

In this section, we heuristically argue that the strategy of
``periodizing the realizations'' 
leads to a bias that is of order $O(L^{-1})$, 
as announced in (\ref{ao18}). 
We argue that this is the case even for an isotropic\footnote{cf.~Subsection \ref{SS:symm}} 
range-one medium in the small contrast regime\footnote{cf.~Subsection \ref{SS:small}}.
However, rather than extending the $a$ restricted to the RVE periodically,
we extend it by even reflection; this amounts to imposing flux boundary instead
of periodic boundary conditions. More precisely, fixing a direction
$\xi=e_1$, we impose flux
boundary conditions in just one of the directions orthogonal to $e_1$, say $e_d$, and resort 
to the strategy of ``periodizing the ensemble'' in the other $d-1$ directions.
Hence we give the naive strategy a pole position: 
We implement it in the less intrusive
form of reflection rather than periodization -- less intrusive
because it does not create discontinuities in the coefficient field.
Nonetheless, treating just one of the directions in this naive way increases the
bias scaling from $O(L^{-d})$ to $O(L^{-1})$.
Admittedly, the heuristic analysis also becomes simpler by considering the
reflective version, and by implementing it in just one direction.

\medskip

We now make this more precise: Periodizing our stationary centered Gaussian ensemble
in directions $i=1,\cdots,d-1$ on the level of the covariance function amounts to
\begin{align}\label{rr17}
c'_L(z):=\sum_{k'\in\mathbb{Z}^{d-1}\times\{0\}}c(z+Lk'),
\end{align}
cf.~(\ref{Def_cL}). We pick a realization 
according to this $\langle\cdot\rangle_L'$, restrict it to the stripe
$\mathbb{R}^{d-1}\times[0,\frac{L}{2}]$, extend it by even reflection to
$\mathbb{R}^{d-1}\times[-\frac{L}{2},\frac{L}{2}]$ and then extend it $L$-periodically 
in the $e_d$-direction to all of $\mathbb{R}^d$. This defines a (non-stationary) centered
Gaussian ensemble $\langle\cdot\rangle_{L}^{sym}$, 
as such determined by
its covariance function $c^{sym}_L(y,x)=c^{sym}_L(x,y):=\langle g(x) g(y)\rangle_L^{sym}$.
The covariance function is characterized by its connection to $c'_L$ via
\begin{align}\label{rr07}
c^{sym}_L(x,y)=c_L'(x-y)\quad\mbox{provided}\;x,y\in\mathbb{R}^{d-1}\times[0,\tfrac{L}{2}]
\end{align}
and its reflection and translation symmetries\footnote{by $c^{sym}_L(y,x)=c^{sym}_L(x,y)$,
it is enough to state it for the $y$-variable}
\begin{align}
c^{sym}_L(x,y)&=c^{sym}_L(x,y-2y_de_d),\label{rr06}\\
c^{sym}_L(x,y)&=c^{sym}_L(x,y+Le_d),\label{rr09}
\end{align}
see Figure \ref{cLsym}.
\begin{figure}[h]
\includegraphics[scale=0.5]{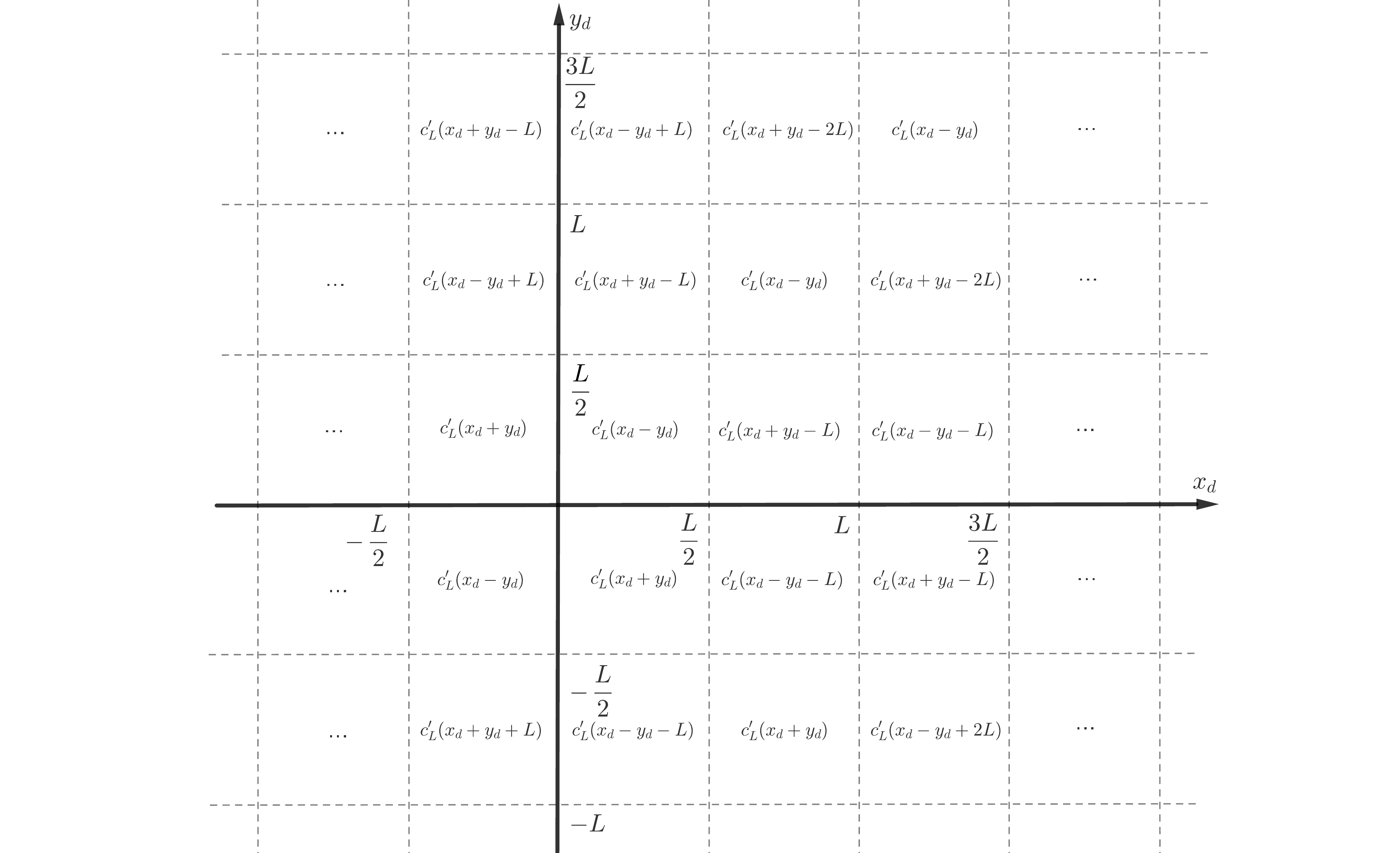}
\caption{Piecewise definition of $c^{sym}_L(x,y)$ as a function of $x_d, y_d$ (with the argument $x'-y'$ of $c'_L$ suppressed)}
\label{cLsym}
\end{figure}
\newline
It obviously inherits stationarity and periodicity in directions $i=1,\cdots,d-1$ from $c'_L$
so that
\begin{align}
c^{sym}_L(x,y)&=c^{sym}_L(x+z',y+z')\quad\mbox{for}\;z'\in\mathbb{R}^{d-1}\times\{0\},\nonumber\\
c^{sym}_L(x,y)&=c^{sym}_L(x,y+Le_i)\quad\mbox{for}\;i=1,\cdots,d.\label{rr04}
\end{align}
Hence comparing $\langle\cdot\rangle_L$ to $\langle\cdot\rangle_L^{sym}$,
we keep (full) periodicity, loose stationarity in direction $e_d$ but gain reflection symmetry
in that direction.
There are two derived symmetries that will play a role, namely
\begin{align}
c^{sym}_L(x,y)&=c^{sym}_L(x,y+(L-2y_d)e_d),\label{rr10}\\
c^{sym}_L(x,y)&=c^{sym}_L(x+\tfrac{L}{2}e_d,y+\tfrac{L}{2}e_d).\label{rr08}
\end{align}
While (\ref{rr10}) is an obvious combination of (\ref{rr06}) and (\ref{rr09}),
(\ref{rr08}) requires an argument, see Appendix \ref{SS:App2}.

\medskip

As for (\ref{Def_a}), we think of $\langle\cdot\rangle_L^{sym}$ as denoting also
the push-forward of the Gaussian ensemble under $a=A(g)$. We now sample $a$
from $\langle\cdot\rangle_L^{sym}$. By construction, 
the scalar $a$ is not only $L$-periodic in every direction, cf.~(\ref{rr04}), but in addition
even under reflection along the hyper planes $\{x_d=0\}$ and $\{x_d=\frac{L}{2}\}$,
cf.~(\ref{rr06}), (\ref{rr10}) and Figure \ref{cLsym}.
These invariances are transmitted to the solution $\phi_i^{(1)}$ of (\ref{pde:9.1_quad}),
and to the flux components $e_i\cdot a(\nabla\phi_i^{(1)}+e_i)$ for any $i\neq d$. On the other
hand, the flux component $e_d\cdot a(\nabla\phi_1^{(1)}+e_1)$ is odd w.~r.~t.~ these reflections and thus vanishes along these two hyper planes. Hence when it
comes to $\phi^{(1)}_1$, the box
$[-\frac{L}{2},\frac{L}{2}]^{d-1}\times[0,\frac{L}{2}]$ can be seen as an RVE with
a flux boundary conditions in direction $e_d$ and periodic boundary conditions
in directions $i=1,\cdots,d-1$. By the above reflection symmetry we have
\begin{align*}
\fint_{[-\frac{L}{2},\frac{L}{2}]^{d-1}\times[0,\frac{L}{2}]}
e_1\cdot a(\nabla\phi_1^{(1)}+e_1)=e_1\cdot \bar a e_1,
\end{align*}
where $\bar a$ is defined as in (\ref{Def_Abar_intro}). We shall heuristically
establish (\ref{ao18}) in the form of
\begin{align*}
e_1\cdot(\langle\bar a\rangle_L^{sym}-a_{\rm hom})e_1=O(L^{-1}).
\end{align*}
More precisely, we shall show that to leading order in $1-\lambda\ll 1$ and $L\gg 1$,
\begin{align}\label{rr11}
e_1\cdot(\langle\bar a\rangle_L^{sym}-a_{\rm hom})e_1\approx -L^{-1}I
\quad\mbox{with}\;I>0.
\end{align}
The sign of the leading-order correction $I$ is consistent with the following heuristics:
The no-flux boundary conditions means that the current is restricted to the stripe $\mathbb{R}^{d-1}\times[0,\frac{L}{2}]$; for $d=2$ and as $L\downarrow 0$, the medium
thus is close to a one-dimensional medium, for which one has the effective conductivity
$\langle a^{-1}\rangle^{-1}\le a_{\rm hom}$ (note that by statistical isotropy, also
$a_{\rm hom}$ is scalar). This is also coherent with the usual observation that RVE approximations with Neumann boundary conditions estimate the actual homogenized coefficient by below, see \cite{kanit2003determination}.

  
\medskip

As for $\langle \bar a\rangle_L$, we shall establish (\ref{rr11})
by monitoring its $L$-derivative.
More precisely, appealing to (\ref{ao14}), we shall establish (\ref{rr11}) in form of
\begin{align}\label{rr12}
\frac{d}{dL}e_1\cdot(\langle\bar a \rangle_L^{sym}
-\langle\bar a\rangle_L)e_1\approx L^{-2} I.
\end{align}
The advantage of monitoring the difference between two ensembles is that we may
use Price's formula in a different, much less subtle, way than in Subsection \ref{SS:formula}.
More precisely, in Appendix \ref{SS:App3}, for a general $[0,L)^d$-periodic centered 
Gaussian ensemble $\langle\cdot\rangle_{L}$ we shall establish the formula
\begin{align}\label{rr14}
\lefteqn{\frac{d}{dL}\langle e_1\cdot\bar a e_1\rangle_L
=-L^{-d}\int_{[-\frac{L}{2},\frac{L}{2})^d}dx\int_{[-\frac{L}{2},\frac{L}{2})^d}dy}\nonumber\\
&\times\big\langle\big(a'(\nabla\phi_1^{(1)}+e_1)\big)(x)\cdot\nabla\nabla G^{per}(x,y)
\cdot\big( a'(\nabla\phi_1^{(1)}+e_1)\big)(y)\big\rangle_L\frac{Dc_L}{\partial L}(x,y),
\end{align}
where the ``material'' derivative of the covariance function $c_L(x,y)$ is defined via
\begin{align}\label{rr13}
\frac{Dc_L}{\partial L}(L\hat x,L\hat y)=\frac{d}{dL}\big(c_L(L\hat x,L\hat y)\big),
\end{align}
and where $G^{per}$ denotes the
Green function associated with the operator $-\nabla\cdot a\nabla$ on the torus $[0,L)^d$, which is unambiguously defined in terms
of its first and mixed derivatives. The present version of (\ref{rr14}) also relies on
\begin{align}\label{rr21}
\frac{Dc_L}{\partial L}(x,y)=0\quad\mbox{for}\;x=y.
\end{align}

\medskip

Note that definition (\ref{rr13}) implies
\begin{align}
\frac{Dc_L}{\partial L}\quad\mbox{is $[0,L)^d$-periodic in both arguments},\label{rr20}\\
\frac{Dc_L}{\partial L}=\frac{\partial c_L}{\partial L}+L^{-1}(x\cdot\nabla_x
+y\cdot\nabla_y)c_L.\label{rr15}
\end{align}
Let us compare formula (\ref{rr14}), which in the presence of stationarity simplifies 
(in the sense that $L^{-d}\int_{[-\frac{L}{2},\frac{L}{2})^d}dy$ 
is replaced by the evaluation at $y=0$), to (\ref{ao35}). The main difference does not lie
in the periodic setting (in view of (\ref{rr20}), 
$\int_{[-\frac{L}{2},\frac{L}{2})^d}dx$ and $G^{per}$ 
can formally be replaced by $\int_{\mathbb{R}^d}dx$ and $G$, respectively),
but in the convective contribution to (\ref{rr15}), which is the generator of rescaling
the space variables of $c$, and thus describes a rescaling of $a$. Indeed, this contribution
vanishes after applying $\int_{\mathbb{R}^d}dx$ (only formally, since 
the integral does not converge absolutely) because the space average $\bar a e_1$ 
$=\lim_{R\uparrow\infty}\fint_{[0,R)^d}a(\nabla\phi_1^{(1)}+e_1)$ is
invariant under rescaling of $a$. 

\medskip

As in Subsection \ref{SS:small},
in the small contrast regime, we have
\begin{align*}
&\langle\big(a'(\nabla\phi_1^{(1)}+e_1)\big)(x)\cdot\nabla\nabla G^{per}(x,y)
\cdot\big( a'(\nabla\phi_1^{(1)}+e_1)\big)(y)\big\rangle_L\approx-\partial_1^2G^{per}_{\rm hom}(x-y)
\langle a'(x) a'(y)\rangle_L,
\end{align*}
so that (\ref{rr14}) simplifies to
\begin{align}\label{rr16}
\frac{d}{dL}\langle e_1\cdot\bar a e_1\rangle_L
&\approx L^{-d}\int_{[-\frac{L}{2},\frac{L}{2})^d}dx\int_{[-\frac{L}{2},\frac{L}{2})^d}dy\partial_1^2 G^{per}_{\rm hom}(x-y)\langle a'(x) a'(y)\rangle_L
\frac{Dc_L}{\partial L}(x,y).
\end{align}
We apply (\ref{rr16}) to both $[0,L)^d$-periodic ensembles, 
$\langle\cdot\rangle_L^{sym}$
and $\langle\cdot\rangle_L$, which we may since
(\ref{rr21}) is satisfied (almost everywhere) for both: Indeed, by reflection
symmetry (\ref{rr06}) and periodicity (\ref{rr09}), it is enough to consider
$x\in\mathbb{R}^{d-1}\times [0,\frac{L}{2})$. Then for $|y-x|$ sufficiently small 
we have $c_L^{sym}(x,y)=c_L(x-y)$ by (\ref{rr07}).
By the finite-range assumption on $c$, this yields
\begin{align*}
c_L^{sym}(x,y)=c_L(x-y)=c(x-y)\quad\mbox{provided}\;|x-y|\;\mbox{is sufficiently small}\;
\mbox{and}\;L\gg 1.
\end{align*}
Hence we have for $y=x$ that
$\frac{\partial}{\partial L}c_L^{sym}(x,y)$ $=\frac{\partial}{\partial L}c_L(0)$ $=0$ 
and $-\nabla_yc^{sym}_L(x,y)$ $=\nabla_xc^{sym}_L(x,y)$ $=\nabla c_L(0)$ $=0$. Introducing in addition the function ${\mathcal A}$ as in Subsection \ref{SS:small} for both
$\langle a'(x) a'(y)\rangle_L$ $={\mathcal A}'(c_L(x,y))$
and $\langle a'(x) a'(y)\rangle_L^{sym}$ $={\mathcal A}'(c_L^{sym}(x,y))$, we obtain from (\ref{rr16})
\begin{align}\label{rr24}
&\frac{d}{dL}e_1\cdot(\langle\bar a\rangle_L^{sym}-\langle\bar a\rangle_L)e_1
\approx L^{-d}\int_{[-\frac{L}{2},\frac{L}{2})^d}dx\int_{[-\frac{L}{2},\frac{L}{2})^d}dy\partial_1^2 G^{per}_{\rm hom}(x-y)\frac{D}{\partial L}\big({\mathcal A}(c_L^{sym})
-{\mathcal A}(c_L)\big)(x,y).
\end{align}

\medskip

There is a cancellation when considering the difference in (\ref{rr24}):
Indeed, by (\ref{rr07}) and (\ref{rr06}) (see also Figure \ref{cLsym}) we have
\begin{align*}
c_L^{sym}(x,y)=c_L'(x-y)\quad\mbox{provided}\;(x_d,y_d)\in[-\frac{L}{2},0]^2\cup[0,\frac{L}{2}]^2.
\end{align*}
Likewise, we obtain from (\ref{Def_cL}), (\ref{rr17}) and the finite range of dependence assumption
\begin{align*}
c_L(x-y)=c_L'(x-y)\quad\mbox{provided}\;(x_d,y_d)\in[-\frac{L}{2},0]^2\cup[0,\frac{L}{2}]^2.
\end{align*}
Hence the integral in (\ref{rr24}) 
reduces to $(x_d,y_d)$ $\in([-\tfrac{L}{2},0]\times[0,\tfrac{L}{2}])$
$\cup([0,\tfrac{L}{2}]\times[\tfrac{L}{2},0])$.
Moreover, since the integrand in (\ref{rr16}) is invariant under permuting
$x$ and $y$, we obtain 
\begin{align}\label{rr18}
\frac{d}{dL}&e_1\cdot(\langle\bar a\rangle_L^{sym}
-\langle\bar a\rangle_L)e_1\approx 2L^{-d}\int_{[-\frac{L}{2},\frac{L}{2})^{d-1}\times[0,\frac{L}{2}]}dx
\int_{[-\frac{L}{2},\frac{L}{2})^{d-1}\times[-\frac{L}{2},0]}dy\nonumber\\
&\times\partial_1^2 G^{per}_{\rm hom}(x-y)
\frac{D}{\partial L}\big({\mathcal A}(c_L^{sym})-{\mathcal A}(c_L)\big)(x,y).
\end{align}
It follows from a combination of symmetries (\ref{rr06})\&(\ref{rr09})\&(\ref{rr08})
that $c_L^{sym}$ is invariant under the inversion at $(\frac{L}{4},-\frac{L}{4})$
in the $(x_d,y_d)$ plane, that is,
\begin{align}\label{rr22}
(x,y)\mapsto ((x',\tfrac{L}{2}-x_d),(y',-\tfrac{L}{2}-y_d)).
\end{align}
By stationarity, periodicity, and (\ref{rr33}), $c_L$ has the same symmetry (\ref{rr22}).
By periodicity and radial symmetry, $(x,y)\mapsto \partial_1^2 G^{per}_{\rm hom}(x-y)$
also has symmetry (\ref{rr22}). Since the triangle in the $(x_d,y_d)$ plane
\begin{align}\label{rr28}
\Delta:=\big\{x_d\ge 0,\quad y_d\le 0,\quad x_d-y_d\le\tfrac{L}{2}\big\}
\end{align}
is such that its (disjoint) union with its image under (\ref{rr22}) renders the rectangle
$[0,\frac{L}{2}]\times[-\frac{L}{2},0]$, (\ref{rr18}) may be rewritten as
\begin{align}\label{rr25}
\frac{d}{dL}&e_1\cdot(\langle\bar a\rangle_L^{sym}
-\langle\bar a\rangle_L)e_1\approx 4L^{-d}\int_{[-\frac{L}{2},\frac{L}{2})^{d-1}}dx'
\int_{[-\frac{L}{2},\frac{L}{2})^{d-1}}dy'\int_{\Delta}dx_ddy_d\nonumber\\
&\times\partial_1^2 G^{per}_{\rm hom}(x-y)
\frac{D}{\partial L}\big({\mathcal A}(c_L^{sym})-{\mathcal A}(c_L)\big)(x,y).
\end{align}
By the finite range assumption and for $L\gg 1$, we have
(where for a stationary ensemble we identify $c(x,y)=c(x-y)$)
\begin{align*}
c_L(x,y)=c_L'(x,y)\quad\mbox{provided}\;(x_d,y_d)\in\Delta,
\end{align*} 
so that in (\ref{rr25}), we may replace $c_L$ by $c_L'$. Likewise, by definition
(\ref{rr07}) and (\ref{rr06}) (see also Figure \ref{cLsym}) we have
\begin{align*}
c_L^{sym}(x,y)=c_L'(x,y-2y_de_d)\quad\mbox{provided}\;(x_d,y_d)\in\Delta,
\end{align*}
so that in (\ref{rr25}), we may express $c_L^{sym}$ in terms of $c_L'$.
Since both
ensembles $\langle\cdot\rangle_L'$ and $\langle\cdot\rangle_L^{sym}$ are
stationary in directions $i=1,\cdots,d-1$, (\ref{rr25}) may
be rewritten as
\begin{align}\label{rr26}
\frac{d}{dL}&e_1\cdot(\langle\bar a\rangle_L^{sym}
-\langle\bar a\rangle_L)e_1\approx 4L^{-1}\int_{[-\frac{L}{2},\frac{L}{2}]^{d-1}}dx'\int_{\Delta}dx_ddy_d\,
\partial_1^2 G^{per}_{\rm hom}(x-(0,y_d))\nonumber\\
&\times
\big(\frac{D}{\partial L}{\mathcal A}(c_L')(x,(0,-y_d))
-\frac{D}{\partial L}{\mathcal A}(c_L')(x,(0,y_d))\big).
\end{align}
By the finite range assumption and for $L\gg 1$, we have
\begin{align*}
c_L'(x,(0,y_d))=c(x,(0,y_d))\quad\mbox{provided}\;x'\in[-\tfrac{L}{2},\tfrac{L}{2}]^{d-1},
\end{align*}
so that the material derivative in (\ref{rr26}) reduces to the convective derivative;
on stationary $c$'s, the convective derivative 
$L^{-1}(x\cdot\nabla_x+y\cdot\nabla_y)$ acts as $L^{-1}z\cdot\nabla$ with $z=x-y$, and
thus as the radial derivative $L^{-1}r\partial_r$ with $R=|z|$. Hence (\ref{rr26}) takes the form
\begin{align}\label{rr35}
\frac{d}{dL}&e_1\cdot(\langle\bar a\rangle_L^{sym}
-\langle\bar a\rangle_L)e_1\approx
4L^{-2}\int_{[-\frac{L}{2},\frac{L}{2}]^{d-1}}dx'\int_{\Delta}dx_ddy_d\,
\partial_1^2 G^{per}_{\rm hom}(x',x_d-y_d)\nonumber\\
&\times
\big((r\partial_r){\mathcal A}(c)(x',x_d+y_d)
-(r\partial_r){\mathcal A}(c)(x',x_d-y_d)\big).
\end{align}

\medskip

We now proceed to a second (and last) approximation. Recall our assumption that $c$
is supported on the unit ball, hence the effective domain of integration
in (\ref{rr35}) is $|x'|\le 1$, next to $0\le x_d-y_d\le\frac{L}{2}$, cf.~(\ref{rr28}).
In this range, we may approximate $G_{\rm hom}^{per}(x',x_d-y_d)$ by $G_{\rm hom}(x',x_d-y_d)$.
After this substitution, we may neglect the restriction $x_d-y_d\le\frac{L}{2}$.
Hence (\ref{rr35}) implies \eqref{rr12} where the $L$-independent quantity $I$ is defined via
\begin{align}\label{rr30}
I&:=4\int_{\mathbb{R}^{d-1}}dx'\int_{0}^{\infty}dx_d\int_{-\infty}^{0}dy_d\,
\partial_1^2 G_{\rm hom}(x',x_d-y_d)\nonumber\\
&\times
\big((r\partial_r){\mathcal A}(c)(x',x_d+y_d)
-(r\partial_r){\mathcal A}(c)(x',x_d-y_d)\big).
\end{align}
In Appendix \ref{SS:App1}, we compute this integral:
\begin{align}\label{rr31}
I=\frac{32}{(d+1)(d-1)}\frac{|B_1'|}{|\partial B_1|}
\int_0^\infty dr{\mathcal A}(c),
\end{align}
where $B_1'$ is the unit ball in $\R^{d-1}$. In particular, we have  $I>0$, see Subsection \ref{SS:small} 
for the explanation why $\mathcal A$ is a nonnegative function (different from 0).


\section{Proofs}



\subsection{Proof of Proposition \ref{P:2}: Limit $T\uparrow \infty$}

The strategy of proof is as follows. First, we reorder the terms of the derivative of $\frac{\dd}{\dd L}\langle \xi^*\cdot\bar{a}_T\xi\rangle_L$ in order to make  appear the ``massive''  analogue (that is, involving the massive operator $\frac{1}{T}-\nabla \cdot a\nabla$) of the r.~h.~s. of \eqref{ao41}.  For this first step, we essentially make rigorous the computations done in Section \ref{SS:asympt}. Second,  we systematically make use of the dominated convergence theorem to obtain the convergence of each term to its massless counterpart, yielding the formula \eqref{ao41}.

\medskip
{\sc Step 1. Formula for $\frac{d}{dL}\langle\xi^*\cdot \bar{a}_T\xi\rangle_L$.}
We establish the ``massive analogue'' of \eqref{ao41}, namely
\begin{equation}\label{ao41bis}
\begin{aligned}
\lefteqn{\frac{d}{dL}\langle\xi^*\cdot\bar {a}_T\xi\rangle_L}\\
&=L^{-(d+1)}\int dz\xi^*\cdot
\big\langle\overline{\Gamma}_{T/L^2ijmn}\,\big(-z_m Q^{(1)}_{Tij}(z)
+Q^{(2)}_{Tijm}(z)\big)\big\rangle_L\xi
\partial_nc(z)\\
&+\int dz
\xi^*\cdot\big\langle\epsilon^{(1)}_{TLijn}(z) Q^{(1)}_{Tij}(z)
+\epsilon^{(2)}_{TLijmn}(z) Q^{(2)}_{Tijm}(z)\big\rangle_L\xi
\partial_nc(z)\\
&-\int dz(1-\eta_L)(z)
\big\langle(\nabla\phi^{*(1)}_T+\xi^*)(0)\cdot a'(0){\mathcal E}_T(0,z)
a'(z)(\nabla\phi^{(1)}_T+\xi)(z)\big\rangle_L
\frac{\partial c_L}{\partial L}(z)\\
&-\int dz\eta_L(z)
\big\langle(\nabla\phi^{*(1)}_T+\xi^*)(0)\cdot a'(0)\nabla\nabla G_T(0,z)
a'(z)(\nabla\phi^{(1)}_T+\xi)(z)\big\rangle_L
\frac{\partial c_L}{\partial L}(z)\\
&+\frac{1}{2}\big\langle(\nabla\phi^{*(1)}_T+\xi^*)\cdot a''
(\nabla\phi^{(1)}_T+\xi)\big\rangle_L\frac{\partial c_L}{\partial L}(0),
\end{aligned}
\end{equation}
where $\epsilon^{(1)}_{TLijn}$ \& $\epsilon^{(2)}_{TLijlm}$
are defined as in \eqref{ao39} \& \eqref{ao40} with $\bar{G}$ replaced
by $\bar{G}_T$, where ${\mathcal E}_T$ is defined like in (\ref{ao47}) with
$G$ \& $\bar{G}$ replaced by $G_T$ \& $\bar{G}_T$ (but with non-massive
first and second-order correctors), and where $Q^{(1)}_T$ \& $Q^{(2)}_T$, which are quartic
expressions in the correctors, are defined like in (\ref{ao45bis}) \& (\ref{ao46}) 
with the first-order correctors $\phi^{(1)}$ \& $\phi^{*(1)}$ (those linear in $\xi$ \& $\xi^*$)
replaced by their massive counterparts $\phi_T^{(1)}$ \& $\phi_T^{*(1)}$ 
(but keeping the non-massive first and second order correctors
$\phi_i^{(1)}$, $\phi_j^{*(1)}$, $\phi_{im}^{(2)}$, $\phi_{jm}^{*(2)}$).
Recall that $\overline{\Gamma}_{T/L^2ijmn}$ is defined in \eqref{ao38}.

\medskip

The starting point is (\ref{ao36}); the second r.~h.~s.~term remains untouched and reappears
as the last term in (\ref{ao41bis}). Writing $\int dz$
$=\int dz(1-\eta_L)(z)$ $+\int dz\eta_L(z)$, we split the
first r.~h.~s.~term in (\ref{ao36}) into a far- and near-field part. The near-field
part remains untouched and reappears as the previous to last term in (\ref{ao41bis}).
In the far-field part, we replace $\nabla\nabla G_T(0,z)$ according to the massive
version of (\ref{ao47}) by ${\mathcal E}_T(0,z)$ and terms involving $\bar{G}_T$.
The contribution with ${\mathcal E}_T(0,z)$ reappears as the third r.~h.~s.~term
in (\ref{ao41bis}). By the massive version of the definition (\ref{ao45bis}) \& (\ref{ao46})
specified above, the terms involving $\bar{G}_T$ give rise to
\begin{align*}
-\int dz (1-\eta_L)(z)\xi^*\cdot\big(
-\big\langle \partial_{ij}\bar{G}_T(z) Q^{(1)}_{Tij}(z)\big\rangle_L
+\big\langle\partial_{ijm}\bar{G}_T(z) Q^{(2)}_{Tijm}(z)\big\rangle_L\big)\xi
\frac{\partial c_L}{\partial L}(z).
\end{align*}
We now insert (\ref{ao37}) and perform the resummation at the end of Subsection 
\ref{SS:formula}, which is based on the periodicity of $Q^{(1)}$ and $Q^{(2)}$
under $\langle\cdot\rangle_L$,
and now is legitimate in view of the good decay properties of $\nabla\nabla G_T$ (see \eqref{Eq:MomentGreenAppendix}) :
\begin{align*}
\int dz\sum_{k}k_n(1-\eta_L)(z)\xi^*\cdot\big(
-\big\langle\partial_{ij }\bar{G}_T(z+Lk) Q^{(1)}_{Tij }(z)\big\rangle_L
+\big\langle\partial_{ijm}\bar{G}_T(z+Lk) Q^{(2)}_{Tijm}(z)\big\rangle_L\big)\xi
\partial_nc(z).
\end{align*}
Using that by parity, $\sum_{k}k_n\partial_{ij}\bar{G}_T(Lk)$ $=0$, we now
appeal to the massive version of the definitions (\ref{ao39}) \& (\ref{ao40}).
This gives rise to the second r.~h.~s.~of (\ref{ao41bis}) and the leading order term
\begin{align*}
\int dz\xi^*\cdot\Big\langle\sum_{k}k_n\partial_{ijm}\bar{G}_T(Lk)
\,\big(-z_m Q^{(1)}_{Tij}(z)+Q^{(2)}_{Tijm}(z)\big)\Big\rangle_L\xi
\partial_n c(z).
\end{align*}
It remains to insert the definition (\ref{ao38}) and appeal to the scaling
$\bar{G}_T(Lx)=L^{2-d}\bar{G}_{T/L^2}(x)$.


\medskip

{\sc Step 2. Limit $T\uparrow \infty$. }
We now show that each term in \eqref{ao41bis} passes to the limit as $T\uparrow \infty$ 
and converges to its massless counterpart. To do so, we need to establish that this limit 
makes sense for each of the five r.~h.~s.\ terms of \eqref{ao41bis}; in this task, 
the dominated convergence theorem is our main tool. Note that \eqref{ao49} and \eqref{ao50} hold uniformly in $T$, 
at the level of the massive quantities. Therefore, combined with the bounds and convergences of the massive quantities \eqref{BoundCorDependL}, \eqref{ConvergenceMassiveQuantities}, \eqref{ConvergenceMassiveQuantitiesBis}, \eqref{Eq:MomentGreenAppendix} and Proposition \ref{Prop4Massive:Statement}, 
the second, third, fourth and fifth r.~h.~s.~terms of \eqref{ao41bis} converge to their 
massless counterparts as $T\uparrow \infty$.  Consequently, the subtle part is in the first 
r.~h.~s.~term of \eqref{ao41bis}, that we treat in details.
  
\smallskip
  
In the sequel, $L \geq 1$ is fixed.  We  prove that
  \begin{align}\label{Step:1.1}
  \begin{aligned}
    &\lim_{T \uparrow \infty} \int dz
\big\langle\overline{\Gamma}_{T/L^2 ijmn}\big(\xi^*\cdot Q^{(2)}_{Tijm}(z)\xi
-z_m \xi^*\cdot Q^{(1)}_{Tij}(z)\xi\big)\big\rangle_L
\partial_nc(z)\\
    &= \int dz
\big\langle\overline{\Gamma}_{ijmn}\big(\xi^*\cdot Q^{(2)}_{ijm}(z)\xi
-z_m \xi^*\cdot Q^{(1)}_{ij}(z)\xi\big)\big\rangle_L
\partial_nc(z).
  \end{aligned}    
  \end{align}
We claim that the only additional ingredient is the well-posedeness of $\overline{\Gamma}_{ijmn}:=\lim_{T\uparrow \infty}\overline{\Gamma}_{Tijmn}$ along with the bound 
  \begin{align}\label{Lim_GammaLT}
    \lt|\overline{\Gamma}_{ijmn}\rt| \leq \sup_{T \geq 1} |\overline{\Gamma}_{Tijmn}| \lesssim 1 \text{ $\langle\cdot\rangle_L$-almost-surely},
  \end{align}
 where we recall \eqref{ao38}
$$\overline{\Gamma}_{Tijmn}:=\sum_{k\in\mathbb{Z}^d} k_n\partial_{ijm}\bar{G}_T(k).$$
Indeed, thanks to the assumption \eqref{c:6_bis} on $c$,  the bounds on the correctors \eqref{BoundCorDependL}, and \eqref{Lim_GammaLT}, the integrand of the l.~h.~s integral in \eqref{Step:1.1} is bounded (uniformly in $T$) by $(1+\vert z\vert)^{-d-2\alpha}$. We then conclude using the convergences \eqref{ConvergenceMassiveQuantities} together with the Lebesgue convergence theorem.

\smallskip

Here comes the argument for \eqref{Lim_GammaLT}. 
We fix a smooth compactly supported $\eta$ with $\eta=1$ on the unit cube
$(-\frac{1}{2},\frac{1}{2})^d$ that we use it to split the lattice, which we interpret
as a Riemann sum:
\begin{align}\label{bb01}
\lefteqn{\sum_{k\not=0}k_n\partial_{ijm}\bar{G}_T(k)}\nonumber\\
&=\int_{\mathbb{R}^d\backslash(-\frac{1}{2},\frac{1}{2})^d}dx\eta(x)x_n\partial_{ijm}\bar{G}_T(x)
+\int dx (1-\eta)(x)x_n\partial_{ijm}\bar{G}_T(x)\nonumber\\
&+\sum_{k\not=0}\big(k_n\partial_{ijm}\bar{G}_T(k)
-\int_{k+(-\frac{1}{2},\frac{1}{2})^d}dx x_n\partial_{ijm}\bar{G}_T(x)\big).
\end{align}
The first r.~h.~s.~integral effectively extends over a compact subset of $\mathbb{R}^d\backslash\{0\}$ 
and thus obviously converges for $T\uparrow\infty$, thanks to \eqref{ConvergenceMassiveQuantities}. On the second r.~h.~s.~integral in (\ref{bb01})
we perform two integrations by parts:
\begin{align*}
\int dx (1-\eta)(x)x_n\partial_{ijm}\bar{G}_T(x)=\int dx\big(-\partial_j\eta(x)\delta_{mn}\partial_i\bar{G}_T(x)
+\partial_m\eta(x)x_n\partial_{ij}\bar{G}_T(x)\big).
\end{align*}
Again, the r.~h.~s.~integral effectively extends over a compact subset of $\mathbb{R}^d\backslash\{0\}$
and converges for $T\uparrow\infty$, thanks to \eqref{ConvergenceMassiveQuantities}. We finally turn to the last contribution 
in (\ref{bb01}) where each summand has a limit $T\uparrow\infty$. This extends to
the sum because of dominated convergence: Each summand is dominated, in absolute value, 
by the Lipschitz
norm of $x\mapsto x_n\partial_{ijm}\bar{G}_T(x)$ on the translated cube
$k+(-\frac{1}{2},\frac{1}{2})^d$, which by the uniform-in-$T$ decay of the derivatives of $\bar{G}_T$ (see \eqref{Eq:BoundHomogGT}),  gives an expression that is summable in $k\in\mathbb{Z}^d\backslash\{0\}$.
%


\subsection{Proof of Lemma \ref{L:5}: Fluctuation estimates}

As announced above, we show only \eqref{ao83} by closely following \cite{JosienOtto_2019}.
The only difference is that we appeal not only to the annealed Calder\'on-Zygmund estimates as  in \cite{JosienOtto_2019}, but also to the annealed weighted estimates contained in Lemma \ref{L:4}.
\\
%
%
%
\newline
For a deterministic and periodic vector field $h$, we consider the random variable of zero average
\begin{align*}
F:=\int_{[0,L)^d} h\cdot\nabla\phi_{ij}^{(2)}.
\end{align*}
We employ on it the spectral gap inequality (\textit{cf.} \cite[Lem.\ 3.1]{JosienOtto_2019}), which, combined with the Bochner estimate
(assuming that $p \geq 2$), reads
\begin{equation}\label{SG}
\langle |F|^{p}\rangle_L^{\frac{1}{p}} 
\lesssim
\Big( \int_{[0,L)^d} \big\langle \big| \frac{\partial F}{\partial g}\big|^p\big\rangle_L^{\frac{2}{p}} \Big)^{\frac{1}{2}},
\end{equation}
where $\frac{\partial F}{\partial g}=\frac{\partial F(g)}{\partial g(x)}$ is the Fr\'echet (or functional or vertical or Malliavin) derivative on $\LL^2([0,L)^d)$ of $F$ w.~r.~t.\ $g$ defined by, for all periodic perturbation $\delta g\in \LL^2([0,L)^d)$
$$\lim_{\varepsilon\downarrow 0}\frac{F(g+\varepsilon\delta g)-F(g)}{\varepsilon}:=\int_{[0,L)^d} dx\,\delta g(x)\frac{\partial F(g)}{\partial g(x)}.$$
(Since $\LL^2([0,L)^d)$ is a Hilbert space, this Fréchet derivative is actually a gradient.)
We split the proof into three steps. First, we establish that the Fr\'echet derivative of $F$ is given by
\begin{align}\label{ao80}
\frac{\partial F}{\partial g}
=\nabla v\cdot a'(\nabla\phi^{(2)}_{ij}+\phi^{(1)}_ie_j)
-(\nabla w_j+ve_j)\cdot a'(\nabla\phi_i^{(1)}+e_i),
\end{align}
where $v$ and $w_j$ are defined through (\ref{ao78}) and (\ref{ao79}) below.
Next, we show that the annealed estimates of Lemma~\ref{L:4} imply 
\begin{align}
\label{ao81-1}
&\big(\int_{[0,L)^d}\langle|\nabla v|^{2p}\rangle_L^\frac{1}{p}\big)^\frac{1}{2}
\lesssim\big(\int_{[0,L)^d}|h|^2\big)^\frac{1}{2},
\\
\label{ao81}
&\big(\int_{[0,L)^d}\langle|\nabla w_j+ve_j|^{2p}\rangle_L^\frac{1}{p}\big)^\frac{1}{2}
\lesssim\big(\int_{[0,L)^d}|x|^2_L|h|^2\big)^\frac{1}{2},
\end{align}
where we recall that $\vert x\vert_L=\inf_{k\in\mathbb{Z}^d}\vert x+kL\vert$. Last, we insert \eqref{ao80} into \eqref{SG} and we appeal to the Cauchy-Schwarz inequality (we also employ \eqref{Reg_A}), to the effect of
\begin{align*}
	\langle |F|^{p}\rangle^{\frac{1}{p}}
	&\lesssim
	\Big( \int_{[0,L)^d} 
	\big(\langle |\nabla v ^{2p}|\rangle^{\frac{1}{p}}
	\langle |\nabla\phi^{(2)}_{ij}+\phi^{(1)}_ie_j|^{2p}\rangle^{\frac{1}{p}}
	\\
	&\qquad \qquad + 
	\langle |\nabla w_j+ve_j|^{2p}\rangle^{\frac{1}{p}}
	\langle |\nabla\phi_i^{(1)}+e_i|^{2p}\rangle^{\frac{1}{p}}\big)
	\Big)^{\frac{1}{2}}.
\end{align*}
Invoking \eqref{ao43} and recalling \eqref{ao81-1} and \eqref{ao81} finally yields the desired estimate \eqref{ao83}.

\medskip

{\sc Step 1. Argument for \eqref{ao80}.}
We give ourselves infinitesimal (periodic) perturbation $\delta g \in \LL^2([0,L)^d)$ of $g$. 
In view of (\ref{pde:9.1_quad}) and (\ref{ao08}), it generates a perturbation $\delta\phi_i^{(1)}$ 
characterized by
\begin{align}\label{ao76}
\nabla\cdot\big(a\nabla\delta\phi_i^{(1)}+\delta g a'(\nabla\phi_i^{(1)}+e_i)\big)=0
\quad\mbox{and}\quad\fint_{[0,L)^d}\delta\phi_i^{(1)}=0.
\end{align}
In view of (\ref{ao06}), this in turn generates the perturbation $\nabla\delta\phi_{ij}^{(2)}$ 
characterized by
\begin{align}\label{ao75}
\begin{aligned}
-\nabla\cdot\big(a(\nabla\delta\phi_{ij}^{(2)}+\delta\phi_i^{(1)}e_j)
+\delta g a'(\nabla\phi^{(2)}_{ij}+\phi^{(1)}_ie_j)\big)
=Pe_j\cdot\big(a\nabla\delta\phi_i^{(1)}+\delta g a'(\nabla\phi_i^{(1)}+e_i)\big),
\end{aligned}
\end{align}
where $P$ denotes the ($L^2$-orthogonal) projection onto functions of vanishing
spatial mean, i.~e.~$Ph=h-\fint_{[0,L)^d}h$. The form of (\ref{ao75}) motivates
the introduction of the periodic function $v$ defined through
\begin{align}\label{ao78}
\nabla\cdot(a^*\nabla v+h)=0\quad\mbox{and}\quad\fint_{[0,L)^d}v=0,
\end{align}
so that, by testing  \eqref{ao78} with $\delta\phi^{(2)}_{ij}$  and \eqref{ao75} with $v$ , we obtain the representation for $\delta F:=\int_{[0,L)^d}h\cdot\nabla\delta\phi^{(2)}_{ij}$:
\begin{align}\label{ao77}
\begin{aligned}
\delta F=\int_{[0,L)^d}\Big(\nabla v\cdot\big(a \delta\phi_i^{(1)}e_j
+\delta g a'(\nabla\phi^{(2)}_{ij}+\phi^{(1)}_ie_j)\big) -ve_j\cdot\big(a\nabla\delta\phi_i^{(1)}+\delta g a'(\nabla\phi_i^{(1)}+e_i)\big)\Big).
\end{aligned}
\end{align}
This in turn prompts the introduction of a second auxiliary periodic function $w_j$ of zero mean
\begin{align}\label{ao79}
-\nabla\cdot a^*(\nabla w_j+ve_j)=Pe_j\cdot a^*\nabla v,
\end{align}
so that by testing (\ref{ao76}) with $w_j$ and \eqref{ao79} with $\delta\phi_i^{(1)}$,  we may eliminate $\delta\phi_i^{(1)}$ in (\ref{ao77}) and recover \eqref{ao80} in form of
\begin{align*}
\delta F
=\int_{[0,L)^d}\delta g\Big(
\nabla v\cdot a'(\nabla\phi^{(2)}_{ij}+\phi^{(1)}_ie_j)
-(ve_j + \nabla w_j)\cdot a'(\nabla\phi_i^{(1)}+e_i) \Big).
\end{align*}
%

%
\medskip

{\sc Step 2. Argument for \eqref{ao81-1} and \eqref{ao81}.}
Notice first that \eqref{ao81-1} is a direct consequence of Lemma \ref{L:4} with weight $w=1$ applied to $v$ satisfying \eqref{ao78}. Therefore, it remains to establish \eqref{ao81}.
In this perspective, we introduce the (gradient) field $h_j$ such that the r.~h.~s.\ of \eqref{ao79} reads
$Pe_j\cdot a^*\nabla v$ $=\nabla\cdot h_j$.
As a consequence of annealed unweighted estimates on $\nabla(-\nabla\cdot a^*\nabla)^{-1}\nabla\cdot$ (namely, Lemma \ref{L:4} with weight $w=1$), we get
\begin{align}\label{Num::5}
\big(\int_{[0,L)^d}\langle|\nabla w_j|^{2p}\rangle_L^\frac{1}{p}\big)^\frac{1}{2}
\lesssim
\big(\int_{[0,L)^d}\langle|h_j|^{2p} + |v|^{2p}\rangle_L^\frac{1}{p}\big)^\frac{1}{2}.
\end{align}
We now claim the following annealed Hardy inequality:
\begin{equation}\label{Num::3}
\big(\int_{[0,L)^d}\langle|v|^{2p}\rangle_L^\frac{1}{p}\big)^\frac{1}{2}
\lesssim\big(\int_{[0,L)^d}|x|^2_L\langle|\nabla v|^{2p}\rangle_L^\frac{1}{p}\big)^\frac{1}{2}.
\end{equation}
As a consequence of annealed weighted estimates on $\nabla(-\nabla\cdot a^*\nabla)^{-1}\nabla\cdot$ (namely, Lemma \ref{L:4} with weight $w=\vert \cdot\vert^2_L$) applied to \eqref{ao78}, we have
\begin{equation}
\label{Num::2}
\big(\int_{[0,L)^d}|x|^2_L\langle|\nabla v|^{2p}\rangle_L^\frac{1}{p}\big)^\frac{1}{2}
\lesssim\big(\int_{[0,L)^d}|x|^2_L|h|^{2}\big)^\frac{1}{2},
\end{equation}
and therefore, by \eqref{Num::3}, there holds
\begin{equation}\label{Num::6}
\big(\int_{[0,L)^d}\langle|v|^{2p}\rangle_L^\frac{1}{p}\big)^\frac{1}{2}
\lesssim\big(\int_{[0,L)^d}|x|^2_L|h|^{2}\big)^\frac{1}{2}.
\end{equation}
Moreover, by the annealed weighted estimates on $\nabla^2(-\triangle)^{-1}$  \cite[Theorem 7.1]{EmielLorist2020}, we obtain 
%
\begin{align*}
\big(\int_{[0,L)^d}|x|^2_L\langle|\nabla h_j|^{2p}\rangle_L^\frac{1}{p}\big)^\frac{1}{2}
\lesssim\big(\int_{[0,L)^d}|x|^2_L\langle|\nabla v|^{2p}\rangle_L^\frac{1}{p}\big)^\frac{1}{2}.
\end{align*}
Combining it with the Hardy inequality \eqref{Num::3} for $h_j$ and with \eqref{Num::2} yields
\begin{align*}
\big(\int_{[0,L)^d}\langle|h_j|^{2p}\rangle_L^\frac{1}{p}\big)^\frac{1}{2}
\lesssim\big(\int_{[0,L)^d}|x|^2_L|h|^{2}\big)^\frac{1}{2}
\end{align*}
Inserting this and \eqref{Num::6} into \eqref{Num::5}, and employing once more \eqref{Num::6} in the triangle inequality gives \eqref{ao81}.

\medskip

{\sc Step 3. Argument for \eqref{Num::3}.}
W.~l.~o.~g.\ we may assume that $L=1$, and we consider random periodic functions of vanishing average $v$.
The annealed Hardy inequality \eqref{Num::3} relies on three ingredients.
First, if $\langle|u|^{2p}\rangle^{\frac{1}{p}}$ is compactly supported, we have
\begin{align}\label{Num::15}
	\int \langle|u|^{2p}\rangle^{\frac{1}{p}} \lesssim \int  |x|^2 \langle|\nabla u|^{2p}\rangle^{\frac{1}{p}}.
\end{align}
Next, for $\Omega := [-\frac{1}{2},\frac{1}{2})^d \backslash [-\frac{1}{4},\frac{1}{4})^d$, the following annealed Poincaré estimate holds:
\begin{align}\label{Num::16}
	\Big(\int_{\Omega} \langle|v|^{2p}\rangle^{\frac{1}{p}}\Big)^{\frac{1}{2}} 
	\lesssim 
	\Big(\int_{\Omega} \langle|\nabla v|^{2p}\rangle^{\frac{1}{p}}\Big)^{\frac{1}{2}} 
	+
	\int_{\Omega} \langle|v|^{2p}\rangle^{\frac{1}{2p}}.
\end{align}
Last, we make use of an annealed Poincaré-Wirtinger estimate
\begin{align}
	\label{Num::17}
	\int_{[-\small\frac{1}{2},\small\frac{1}{2})^d} \langle|v|^{2p}\rangle^{\frac{1}{2p}}
	\lesssim
	\int_{[-\small\frac{1}{2},\small\frac{1}{2})^d} \langle|\nabla v|^{2p}\rangle^{\frac{1}{2p}}.
\end{align}
Using \eqref{Num::15} for $u:=\eta v$ where $\eta$ is a cut-off function of $[-\small\frac{1}{2},\small\frac{1}{2})^d$ into $[-\small\frac{3}{4},\small\frac{3}{4})^d$, we have by periodicity of $v$ 
\begin{align*}
	\int_{[-\small\frac{1}{2},\small\frac{1}{2})^d}\langle|v|^{2p}\rangle^{\frac{1}{p}}
	\lesssim \int _{[-\small\frac{1}{2},\small\frac{1}{2})^d} |x|^2_1 \langle|\nabla v|^{2p}\rangle^{\frac{1}{p}}
	+
	\int_{\Omega} \langle|v|^{2p}\rangle^{\frac{1}{p}},
\end{align*}
Inserting \eqref{Num::16} and then \eqref{Num::17} into the above estimate yields
\begin{equation}\label{Num::18}
\begin{aligned}
\int_{[-\small\frac{1}{2},\small\frac{1}{2})^d}\langle|v|^{2p}\rangle^{\frac{1}{p}}
&\lesssim 
\int_{[-\small\frac{1}{2},\small\frac{1}{2})^d} |x|^2_1 \langle|\nabla v|^{2p}\rangle^{\frac{1}{p}}
+
\Big(\int_{[-\small\frac{1}{2},\small\frac{1}{2})^d} \langle|\nabla v|^{2p}\rangle^{\frac{1}{2p}}\Big)^{2}.
\end{aligned}
\end{equation}
Since $d>2$, we may employ the Cauchy-Schwarz inequality to get
\begin{align*}
	\Big(\int_{[-\small\frac{1}{2},\small\frac{1}{2})^d} \langle|\nabla v|^{2p}\rangle^{\frac{1}{2p}}\Big)^{2}
	&\leq
	\int_{[-\small\frac{1}{2},\small\frac{1}{2})^d} |x|^2_1 \langle|\nabla v|^{2p}\rangle^{\frac{1}{p}}
	\int_{[-\small\frac{1}{2},\small\frac{1}{2})^d} |x|^{-2}_1
	\\
	&\lesssim
	\int_{[-\small\frac{1}{2},\small\frac{1}{2})^d} |x|^2_1 \langle|\nabla v|^{2p}\rangle^{\frac{1}{p}}.
\end{align*}
Inserting this into \eqref{Num::18} yields the desired \eqref{Num::3} (noting that by periodicity, we can replace $[-\tfrac{1}{2},\tfrac{1}{2})$ by $[0,1)^d$).

We now establish successively \eqref{Num::15}, \eqref{Num::16}, and \eqref{Num::17}.
First, \eqref{Num::15} comes by applying the following Hardy inequality for compactly supported functions $\phi$ (see \cite[Theorem 1.2.8]{Book_Hardy} with $p=2$ and $V=|x|^2$):
\begin{align*}
\int_{\R^d} |\phi|^2 \lesssim \int_{\R^d} |x|^2 |\nabla \phi|^2,
\end{align*}
to $\phi \rightsquigarrow \langle|u|^{2p}\rangle^{\frac{1}{2p}}$, and noticing that by the Hölder inequality with exponents $(\frac{2p}{2p-1},2p)$
\begin{align*}
	\vert\nabla \langle|u|^{2p}\rangle^{\frac{1}{2p}}\vert
	= \vert\langle|u|^{2p}\rangle^{\frac{1}{2p}-1}  \langle|u|^{2p-1} \nabla |u|\rangle\vert
	\leq \langle|\nabla u|^{2p}\rangle^{\frac{1}{2p}}.
\end{align*}
Similarly, we get \eqref{Num::16} from the usual Poincaré inequality applied to the function $\langle|v|^{2p}\rangle^{\frac{1}{2p}}$.
Last, we get \eqref{Num::17} by recalling that $v$ is periodic of vanishing average in $[-\small\frac{1}{2},\small\frac{1}{2})^d$, so that
\begin{align*}
&\int_{[-\small\frac{1}{2},\small\frac{1}{2})^d} \langle|v|^{2p}\rangle^{\frac{1}{2p}}
=
\int_{[-\small\frac{1}{2},\small\frac{1}{2})^d} \big\langle\big|v -\int_{[-\small\frac{1}{2},\small\frac{1}{2})^d} v \big|^{2p}\big\rangle^{\frac{1}{2p}}
\\
&\qquad \leq
\int_{[-\small\frac{1}{2},\small\frac{1}{2})^d} dx \int_{[-\small\frac{1}{2},\small\frac{1}{2})^d} dy \langle|v(x) - v(x+y) |^{2p}\rangle^{\frac{1}{2p}}
\\
&\qquad
\leq
\int_{[-\small\frac{1}{2},\small\frac{1}{2})^d} dx \int_{[-\small\frac{1}{2},\small\frac{1}{2})^d} dy \big\langle\big(|y|\int_0^1 d s |\nabla v(x+sy)|\big)^{2p}\big\rangle^{\frac{1}{2p}}
\\
&\qquad \lesssim \int_{[-\small\frac{1}{2},\small\frac{1}{2})^d}dx \int_{[-\small\frac{1}{2},\small\frac{1}{2})^d} dz \langle| \nabla v(z)|^{2p}\rangle^{\frac{1}{2p}} 
=\int_{[-\small\frac{1}{2},\small\frac{1}{2})^d} \langle| \nabla v|^{2p}\rangle^{\frac{1}{2p}}.
\end{align*}

\subsection{Proof of Corollary \ref{Cor:1}: Limit $L\uparrow\infty$}

In view of (\ref{pde:9.1_quad}) and (\ref{ao08}), 
we may consider the field $\phi^{(1)}$ as a function of $g$, provided the latter is periodic.
The same applies to $\phi^{(2)}$, cf.~(\ref{ao06}), provided we make it
unique through
\begin{equation}\label{AnchorQualitativeNew}
\phi^{(2)}_{ij}(0)=0.
\end{equation}
We will monitor the joint distribution of the triplet of fields
$(g,\phi^{(1)}_i(g),\phi^{(2)}_{ij}(g))$ under $\langle\cdot\rangle_L$. 
This amounts to the push-forward $\langle\cdot\rangle_{L,ext}$ of $\langle\cdot\rangle_L$ 
under the map $g\mapsto(g,\phi^{(1)}_i(g),\phi^{(2)}_{ij}(g))$,
which we denote by $(\text{Id},\phi^{(1)}_i,\phi^{(2)}_{ij})$:
\begin{equation}\label{QualitativeNew:Eq8}
\langle\cdot\rangle_{L,ext}:=(\text{Id},\phi^{(1)}_i,\phi^{(2)}_{ij})\#\langle\cdot\rangle_L;
\end{equation}
the index ``ext'' hints to the fact that $\langle\cdot\rangle_{L,ext}$ is an extension of
$\langle\cdot\rangle_L$ in the sense that the latter is the marginal of the former
w.~r.~t.~the first component. 
As will become apparent in Step 1, $\langle\cdot\rangle_{L,ext}$
is a probability measure on the product space 
$\mathcal{C}^{0,\alpha,\beta}\times\mathcal{C}^{1,\alpha,\beta}\times 
\mathcal{C}^{1,\alpha,\beta}$, provided $\alpha\in(0,1)$ and $\beta\in(\frac{1}{2},\infty)$,
where $\mathcal{C}^{n,\alpha,\beta}$ denotes the space of functions that are locally
in $\mathcal{C}^{n,\alpha}$ but are allowed to grow at rate $\beta$:
\begin{align*}
\mathcal{C}^{n,\alpha,\beta}:=\{\,\phi\colon\mathbb{R}^d\rightarrow\mathbb{R}
\,\vert\,\|\phi\|_{n,\alpha,\beta}:=\sup_{x\in\mathbb{R}^d}(1+\vert x\vert)^{-\beta}
\|\phi\|_{C^{n,\alpha}(B_1(x))}<\infty\}.
\end{align*}
%
%
%
On the one hand, this norm is weak enough so that
Proposition \ref{P:3} implies that the family
$\{\langle\cdot\rangle_{L,ext}\}_{L\uparrow\infty}$ is tight,
see Step 1.
On the other hand, it is (obviously) strong enough so that $q$, $Q^{(1)}_{ij}(z)$ 
and $Q^{(2)}_{ijm}(z)$ are continuous functions of $g$, cf.~\eqref{Defq:Intro}, (\ref{ao45bis}) and (\ref{ao46}), which will imply part $iii)$ of the corollary. 
In Step 2, we show that any (weak) limit\footnote{We use here the notion of tight 
convergence of measures, namely for all $p<\infty$ and for all continuous function 
$F$ on $\mathcal{C}^{0}_{\alpha,\beta}\times \mathcal{C}^1_{\alpha,\beta}\times 
\mathcal{C}^{1}_{\alpha,\beta}$ such that 
$\vert F(g,\phi,\psi)\vert\lesssim 1+\|g\|^p_{0,\alpha,\beta}+\|\phi,\psi\|^p_{1,\alpha,\beta}$ there holds $\langle F\rangle_{L,ext}\underset{L\uparrow +\infty}{\rightarrow}\langle F\rangle_{ext}$.} 
$\langle\cdot\rangle_{ext}$ is stationary, satisfies the moment
bounds of Corollary \ref{Cor:1}, and is supported on fields that satisfy the relations
of Corollary \ref{Cor:1}.  
In Step 3, we identify the first marginal of $\langle\cdot\rangle_{ext}$ with our
whole-space ensemble $\langle\cdot\rangle$. In Step 4, we construct $\phi^{(1)}_i$ and
$\phi^{(2)}_{ij}$, now only defined almost surely, 
satisfying the requirements of Corollary \ref{Cor:1}.
Finally, in Step 5, we argue that $\langle\cdot\rangle_{ext}$ is the push-forward
of the whole-space ensemble $\langle\cdot\rangle$ under the
map $(\text{Id},\phi^{(1)}_i,\phi^{(2)}_{ij})$.
In the following, we drop the indices $i$ and $j$.

\medskip

{\sc Step 1. Compactness result.}
We show that Proposition \ref{P:3} implies for $\alpha\in (0,1)$, $\beta\in(\frac{1}{2},\infty)$, and $p<\infty$
\begin{equation}\label{QualitativeLNew:Eq1}
\sup_{L\geq 1}\langle (\|\text{Id}\|_{0,\alpha,\beta}+\|(\phi^{(1)},\phi^{(2)})\|_{1,\alpha,\beta})^p\rangle_L<\infty.
\end{equation}
We combine this with the compact embedding 
$\mathcal{C}^{n,\alpha',\beta'}\subset\mathcal{C}^{n,\alpha,\beta}$ for $\alpha<\alpha'$ 
and $\beta'<\beta$, which is a consequence of Arzel\`a-Ascoli's theorem. This implies by
Prohorov's theorem \cite[Theorem 3.8.4]{bogachev1998gaussian} that there exists a 
probability measure $\langle\cdot \rangle_{ext}$ on 
$\mathcal{C}^{0,\alpha',\beta'}\times\mathcal{C}^{1,\alpha',\beta'}
\times\mathcal{C}^{1,\alpha',\beta'}$ such that, up to a subsequence that we do not relabel,
\begin{equation}\label{QualitativeNew:Eq6}
\langle\cdot\rangle_{L,ext}\underset{L\uparrow \infty}{\rightharpoonup}\langle\cdot\rangle_{ext}.
\end{equation}
We argue for \eqref{QualitativeLNew:Eq1}.
Since the balls $\{B_1(x)\}_{x\in\frac{1}{\sqrt{d}}\mathbb{Z}^d}$ cover
$\mathbb{R}^d$, we have by a union bound argument
%
\begin{align}\label{fw26}
\big\langle\big(\|g\|_{0,\alpha,\beta}+\|(\phi^{(1)},\phi^{(2)})\|_{1,\alpha,\beta}
\big)^p\big\rangle_L 
&\lesssim \big\langle\big(\sup_{x\in\frac{1}{\sqrt{d}}\mathbb{Z}^d}(1+\vert x\vert)^{-\beta}
(\|g\|_{C^{0,\alpha}(B_1(x))}
+\|(\phi^{(1)},\phi^{(2)})\|_{C^{1,\alpha}(B_1(x))})\big)^p\big\rangle_L\nonumber\\
&\le\sum_{x\in\frac{1}{\sqrt{d}}\mathbb{Z}^d} (1+\vert x\vert)^{-p\beta} 
\langle(\|g\|_{C^{0,\alpha}(B_1(x))}
+\|(\phi^{(1)},\phi^{(2)})\|_{C^{1,\alpha}(B_1(x))})^p\rangle_L.
\end{align}
By local Schauder theory we obtain from (\ref{pde:9.1_quad}) and (\ref{ao06}) 
\begin{align*}
\|\phi^{(1)}\|_{C^{1,\alpha}(B_1(x))}&\lesssim C(\|a\|_{C^{0,\alpha}(B_2(x))})
\Big(1+\Big(\int_{B_2(x)}|\phi^{(1)}|^2\Big)^{\frac{1}{2}}\Big),\nonumber\\
\|\phi^{(2)}\|_{C^{1,\alpha}(B_1(x))}&\lesssim C(\|a\|_{C^{0,\alpha}(B_2(x))})
\Big(\|\phi^{(1)}\|_{C^{1,\alpha}(B_2(x))}+\Big(\int_{B_2(x)}|\phi^{(2)}|^2\Big)^{\frac{1}{2}}\Big)
\end{align*}
with an at most polynomial dependence of the constant on $\|a\|_{C^{0,\alpha}(B_2(x))}$.
By (\ref{fw33})
this yields
\begin{align}\label{fw25}
\big\langle\big(\|g\|_{C^{0,\alpha}(B_1(x))}
+\|(\phi^{(1)},\phi^{(2)})\|_{C^{1,\alpha}(B_1(x))}\big)^p\big\rangle_L^\frac{1}{p}
\lesssim_{p,p'}1+
\int_{B_2(x)}\langle|(\phi^{(1)},\phi^{(2)})|^{p'}\rangle_L^\frac{1}{p'},
\end{align}
provided $p<p'$. By Proposition \ref{P:3} $i)$ for $\phi^{(1)}$ and by
$iv)$ for $\phi^{(2)}$ together with (\ref{AnchorQualitativeNew}), we have
that the r.~h.~s.~of (\ref{fw25}) is estimated by $(1+|x|)^\frac{1}{2}$.
Hence the summand on the r.~h.~s.~of (\ref{fw26}) is estimated by 
$(1+|x|)^{-p(\beta-\frac{1}{2})}$, which is summable provided
$p>\frac{d}{\beta-\frac{1}{2}}$. 
The remaining range is then obtained by Jensen's inequality.

\medskip

{\sc Step 2. Stationarity, moment bounds, and PDE-constrained support of 
$\langle\cdot\rangle_{ext}$.}
Here and in the sequel, we denote by $(g,\phi,\psi)\in{\mathcal C}_{\alpha,\beta}^0\times
{\mathcal C}_{\alpha,\beta}^1\times{\mathcal C}_{\alpha,\beta}^1$ the integration variables
of $\langle\cdot\rangle_{L,ext}$ and its limit $\langle\cdot\rangle_{ext}$.
First, because ${\mathcal C}_{\alpha,\beta}^1\ni\psi\mapsto\psi^2(0)$
is continuous, \eqref{AnchorQualitativeNew} is preserved under \eqref{QualitativeNew:Eq6}:
\begin{equation}\label{QualitativeNew:Eq16}
\psi(0)=0\quad\text{for $\langle\cdot\rangle_{ext}$-a.~e. $\psi$}.
\end{equation}
Next, note that $\phi^{(1)}$ and $\nabla\phi^{(2)}$ are shift-covariant\footnote{We recall that
a random field $\phi=\phi(g,x)$ is called shift-covariant iff $\phi(g,z+x)=\phi(g(z+\cdot),x)$.} 
(the anchoring \eqref{AnchorQualitativeNew} of $\phi^{(2)}$ does not admit shift-covariance,
but does not affect the shift-covariance of $\nabla\phi^{(2)}$) and $\langle\cdot \rangle_L$ 
is stationary\footnote{We recall that a measure on a function space is called stationary
iff it is invariant under shift of the functions.}. Hence, we may introduce the push-forward
of $\langle\cdot\rangle_{L,ext}$ under the map $(g,\phi,\psi)\mapsto (g,\phi,\nabla\psi)$,
which we call $(\text{Id},\text{Id},\nabla)$ and which is stationary,
a linear constraint that is preserved under the weak convergence \eqref{QualitativeNew:Eq6}:
\begin{equation}\label{QualitativeNew:Eq10}
(\text{Id},\text{Id},\nabla)\#\langle\cdot\rangle_{ext}  \text{ is stationary.}
\end{equation}
We now turn to the estimates of Proposition \ref{P:3}. Clearly, the bounds (\ref{ao43}), the second bound in \eqref{ao43Bis} for any compactly supported function $\eta$ and (\ref{ao69})
are preserved under \eqref{QualitativeNew:Eq6}:
\begin{equation}\label{QualitativeNew:Eq11}
\langle\vert\nabla\phi\vert^p+\vert\phi\vert^p+\vert\nabla\psi\vert^p\rangle_{ext}\lesssim_p 1,
\quad\big\langle\big|\int\eta\phi\big|^p\big\rangle_{ext}^\frac{1}{p}\lesssim_p
\big(\int|\eta|^\frac{2d}{d+2}\big)^\frac{d+2}{2d},\quad
\mbox{and}\quad\langle\vert\psi(z)\vert^p\rangle_{ext}^\frac{1}{p}
\lesssim_p \mu^{(2)}_d(\vert z\vert).
\end{equation}

\medskip

Finally,  from the weak convergence \eqref{QualitativeNew:Eq6}, the definition of $\overline{a}$ in \eqref{Def_Abar_intro} together with the stationarity of $\langle\cdot\rangle_L$ in form of $\langle\overline{a}\rangle_L=\langle a(\nabla\phi^{(1)}+e_i)(0)\rangle_L$ and the decay of averages \eqref{ao43Bis} for the flux (applied with $\eta=L^{-d}\mathds{1}_{[0,L)^d}$), one has 
$$\langle \overline{a}\rangle_L\underset{L\uparrow \infty}{\rightarrow}\langle a(\nabla\phi+e_i)(0)\rangle_{ext}\quad\mbox{and}\quad\langle\vert \overline{a}-\langle\overline{a}\rangle_L\vert\rangle_L\lesssim L^{-\frac{d}{2}},$$
so that $\langle\vert \overline{a}-\langle a(\nabla\phi+e_i)\rangle_{ext}\vert\rangle_L\underset{L\uparrow\infty}{\rightarrow} 0$. Therefore, introducing the notation $a_{\text{hom},ext}:=\langle a(\nabla\phi+e_i)(0)\rangle_{ext}$, we obtain under the weak convergence
(\ref{QualitativeNew:Eq6}) that the equations (\ref{pde:9.1_quad}) and (\ref{ao06}), 
when tested against smooth compactly supported functions, are preserved
in the sense that
\begin{equation}\label{QualitativeNew:Eq14}
\nabla\cdot a(\nabla\phi+e_i)=0\quad\text{and}\quad
-\nabla\cdot a(\nabla\psi+\phi e_j)=e_j\cdot (a(\nabla\phi+e_i)-a_{\text{hom},ext}e_i)
\quad\mbox{for}\;\langle\cdot\rangle_{ext}\mbox{-a.~e.}\;(g,\phi,\psi).
\end{equation}
%
%
%
%

\medskip

{\sc Step 3. Identification of the first marginal of $\langle\cdot\rangle_{ext}$.}
More precisely, we show that is given by $\langle\cdot\rangle$. 
By the definition \eqref{QualitativeNew:Eq8}, 
the first marginal of $\langle\cdot\rangle_{L,ext}$ is the Gaussian measure 
$\langle\cdot\rangle_L$. By \eqref{QualitativeNew:Eq6}, the sequence 
$\{\langle\cdot\rangle_L\}_{L\uparrow \infty}$ of Gaussian measures is tight on 
$\mathcal{C}^{0,\alpha,\beta}$. By \cite[Corollary 3.8.5]{bogachev1998gaussian},
it is thus enough to prove the weak convergence of 
$\langle\cdot\rangle_L$ to $\langle\cdot\rangle$ on squares of bounded linear forms:
\begin{equation}\label{QualitativeNew:Eq18}
\lim_{L\uparrow\infty}\langle\ell^2\rangle_L=\langle\ell^2\rangle
\quad\text{for all $\ell\in (\mathcal{C}^{0,\alpha,\beta})^{*}$.}
\end{equation}
By density and tightness, it is enough to check \eqref{QualitativeNew:Eq18} for linear forms
$\ell$ of the form $g\mapsto\int \zeta g$ for an arbitrary Schwartz function $\zeta$. 
The definition (\ref{Def_cL}) of $c_L$ can also be stated in terms of 
the (distributional) Fourier transform of the (periodic) $c_L$ as
\begin{align*}
{\mathcal F}c_L=\big(\frac{2\pi}{L}\big)^d
\sum_{q\in\frac{2\pi}{L}\mathbb{Z}^d}{\mathcal F}c(q)\delta(\cdot-q).
\end{align*}
Hence the l.~h.~s.~of (\ref{QualitativeNew:Eq18}) assumes the form of a Riemann sum:
\begin{align*}
\langle\ell^2\rangle_L=\big(\frac{2\pi}{L}\big)^{d}
\sum_{q\in \frac{2\pi}{L}\mathbb{Z}^d}\mathcal{F}c(q)\vert\mathcal{F}\zeta(q)\vert^2.
\end{align*}
Since by Assumption \ref{Ass:1} in form of the integrability of $c$, see (\ref{fw32}), 
$\mathcal{F}c$ is continuous, we obtain (\ref{QualitativeNew:Eq18}): 
\begin{align*}
\lim_{L\uparrow\infty}\langle\ell^2\rangle_L
=\int \mathcal{F}c\vert\mathcal{F}\zeta\vert^2=\langle\ell^2\rangle.
\end{align*}

\medskip

{\sc Step 4. Construction of $\phi^{(1)}$ and $\phi^{(2)}$. }
We show that there exist 
$\phi^{(1)}$ and $\phi^{(2)}$ satisfying Corollary \ref{Cor:1} $i)$ and $ii)$ respectively. 
%
%
We construct these random variables via disintegration 
of the measure $\langle\cdot\rangle_{ext}$ with respect to its first marginal 
$\langle\cdot\rangle$, which amounts to conditional expectation w.~r.~t.~$g$.
By \cite[Theorem 6.4]{kolmogorov2018foundations} there exists a family of 
measures $\{\langle\cdot|g\rangle_{ext}\}_{g\in \mathcal{C}^{0}_{\alpha,\beta}}$ 
on $\mathcal{C}^1_{\alpha,\beta}\times\mathcal{C}^1_{\alpha,\beta}$ such that for all
$\langle\cdot\rangle_{ext}$-integrable functions $F$,
\begin{align*}
\langle F\rangle_{ext}=\langle\langle F|g\rangle_{ext}\rangle.
\end{align*}
We now define the $\langle\cdot\rangle$-integrable functions $\phi^{(1)}$ and $\phi^{(2)}$ through
conditional expectation
\begin{equation}\label{QualitativeNew:Eq12}
\phi^{(1)}(g):=\langle\phi|g\rangle_{ext}\quad\mbox{and}\quad
%
%
\phi^{(2)}(g):=\langle\psi|g\rangle_{ext}\quad\mbox{for all}\;g\in \mathcal{C}^0_{\alpha,\beta}
\end{equation}
and verify that they satisfy all the requirements of Corollary \ref{Cor:1} $i)$ and $ii)$. 
%
%
Since the conditioning w.~r.~t.~$g$ commutes with the multiplication by $a=A(g)$,
the linear equations \eqref{QualitativeNew:Eq14} are preserved and give rise to
\eqref{EquationCor1Corrector} and \eqref{ao54}.
It is easy to check that the stationarity \eqref{QualitativeNew:Eq10}
translates into shift-covariance of $\phi^{(1)}$ and $\nabla\phi^{(2)}$.
It follows from \eqref{QualitativeNew:Eq16}, via Jensen's inequality applied 
to the conditional expectation, that $\phi^{(1)}$ and $\phi^{(2)}$ satisfy 
the moment bounds of Corollary \ref{Cor:1} 
$i)$ and $ii)$, and also the decay bound (\ref{fw30}) and the growth bound (\ref{fw35}).

\medskip

{\sc Step 5. Identification of $\langle\cdot\rangle_{ext}$. }
We now establish
\begin{align*}
\langle\cdot\rangle_{ext}=(\text{Id},\phi^{(1)},\phi^{(2)})\#\langle\cdot\rangle.
\end{align*}
by showing 
\begin{equation}\label{QualitativeNew:Eq15}
u^{(1)}:=\phi-\phi^{(1)}(g)=0\quad\mbox{and}\quad
u^{(2)}:=\psi-\phi^{(2)}(g)=0\quad\mbox{for}\;\langle\cdot\rangle_{ext}-
\mbox{a.~e.}\;(g, \phi,\psi).
\end{equation}
%
By \eqref{EquationCor1Corrector} and \eqref{QualitativeNew:Eq14}, 
we have $-\nabla\cdot a\nabla u^{(1)}=0$
$\langle\cdot\rangle_{ext}$-a.~s..
By Caccioppoli's inequality we thus obtain
$$
\fint_{B_R}\vert\nabla u^{(1)}\vert^2\lesssim 
\frac{1}{R^2}\fint_{B_{2R}}\vert u^{(1)}\vert^2\quad\mbox{for all}\;R<\infty.
$$
Taking the expectation $\langle\cdot\rangle_{ext}$ and using the shift-covariance of 
$\phi^{(1)}$ and \eqref{QualitativeNew:Eq10} this implies
\begin{align*}
\langle \vert\nabla u^{(1)}\vert^2\rangle_{ext}
\lesssim \frac{1}{R^2}\langle \vert u^{(1)}\vert^2\rangle_{ext}\quad\mbox{for all}\;R<\infty.
\end{align*}
Letting $R\uparrow \infty$ while appealing to (\ref{ao43}) and
\eqref{QualitativeNew:Eq11} yields $\nabla u^{(1)}=0$.
We now use (\ref{fw30}) and the associated property in \eqref{QualitativeNew:Eq11},
for the averaging function $\eta=R^{-d}\mathds{1}_{B_R}$.
Because of $\frac{2d}{d+2}>1$ this yields
$\lim_{R\uparrow\infty}\langle|\fint_{B_R}u^{(1)}|^2\rangle_{ext}=0$.
Hence we obtain the first claim of \eqref{QualitativeNew:Eq15}.  Note that this implies in particular $a_{\text{hom},ext}=a_{\text{hom}}$, namely the first claim of Corollary \ref{Cor:1} $iii)$.

\medskip

It now follows from (\ref{ao54}) and \eqref{QualitativeNew:Eq14} 
that $-\nabla\cdot a\nabla u^{(2)}=0$ $\langle\cdot\rangle_{ext}$-a.~s..
Starting again with Caccioppoli's inequality, followed by the combination
of shift-covariance of $\nabla\phi^{(2)}$ and \eqref{QualitativeNew:Eq10}, 
finally followed by the combination of (\ref{fw35}) and \eqref{QualitativeNew:Eq11}, we obtain
$$\langle\vert\nabla u^{(2)}\vert^2\rangle_{ext}
\lesssim \frac{1}{R^2}\fint_{B_{2R}}dz\langle\vert u^{(2)}(z)\vert^2\rangle_{ext}
\lesssim \big(\frac{\mu^{(2)}_d(R)}{R}\big)^2\quad\mbox{for all}\;R<\infty.$$
Letting $R\uparrow \infty$ we conclude $\nabla u^{(2)}=0$. 
Using once more the combination of (\ref{fw35}) and
\eqref{QualitativeNew:Eq11}, this time for $z=0$, we find $u^{(2)}(0)=0$,
which gives the second claim of \eqref{QualitativeNew:Eq15}. 
The same argument shows that there is at most one pair $(\phi^{(1)},\phi^{(2)})$ 
of random variables satisfying the properties of Corollary \ref{Cor:1} $i)$ and $ii)$. 
Therefore, $\langle\cdot\rangle_{ext}$
is unique and thus the limit of $\langle\cdot\rangle_{L,ext}$ for the entire
sequence $L\uparrow\infty$.


\subsection{Proof of Lemma $\ref{L:1}$: Improved Caccioppoli inequality}

By scaling, it is enough to consider $R=2$; we fix a smooth cut-off function $\eta$ for $B_1$ in $B_2$. Our starting point is the following localized version
of a standard $L^2$-based interpolation estimate, with $n\in\mathbb{N}$ to be fixed later,
which we take from the proof of \cite[Lemma 4]{Bella_Giunti_Otto_2017}:
\begin{align}\label{fs02}
\int(\eta^{4n}\triangle^{2n}w)^2
\lesssim \big(\int|\eta^{4n+1}\nabla\triangle^{2n}w|^2\big)^\frac{4n}{4n+1}
\big(\int w^2\big)^\frac{1}{4n+1}+\int w^2.
\end{align}
We apply it to $w\in H^{2n}_0(B_2)$ (as usual, $H^{2n}_0(B_2)$ denotes the closure of $C^\infty_0(B_2)$
w.~r.~t.~the $H^{2n}(B_2)$-norm) that (weakly) solves
\begin{align}\label{fs01}
\triangle^{2n}w=u\quad\mbox{in}\quad B_2.
\end{align}
We construct $w$ with help of the Riesz representation theorem, so that we automatically have
\begin{align}\label{fs03}
\int u w=\int(\triangle^nw)^2\sim\int|\nabla^{2n}w|^2\gtrsim\int w^2,
\end{align}
where we used higher-order $L^2$-regularity and a higher-order Poincar\'e estimate.
We obtain from inserting (\ref{fs01}) into (\ref{fs02})
\begin{align*}
\int(\eta^{4n}u)^2
\lesssim\big(\int|\eta^{4n+1}\nabla u|^2\big)^\frac{4n}{4n+1}
\big(\int w^2\big)^\frac{1}{4n+1}+\int w^2.
\end{align*}
Combining this with Caccioppoli's estimate
$\int|\eta^{4n+1}\nabla u|^2\lesssim\int(\eta^{4n}u)^2$ and Young's inequality, we obtain
$$\int (\eta^{4n}u)^2\lesssim \int w^2,$$
so that, by the choice of $\eta$, we deduce
\begin{align}\label{fs09}
\int_{B_1}|\nabla u|^2\le\int(\eta^{4n}u)^2\lesssim\int w^2.
\end{align}
It remains to post-process this inner regularity estimate for an $a$-harmonic function $u$.

\medskip

In route to an annealed estimate,
we express the r.~h.~s.~of (\ref{fs09}) in terms of $u$, which is conveniently done
in terms of the complete orthonormal system of eigenfunctions
$\{w_k\}_{k\in\mathbb{N}}\subset H_0^{2n}(B_2)$
and eigenvalues $\{\lambda_k\}_k\subset (0,\infty)$ of the Dirichlet-$\triangle^{2n}$,
which is a positive operator with compact inverse:
\begin{align*}
\int w^2=\sum_{k}\frac{1}{\lambda_k^2}\big(\int u w_k\big)^{2}
=\sum_{k}\frac{1}{\lambda_k}\frac{\big(\int u w_k\big)^{2}}{\int(\triangle^{n}w_k)^2}
\stackrel{(\ref{fs03})}{\lesssim}
\sum_{k}\frac{1}{\lambda_k}\frac{\big(\int u w_k\big)^{2}}{\int|\nabla^{2n}w_k|^2}.
\end{align*}
We insert this into (\ref{fs09})
\begin{align*}
\int_{B_1}|\nabla u|^2
\lesssim\sum_{k}\frac{1}{\lambda_k} \frac{\big(\int uw_k\big)^2}{\int|\nabla^{2n}w_k|^2}
\end{align*}
and apply $\langle(\cdot)^\frac{p}{2}\rangle$. By H\"older's inequality in $k$ we obtain
\begin{align*}
\big\langle\big(\int_{B_1}|\nabla u|^2\big)^\frac{p}{2}\big\rangle
\lesssim\big(\sum_{k'}\frac{1}{\lambda_{k'}}\big)^{\frac{p}{2}-1} \sum_{k}\frac{1}{\lambda_k}
\frac{\big\langle\big|\int uw_k\big|^{p}\big\rangle}{\big(\int|\nabla^{2n}w_k|^2\big)^\frac{p}{2}}.
\end{align*}
In order to proceed, we need $\sum_{k}\frac{1}{\lambda_k}<\infty$, which means that
the inverse of the Dirichlet-$\triangle^{2n}$ has finite trace, which in turn follows
from the finiteness of the corresponding
Green function along the diagonal, which requires that Dirac distributions
are in $H^{-2n}(B_2)$, which amounts to the Sobolev embedding $H^{2n}_0(B_2)\subset C^0_0(B_2)$,
and thus holds provided $2n>d$, which we henceforth assume.
Hence by the density of $C^\infty_0(B_2)$ in $H^{2m}_0(B_2)\ni w_k$ we obtain
the annealed inner regularity estimate
\begin{align}\label{fs06}
\big\langle\big(\int_{B_1}|\nabla u|^2\big)^\frac{p}{2}\big\rangle
\lesssim\sup_{w\in C^\infty_0(B_2)}
\frac{\big\langle\big|\int uw\big|^{p}\big\rangle}{\big(\int|\nabla^{2n}w|^2\big)^\frac{p}{2}}.
\end{align}

\medskip

It remains to post-process (\ref{fs06}).
Provided $2n>\frac{d}{2}+3$ we may appeal to Sobolev's embedding applied to $\nabla^3 w$
in order to upgrade (\ref{fs06}) to
%
%
%
\begin{align}\label{fs07}
\big\langle\big(\int_{B_1}|\nabla u|^2\big)^\frac{p}{2}\big\rangle^\frac{1}{p}
\lesssim\sup_{w\in C^\infty_0(B_2)}
\frac{\big\langle\big|\int uw\big|^{p}\big\rangle^\frac{1}{p}}
{\sup|\nabla^{3}w|}.
\end{align}
Since we may w.~l.~o.~g.~assume $\int_{B_2}u=0$, we may restrict to $w$ with $\int w=0$. A standard argument\footnote{We give the formula for $d=2$: Fix a smooth $\eta_2=\eta(x_2)$ supported in
$(-2,2)$ and of integral 1; then $h(x):=(h_1(x_1)\eta(x_2),h_2(x))$
with $h_1(x_1):=\int_{-\infty}^{x_1}dx_1'\int dx_2'w(x_1',x_2')$ and
$h_2(x_1,x_2):=\int_{-\infty}^{x_2}dx_2'(w(x_1,x_2')-\frac{dh_1}{dx_1}(x_1)\eta(x_2'))$
has the desired support properties. By definition
$\frac{dh_1}{dx_1}(x_1)=\int dx_2'w(x_1,x_2')$,
$\partial_2h_2(x)=w(x)-\frac{dh_1}{dx_1}(x_1)\eta_2(x_2)$,
and thus $\nabla\cdot h=w$.} in the theory of distributions yields the existence of a vector field
$g\in C^\infty_0((-2,2)^d)\subset C^\infty_0(B_{2\sqrt{d}})$ such that
\begin{align*}
\nabla\cdot g=w\quad\mbox{and}\quad\sup|\nabla^3 g|\lesssim\sup|\nabla^3 w|.
\end{align*}
Hence (\ref{fs07}) may be upgraded to the desired
\begin{align}\label{fs08}
\big\langle\big(\int_{B_1}|\nabla u|^2\big)^\frac{p}{2}\big\rangle^\frac{1}{p}
\lesssim\sup_{g\in C^\infty_0(B_{2\sqrt{d}})}
\frac{\big\langle\big|\int g\cdot\nabla u\big|^p\big\rangle^\frac{1}{p}}{\sup|\nabla^3g|}.
\end{align}
\qed

\subsection{Proof of Lemma \ref{L:3}: Annealed estimate on the two-scale expansion error}

In Step 1, we establish a pointwise bound on $\nabla^3\overline{u}$,
which in Step 2 we combine with a dyadic decomposition argument.

\medskip

{\sc Step 1. Pointwise bound on $\nabla^3\bar u$. }We claim that
\begin{equation}\label{Lem3:Eq1}
\langle\vert\nabla^3\bar u (x)\vert^p\rangle_L^{\frac{1}{p}}
\lesssim R\sup \vert\nabla^2 f\vert \big(\frac{R}{R+\vert x-y\vert}\big)^d.
\end{equation}
Indeed, by \eqref{ao70} we have
$\bar u(x)=\int_{B_R(y)}dz\,f(z) \bar G(x-z)$.
From the bounds on the constant-coefficient Green function $\bar G$,
which are uniform 
in the random coefficient $\lambda{\rm id}\le\bar a\le{\rm id}$, we obtain
$$\langle\vert\nabla^3\bar u(x)\vert^p\rangle_L^{\frac{1}{p}}\lesssim
 \left\{
    \begin{array}{lll}\mbox{for }|x-y|\le 2R:&
\int_{B_R(y)}dz\,\vert\nabla^2 h(z)\vert
\langle\vert\nabla\bar G(x-z)\vert^p\rangle^{\frac{1}{p}}&
\lesssim\sup|\nabla^2 f|R\\
\mbox{for }|x-y|\ge 2R:&
\int_{B_R(y)}dz\,\vert\nabla f(z)\vert
\langle\vert\nabla^2\bar G(x-z)\vert^p\rangle^{\frac{1}{p}}&
\lesssim\sup|\nabla f|(\frac{R}{|x-y|})^d
    \end{array}\right\}.
$$
It remains to appeal to $\sup|\nabla f|\le R\sup|\nabla^2 f|$.

\medskip

{\sc Step 2. Dyadic decomposition . }We restrict the r.~h.~s.~of \eqref{ao66} 
to dyadic annuli:
    \begin{equation}\label{Def_g}
    \begin{aligned}
      h_k:= \mathds{1}_{B_{2^k}\setminus B_{2^{k-1}}} \big((\phi_{ij}^{(2)}-\phi_{ij}^{(2)}(0))a
-(\sigma_{ij}^{(2)}-\sigma_{ij}^{(2)}(0))\big) \nabla\partial_{ij}  \bar{u};
    \end{aligned}
    \end{equation}
this induces the decomposition $\nabla w=\sum_{k\in\mathbb{Z}} \nabla w_k$, 
where $\nabla w_k$ is the square-integrable solution of
    \begin{equation*}
      -\nabla\cdot a\nabla w_k = \nabla\cdot h_k.
    \end{equation*}
    We observe from $\eqref{Def_g}$ that $w_k$ is $a$-harmonic in $B_{2^{k-1}}$.

    \medskip

    Due to the above decomposition, the desired estimate
    $\eqref{ao67}$ is reduced to estimating $\nabla w_k(0)$ for each $k$.
    Using first Lemma~\ref{L:2}, 
then the energy estimate, and finally Minkowski's inequality, we obtain provided $p'>p\ge 2$
    \begin{equation}\label{Num:220-8}
 \begin{aligned}
    \langle |\nabla w_k(0)|^{p} \rangle_L^{\frac{1}{p}}
    &\stackrel{\eqref{ao59}}{\lesssim}
    \big\langle \big(\fint_{\Boule_{2^{k-1}}} |\nabla w_k|^2 \big)^{\frac{p'}{2}} \big\rangle_L^{\frac{1}{p'}}
    \lesssim
    \big\langle \big(\fint_{\Boule_{2^{k-1}}} |h_k|^2 \big)^{\frac{p'}{2}} \big\rangle_L^{\frac{1}{p'}}\lesssim
    \big(\fint_{\Boule_{2^{k}}} \big\langle |h_k|^{p'}
    \big\rangle_L^{\frac{2}{p'}}\big)^{\frac{1}{2}}.
    \end{aligned}
    \end{equation}
    In view of the definition \eqref{Def_g} of $h_k$, \eqref{ao69} in Proposition \ref{P:3}, 
and \eqref{Lem3:Eq1}, we have for $p''>p'$
    \begin{equation}\label{Num:220-9}
    \begin{aligned}
    \big(\fint_{\Boule_{2^{k}}} \big\langle |h_k|^{p'}
    \big\rangle_L^{\frac{2}{p'}}\big)^{\frac{1}{2}}
    \overset{\eqref{Def_g},\eqref{ao69}}{\lesssim}&\mu_d^{(2)}(2^k)
    \big(\fint_{B_{2^k}}
    \big\langle|\nabla^3 \bar{u}|^{p''}\big\rangle_L^{\frac{2}{p''}}
    \big)^{\frac{1}{2}}\\
    \lesssim& \mu_d^{(2)}(2^k)R\sup\vert\nabla^2 f\vert
    \big(\fint_{B_{2^k}}dx\,\big(\frac{R}{R+\vert x-y\vert}\big)^{2d}
    \big)^{\frac{1}{2}}.
    \end{aligned}
    \end{equation}
We now distinguish the two cases of $2^k\le R$, where we use
$ \fint_{B_{2^k}}dx\,\big(\frac{R}{R+\vert x-y\vert}\big)^{2d}\leq 1$,
and of $2^k>R$, where we use
%
    \begin{equation*}
     \fint_{B_{2^k}}dx\,\big(\frac{R}{R+\vert x-y\vert}\big)^{2d}
   \lesssim 2^{-kd} \int dx\,\big(\frac{R}{R+\vert x-y\vert}\big)^{2d}\lesssim \big(\frac{R}{2^k}\big)^{d}.
    \end{equation*}

    The combination of $\eqref{Num:220-8}$, $\eqref{Num:220-9}$ and the two last estimates yields
    \begin{equation}\label{f:4.26_bis}
    \big\langle |\nabla w(0)|^{p} \big\rangle_L^{\frac{1}{p}}
    \lesssim\big(\sum_{2^k\le R}\mu_d^{(2)}(2^k)+\sum_{2^k>R}(\frac{R}{2^k})^{\frac{d}{2}}
\mu_d^{(2)}(2^k)\big)
    R\sup|\nabla^2 h|.
    \end{equation}
%
%
Since for any $d>2$, $\mu_d^{(2)}(r)$ is non-decreasing in $r$, linear for $r \le 1$, 
and not increasing faster than $r^\frac{1}{2}$ for $r\ge 1$, we recover \eqref{ao67}.

\subsection{Proof of Proposition \ref{P:4}: Annealed error estimate on the expansion of the Green function}\label{ProofOfProp4Random}

\bigskip

Throughout the proof, we fix two ``base points'' $x_0,y_0\in\mathbb{R}^d$ with $|x_0-y_0|\ge 2$.

\medskip

{\sc Step 1. Passage to the full error in the two-scale expansion.}
Recall that ${\mathcal E}$, cf.~(\ref{ao47}) for its definition, is the {\it truncated}
version, on the level of the mixed derivatives, of the full error of the second-order
two-scale expansion of the constant-coefficient fundamental solution $\bar G$,
which is given by\footnote{note the change of sign that is due to
the fact that $\partial_j$ acts on the argument of $\bar G$ and not on $y$}
    \begin{equation}\label{Def_V}
    \begin{aligned}
      w_{x_0,y_0}(x,y) := G(x,y)-\big(1&+\phi^{(1)}_{i}(x)\partial_i
      +(\phi^{(2)}_{im}-\phi^{(2)}_{im}(x_0))(x)\partial_{im}\big)\\
      &\times\big(1-\phi^{*(1)}_j(y)\partial_j+(\phi_{jn}^{*(2)}
      -\phi_{jn}^{*(2)}(y_0))(y)
      \partial_{jn}\big)\bar{G}(x-y).
    \end{aligned}
    \end{equation}
We now find the mixed derivative:
\begin{equation}\label{G_expan}
\begin{aligned}
&\nabla \nabla w_{x_0,y_0}(x,y)
=\nabla\nabla G(x,y) \\
&\qquad-\big[(e_i+\nabla\phi_i^{(1)})(x)\partial_i
+(\nabla\phi_{im}^{(2)}+\phi_{i}^{(1)}e_m)(x)\partial_{im}
+(\phi_{im}^{(2)}-\phi_{im}^{(2)}(x_0))(x)e_k\partial_{imk}\big]\\
&\qquad\otimes\big[-(e_j+\nabla\phi_{j}^{*(1)})(y)\partial_j
+ (\nabla\phi_{jn}^{*(2)}
+\phi_{j}^{*(1)}e_n)(y)\partial_{jn}
-(\phi_{jn}^{*(2)}-\phi_{jn}^{*(2)}(y_0))(y)e_l\partial_{jnl}
\big]\bar{G}(x-y).
\end{aligned}
\end{equation}
We further consider the difference between the mixed derivative of full error and its truncated version $\mathcal{E}$ by setting $x=x_0$ and $y=y_0$, 
which gives, together with Proposition $\ref{P:3}$,
 \begin{equation}\label{Num:252}
      \big\langl \lt|\nabla\nabla w_{x_0,y_0}(x_0,y_0) - \mathcal{E}(x_0,y_0)\rt|^p \big\rangl_L^{\frac{1}{p}} \lesssim  \mu^{(2)}_d(\vert x_0-y_0\vert)|x_0-y_0|^{-d-2}\text{\quad for any $p<\infty$.}
    \end{equation}
Therefore, to obtain the desire estimate $\eqref{ao73}$
it suffices to show
\begin{equation}\label{Num:2370}
  \big\langle |\nabla \nabla w_{x_0,y_0}(x_0,y_0)|^{p} \big \rangle_L^{\frac{1}{p}}
  \lesssim \max\{\mu^{(2)}_d(|x_0-y_0|),\ln \vert x_0-y_0\vert \} |x_0-y_0|^{-d-2}.
\end{equation}

\medskip

{\sc Step 2. A decomposition of the full error $\nabla\nabla w_{x_0,y_0}(x_0,y_0)$.}
In this step, we shall
derive a characterizing PDE (\ref{Eq_sur_V}) of the full error \eqref{Def_V} in order to split
it into a far-field part $w_{x_0,y_0,\infty}$ and dyadic near-field parts $w_{x_0,y_0,k}(x,\cdot)$, which will be 
explicitly given later on. The distinction between
far and near fields refers to the scale $R:=|x_0-y_0|/2$.
Recall that $w_{x_0,y_0}(x,y)$ involves the two-scale expansion in both the $x$
and $y$ variables; we now freeze $x=x_0$ and consider $y$ as the ``active'' variable. For the ease of the statement, we make use of the notation
\begin{equation}\label{NotationuBarExpG}
\bar u_{x_0}(x,\cdot):=
\big(1+\phi_i^{(1)}(x)\partial_i+(\phi_{im}^{(2)}(x)-\phi_{im}^{(2)}(x_0))\partial_{im}\big)
\bar G(x-\cdot).
\end{equation}
This amounts to rewriting $ w_{x_0,y_0}(x_0,\cdot)$ as follows:
  \begin{equation}\label{ExpandV}
  \begin{aligned}
     w_{x_0,y_0}(x_0,\cdot)=
    & G(x_0,\cdot)-
    \big(1-\phi^{*(1)}_{j}\partial_j+
    (\phi_{jn}^{*(2)}-\phi_{jn}^{*(2)}(y_0))
    \partial_{jn}\big)
     \bar{u}_{x_0}(x_0,\cdot).
  \end{aligned}
  \end{equation}
We note that $G(x_0,\cdot)$ is $a^{*}$-harmonic whereas $\overline{u}_{x_0}(x_0,\cdot)$ is $\overline{a}^{*}$-harmonic in $\mathbb{R}^d\backslash \{x_0\}$.
Hence, the representation of the error in the second order two-scale expansion introduced in Subsection $\ref{SS:twoscale}$, we thus have
\begin{equation}\label{Eq_sur_V}
  -\nabla  \cdot  a^* \nabla  w_{x_0,y_0}(x_0,\cdot)= \nabla \cdot h_{x_0,y_0}(x_0,\cdot)  \quad
  \text{in}\quad \R^d \setminus \{x_0\},
\end{equation}
where the vector field $h_{x_0,y_0}$ is given by
\begin{equation}\label{Def_h}
  h_{x_0,y_0}(x,\cdot) := \big((\phi^{*(2)}_{jn}-\phi^{*(2)}_{jn}(y_0))a^* -(\sigma^{*(2)}_{jn}-\sigma_{jn}^{*(2)}(y_0))\big) \nabla\partial_{jn}\bar{u}_{x_0}(x,\cdot).
\end{equation}
Next we define the dyadic near-field (scalar) functions:
\begin{equation}\label{Num:220}
\begin{aligned}
-\nabla\cdot a^*\nabla w_{x_0,y_0,k}(x,\cdot)
=\nabla\cdot
\mathds{1}_{B_{2^k}(y_0) \backslash B_{2^{k-1}}(y_0)}h_{x_0,y_0}(x,\cdot),
\end{aligned}
\end{equation}
and the far-field function:
\begin{equation}\label{Num:220-1}
\begin{aligned}
w_{x_0,y_0,\infty}:=w_{x_0,y_0}-\sum_{2^k\le R}w_{x_0,y_0,k},
\end{aligned}
\end{equation}
In fact, we
are interested in the quantities
$\nabla_{x} w_{x_0,y_0,k}(x_0,\cdot)$ and
$\nabla_{x}  w_{x_0,y_0,\infty}(x_0,\cdot)$, which we address in two steps.

\medskip

{\sc Step 3. Estimate of the near-field parts
$\nabla\nabla_{x} w_{x_0,y_0,k}(x_0,\cdot)$. }
Note that applying $\nabla_x$ and evaluating at $x=x_0$ commutes with the differential operator $\nabla\cdot a^*\nabla
$.
For the ease of notation, we fix an arbitrary coordinate direction $i=1,\cdots,d$ and introduce the abbreviation\footnote{Throughout the proof,
we use $\partial_{x_i}$ (or $\nabla_x$) if
the partial derivative (or the gradient) is
taken w.r.t. the first variable.
But we may write $\nabla$ for $\nabla_y$
since $y$ here is the ``active'' variable.}
$w_{k,i}(y):=\partial_{x_i} w_{x_0,y_0,k}(x_0,y)$. We start by estimating its constitutive element $\bar u_{x_0}$ (see \eqref{NotationuBarExpG}) and we obtain from
Proposition $\ref{P:3}$ and the $-d$-homogeneity of $\bar G$
  \begin{equation}\label{Num:2401}
    \sup_{y\in B_{R}(y_0)}\big\langle\lt|\nabla^j
    \partial_{x_i}\bar{u}_{x_0}(x_0,y)\rt|^{p}\big\rangle_L^{\frac{1}{p}}
    \lesssim R^{-j-d+1} \big(1+ \big \langle \big\vert \big(\nabla\phi^{(1)}(x_0),\frac{\phi^{(1)}(x_0)}{R},
    \frac{\nabla\phi^{(2)}(x_0)}{R}\big) \big\vert^p \big \rangle_L^{\frac{1}{p}}\big)
    \lesssim R^{-j-d+1}
  \end{equation}
for any $j\geq 0$ and $p<\infty$ (we also used $R\geq 1$ in the last estimate).
As in the proof of Lemma \ref{L:3}, $2<p<p'$ denote generic exponents for stochastic integrability. Thus, using \eqref{ao59} in Lemma~\ref{L:2} and the energy estimate, we have
\begin{equation}\label{num:220-7}
  \begin{aligned}
  	\sum_{2^k\le R}
    \big\langle
    \vert \nabla
    w_{k,i}(y_0) \vert^{p}
    \big\rangle_L^{\frac{1}{p}}
    \overset{\eqref{ao59}}{\lesssim} \sum_{2^k\le R}\big\langle \big(\fint_{B_{2^{k-1}}(y_0)}\vert\nabla w_{k,i} \vert^2\big)^{\frac{p'}{2}} \big\rangle_L^{\frac{1}{p'}}
    \lesssim \sum_{2^k\le R}\big\langle \big(\fint_{B_{2^{k}}(y_0)}\vert
    \partial_{x_i} h_{x_0,y_0}(x_0,\cdot)\vert^2\big)^{\frac{p'}{2}} \big\rangle_L^{\frac{1}{p'}}.
  \end{aligned}
\end{equation}
By Minkowski's inequality and Proposition \ref{P:3}, we also have
\begin{equation}\label{num:220-4}
\begin{aligned}
&\sum_{2^k\le R}\big\langle \big(\fint_{B_{2^{k}}(y_0)}\vert \partial_{x_i} h_{x_0,y_0}(x_0,\cdot)\vert^2\big)^{\frac{p'}{2}}
\big\rangle_L^{\frac{1}{p'}}
\leq
\sum_{2^k\le R}
\bigg(\fint_{B_{2^{k}}(y_0)}\big\langle\vert \partial_{x_i} h_{x_0,y_0}(x_0,\cdot)\vert^{p'}\big\rangle_L^{\frac{2}{p'}}\bigg)^{\frac{1}{2}}\\
&\qquad\overset{\eqref{Def_h}}{\leq}
\sum_{2^k\le R}
\bigg( \fint_{B_{2^{k}}(y_0)}\big\langle\big\vert \big(\phi^{*(2)}_{jn}-\phi^{*(2)}_{jn}(y_0),
\sigma^{*(2)}_{jn}-\sigma^{*(2)}_{jn}(y_0)\big)\big\vert^{2p'}
\big\rangle^{\frac{2}{p'}}_L
\bigg)^{\frac{1}{4}}\\
&\qquad\qquad\times\bigg(\fint_{B_{2^{k}}(y_0)}
\big\langle\vert \nabla \partial_{jn}\partial_{x_i}\bar{u}_{x_0}(x_0,\cdot)\vert^{2p'}\big
\rangle_L^{\frac{2}{p'}}
\bigg)^{\frac{1}{4}}\\
&\overset{\eqref{ao69},\eqref{Num:2401}}{\lesssim} \sum_{2^k\le R} \mu_d^{(2)}(2^k)R^{-2-d}
\lesssim \max\{\mu_d^{(2)}(R),\ln R\} R^{-2-d}.
\end{aligned}
\end{equation}
The combination of $\eqref{num:220-7}$ and
$\eqref{num:220-4}$ leads to
\begin{equation}\label{pri:4.23}
\sum_{2^k\le R}
    \big\langle
    \vert\nabla w_{k,i}(y_0) \vert^{p}
    \big\rangle_L^{\frac{1}{p}}
    \lesssim \max\{\mu_d^{(2)}(R),\ln(R)\} R^{-2-d}.
\end{equation}
\medskip

{\sc Step 4. Estimate of the near-field parts $\nabla\nabla_{x} w_{x_0,y_0,k}(x_0,\cdot)$ in a weak norm. }Let $p<p'<p''$ be three stochastic exponents. In the sequel, $h=h(y)$ always denotes an
arbitrary smooth vector field compactly supported in $B_R(y_0)$.
We now justify a weak control on $\nabla\nabla_{x}  w_{x_0,y_0,\infty}(x_0,\cdot)$ that will appear useful in Step $5$ when appealing to Corollary \ref{Cor:2}:
\begin{equation}\label{Num:220-3}
\sum_{2^k\le R}\big\langle\big|
\int h\cdot\nabla w_{k,i}\big|^{p}\big\rangle^\frac{1}{p}_L \lesssim \mu_{d}^{(2)}(R)R\sup|\nabla^3 h|.
\end{equation}
We start with
a strong estimate of $w_{k,i}$ on this set. As opposed to (\ref{num:220-7}),
we use Jensen's inequality
to pass to the spatial $L^2$-norm and then replace the energy estimate by the annealed Calder\'on-Zygmund estimate
\cite[Proposition 7.1]{JosienOtto_2019},
\begin{equation*}
\begin{aligned}
 \sum_{2^k\le R}\big\langle\big(\fint_{B_R(y_0)}|\nabla w_{k,i}|\big)^{p}\big\rangle^{\frac{1}{p}}_L
& \leq \sum_{2^k\le R}\fint_{B_R(y_0)}
 \big\langle|\nabla w_{k,i}|^{p}\big\rangle^{\frac{1}{p}}_L
 \leq \sum_{2^k\le R}\big(\fint_{B_R(y_0)}
 \big\langle|\nabla w_{k,i}|^{p}\big\rangle^{\frac{2}{p}}_L\big)^{\frac{1}{2}}\\
& \lesssim
R^{-\frac{d}{2}}
\sum_{2^k\le R}
2^{\frac{kd}{2}}\big(\fint_{B_{2^{k}(y_0)}}
 \big\langle|\partial_{x_i} h_{x_0,y_0}(x_0,\cdot)|^{p'}\big\rangle^{\frac{2}{p'}}_L
 \big)^{\frac{1}{2}}.
\end{aligned}
\end{equation*}
Then, to bound the r.~h.~s., we appeal to the definition of $h_{x_0,y_0}$ in
$\eqref{Def_h}$, Proposition $\ref{P:3}$
and $\eqref{Num:2401}$ to get the following estimate
similar to  $\eqref{num:220-4}$,
\begin{equation*}
\begin{aligned}
 \sum_{2^k\le R}\big\langle\big(\fint_{B_R(y_0)}|\nabla w_{k,i}|\big)^{p}\big\rangle^{\frac{1}{p}}_L
&\lesssim R^{-\frac{d}{2}}
\sum_{2^k\le R}
\mu^{(2)}_d(2^k)2^{\frac{kd}{2}}
\sup_{y\in B_R(y_0)}\big\langle|\nabla^3 \partial_{x_i}\bar{u}_{x_0}(x_0,y)|^{p''}\big\rangle^{\frac{1}{p''}}_L
\overset{\eqref{Num:2401}}{\lesssim}\mu_d^{(2)}(R) R^{-2-d}.
\end{aligned}
\end{equation*}
This shows \eqref{Num:220-3} in form of
\begin{equation*}
\sum_{2^k\le R}\big\langle\big|
\int h\cdot\nabla w_{k,i}\big|^{p'}\big\rangle^\frac{1}{p'}_L
\leq \sup|h|
\sum_{2^k\le R}\big\langle
\big(\int_{B_R(y_0)}|\nabla w_{k,i}|\big)^{p}\big\rangle^{\frac{1}{p}}_L
\lesssim \mu_{d}^{(2)}(R)R\sup|\nabla^3 h|.
\end{equation*}

\medskip

{\sc Step 5. Estimate of the far-field part $\nabla\nabla_{x}  w_{x_0,y_0,\infty}(x_0,\cdot)$ by a duality argument. }Again,
for the ease of notation we introduce the abbreviation
$w_{\infty,i}(y):=\partial_{x_i}w_{x_0,y_0,\infty}(x_0,y)$
with an arbitrary coordinate direction
$i=1,\cdots,d$,
which is $a^*$-harmonic on $B_{2^{k_0}}(y_0)$.
While in the previous
two steps (mostly) relied on homogenization in the $y$-variable in form of control of
$(\phi^{*(2)},\sigma^{*(2)})$, we now (primarily) need homogenization in the $x$-variable,
in form of Lemma \ref{L:3}, next to control of $\phi^{*(2)}$. We start with
an application of Corollary \ref{Cor:2}:
there holds
\begin{equation}\label{Num:235}
\begin{aligned}
  \big\langle |\nabla  w_{\infty,i}(y_0)|^{p} \big \rangle_L^{\frac{1}{p}}\overset{\eqref{ao58}}{\lesssim}&
 \sup_{h\in C^\infty_0(B_{R}(y_0))}
\frac{\big\langle\big|
\fint_{B_R(y_0)}h\cdot \nabla w_{\infty,i} \big|^{p'}\big\rangle^\frac{1}{p'}_L}{R^{3}\sup|\nabla^3 h|}\\
 \overset{\eqref{Num:220-1}}{\leq} &
\sup_{h\in C^\infty_0(B_{R}(y_0))}
\frac{\big\langle\big|\int
h\cdot \nabla\partial_{x_i} w_{x_0,y_0}(x_0,\cdot)\big|^{p'}\big\rangle^\frac{1}{p'}_L}{R^{d+3}
\sup|\nabla^3 h|}\\
+&
\sup_{h\in C^\infty_0(B_{R}(y_0))}
\frac{\sum_{2^k\le R}\big\langle\big|
\int h\cdot\nabla w_{k,i} \big|^{p'}\big\rangle^\frac{1}{p'}_L}{R^{d+3}\sup|\nabla^3 h|}.
\end{aligned}
\end{equation}
While the second contribution has been estimated in $\eqref{Num:220-3}$, we now need a similar estimate on the first contribution, namely,
    \begin{equation}\label{pri:4.22}
    \begin{aligned}
    \big\langle \big|
    \int h\cdot \nabla \partial_{x_i} w_{x_0,y_0}(x_0,\cdot) \big|^p  \big\rangle_L^{\frac{1}{p}}
    \lesssim_{p}
    \max\{\mu^{(2)}_3(R),\ln R\}R\sup |\nabla^3 h|\quad \text{for any $p<\infty$.}
    \end{aligned}
    \end{equation}
Equipped with \eqref{pri:4.22}, \eqref{Num:2370} follows from the combination of \eqref{Num:235}, \eqref{Num:220-3}, \eqref{pri:4.23} and \eqref{Num:220-1}.

\medskip

Now, we focus on the
argument for $\eqref{pri:4.22}$.
Let $h\in C_0^\infty(B_R(y_0))$ be arbitrary, with $u$ and  $\bar{u}$ satisfying $\eqref{ao70}$. We recall the definition of the error in the two-scale expansion that we express in terms of the Green functions $G,\bar G$ using \eqref{ao70}:
  \begin{equation*}
  \begin{aligned}
   w_{x_0}(x)    &:= u(x) - \big(1+\phi_{i}^{(1)}(x)\partial_i+
   (\phi_{im}^{(2)}-\phi_{im}^{(2)}(x_0))(x)
   \partial_{im}
   \big)
   \bar{u}(x)\\
   &\overset{\eqref{ao70}}{=} \int
   (\nabla\cdot h)(G(x,\cdot) - \overline{u}_{x_0}(x,\cdot)),
  \end{aligned}
  \end{equation*}
  where we recall that $\bar u_{x_0}$ is defined in \eqref{NotationuBarExpG}.
 Then, by taking derivatives on the both sides of
  the above equation with respect to the $x$-variable and by integrating by parts with respect to the $y$-variable lead to
  \begin{equation}\label{Num:220-6}
  \begin{aligned}
  \partial_{x_i} w_{x_0}(x)
   =\int
  (\nabla\cdot h)
  \big(\partial_{x_i} G(x,\cdot) - \partial_{x_i}\bar{u}_{x_0}(x,\cdot)\big)
   = - \int
   h\cdot
   \nabla\big(
   \partial_{x_i} G(x,\cdot) - \partial_{x_i}\bar{u}_{x_0}(x,\cdot)\big).
  \end{aligned}
  \end{equation}
We now express the integral in the l.~h.~s. of \eqref{pri:4.22} with help of \eqref{Num:220-6}. First, by applying $\nabla$ to $\eqref{ExpandV}$, we obtain
\begin{equation}\label{Num:220-5}
\begin{aligned}
\nabla  \partial_{x_i}  w_{x_0,y_0}(x_0,\cdot)
&=\nabla\partial_{x_i}  G(x_0,\cdot)-\nabla \partial_{x_i} \bar{u}_{x_0}(x_0,\cdot) 
+\nabla\big[\big(\phi^{*(1)}_{j}\partial_j
-(\phi_{jn}^{*(2)}-\phi_{jn}^{*(2)}(y_0))\partial_{jn}
\big)\partial_{x_i} \bar{u}_{x_0}(x_0,\cdot)\big].
\end{aligned}
\end{equation}
Second, we split the integral in the l.~h.~s. of \eqref{pri:4.22} as follows:
%
  \begin{equation}\label{Split_v}
  \begin{aligned}
  &\int h\cdot
  \nabla \partial_{x_i}  w_{x_0,y_0}(x_0,\cdot)\\
  &\overset{\eqref{Num:220-5}}{=}\int
  h\cdot\nabla
  \big(\partial_{x_i} G(x_0,\cdot)-
  \partial_{x_i}\bar{u}_{x_0}(x_0,\cdot) \big)
  + \int
  h\cdot\nabla \big[\big(\phi^{*(1)}_{j}\partial_j-(\phi_{jn}^{*(2)}
  -\phi_{jn}^{*(2)}(y_0))
  \partial_{jn}\big)\partial_{x_i}\bar{u}_{x_0}(x_0,\cdot)\big]
  \\
  &\overset{\eqref{Num:220-6}}{=}
  -\partial_{x_i} w_{x_0}(x_0)- \int
   (\nabla\cdot h)
  \phi^{*(1)}_j \partial_j \partial_{x_i}\bar{u}_{x_0}(x_0,\cdot)
  +
  \int
  (\nabla\cdot h)
  (\phi_{jn}^{*(2)} -\phi_{jn}^{*(2)}(y_0))\partial_{jn}
  \partial_{x_i}\bar{u}_{x_0}(x_0,\cdot).
  \end{aligned}
  \end{equation}
  For the first term r.~h.~s. term of $\eqref{Split_v}$,
  it follows from Lemma $\ref{L:3}$ applied with $f=\nabla\cdot h$ that
  \begin{equation}\label{pri:4.24}
  \big\langle |\nabla_x w_{x_0}(x_0) |^p\rangle^{\frac{1}{p}}_L
  \lesssim \max\{\mu_{3}^{(2)}(R),\ln(R)\}R\sup|\nabla^3 h|.
  \end{equation}
  For the second term r.~h.~s.\ term of \eqref{Split_v}, we exploit the structure \eqref{NotationuBarExpG} of $\partial_{x_i}\bar{u}_{x_0}(x_0,\cdot)$, the random part of which is independent of the integration variable $y$.
  Hence, we may use the Cauchy-Schwarz inequality, and then appealing to the definition \eqref{Eq:RewriteDivergence} of $\omega^*_j$ (with $\phi^{(1)}_j$ replaced by $\phi^{*(1)}$) together with \eqref{Eq:RewriteDivergence2} and \eqref{ao43}, and finally recall that $h$ is supported in $\Boule_R$, to the effect of
  \begin{equation}
  \label{Num:99001}
  \begin{aligned}
  	&\big\langle\big|\int
  	(\nabla\cdot h)
  	\phi^{*(1)}_j \partial_j \partial_{x_i}\bar{u}_{x_0}(x_0,\cdot) \big|^p\rangle^{\frac{1}{p}}_L
  	\\
  	&\leq
  	\big\langle\big|\int
  	(\nabla\cdot h)
  	\phi^{*(1)}_j 
  	\partial_{jk} \bar{G}(x_0-\cdot)
  	\big|^{2p} \big\rangle^{\frac{1}{2p}}_L
  	\big\langle |\delta_{ik}+\partial_i \phi^{(1)}_k(x_0)|^{2p} \big \rangle^{\frac{1}{2p}}_L
  	\\
  	&\quad+
  	\big\langle\big|\int
  	(\nabla\cdot h)
  	\phi^{*(1)}_j 
  	\partial_{jkm} \bar{G}(x_0-\cdot)
  	\big|^{2p} \big\rangle^{\frac{1}{2p}}_L
  	\big\langle \big|(\phi_k^{(1)} \delta_{im} + \partial_i \phi^{(2)}_{km})(x_0)  \big|^{2p} \big\rangle^{\frac{1}{2p}}_L
  	\\
  	&\lesssim 	\Big(\big\langle\big|\int
  	\nabla(\nabla\cdot h)\cdot(\nabla\omega^*_j-\nabla\omega^*_j(x_0))
  	\partial_{jk} \bar{G}(x_0-\cdot)
  	\big|^{2p} \big\rangle^{\frac{1}{2p}}_L\\
  	&\quad+\big\langle\big|\int
  	(\nabla\cdot h)(\nabla\omega^*_j-\nabla\omega^*_j(x_0))\cdot
  	\nabla\partial_{jk} \bar{G}(x_0-\cdot)
  	\big|^{2p} \big\rangle^{\frac{1}{2p}}_L\Big)
  	\big\langle |\delta_{ik}+\partial_i \phi^{(1)}_k(x_0)|^{2p} \big \rangle^{\frac{1}{2p}}_L
  	\\
  	&\quad+
  	\big\langle\big|\int
  	(\nabla\cdot h)
  	\phi^{*(1)}_j 
  	\partial_{jkm} \bar{G}(x_0-\cdot)
  	\big|^{2p} \big\rangle^{\frac{1}{2p}}_L
  	\big\langle \big|(\phi_k^{(1)} \delta_{im} + \partial_i \phi^{(2)}_{km})(x_0)  \big|^{2p} \big\rangle^{\frac{1}{2p}}_L\\
  	&\lesssim
  	\sup |\nabla^2 h| \mu^{(2)}_{d}(R)+\sup\vert\nabla h\vert(\mu^{(2)}_d(R)+1)R^{-1} 
	\lesssim \mu_d^{(2)}(R)R\sup|\nabla^3h|.
  \end{aligned}
  \end{equation}
  The third r.~h.~s.\ term in \eqref{Split_v} is easily dealt with by recalling that $h$ is of compact support, using Jensen's inequality and the Cauchy-Schwarz inequality, and then a combination of \eqref{ao69} and \eqref{Num:2401}
  \begin{equation}
  \label{Num:99002}
  \begin{aligned}
  	&\big\langle \big| \int
  	(\nabla\cdot h)
  	(\phi_{jn}^{*(2)} -\phi_{jn}^{*(2)}(y_0))\partial_{jn}
  	\partial_{x_i}\bar{u}_{x_0}(x_0,\cdot)
  	\big|^{p} \big\rangle^{\frac{1}{p}}_L
  	\\
  	&\lesssim
  	\sup |\nabla h|
  	\int_{\Boule_R}\langle |\phi_{jn}^{*(2)} -\phi_{jn}^{*(2)}(y_0)|^{2p} \rangle^{\frac{1}{2p}}
  	\langle | \partial_{jn}
  	\partial_{x_i}\bar{u}_{x_0}(x_0,\cdot)|^{2p}\rangle^{\frac{1}{2p}}_L
  	\\
  	&\lesssim
  	R\sup|\nabla^2h| R^d \mu_d^{(2)}(R)R^{-d-1}
  	\\
  	&\lesssim \mu_d^{(2)}(R)R\sup|\nabla^3 h|.
  \end{aligned}
  \end{equation}
  Inserting the estimates \eqref{pri:4.24}, \eqref{Num:99001}, and \eqref{Num:99002}, into \eqref{Split_v} entails \eqref{pri:4.22}. 
\qed

\appendix

\section{$2^{nd}$-order two-scale expansion of the mixed derivative of the massive Green function}
We present in this section the proof of the periodic second order two-scale expansion on the mixed derivative of the massive Green function \eqref{ao42} that we restate in the following proposition.
\begin{proposition}\label{Prop4Massive:Statement}
Let $d>2$ and $a$ being $L$-periodic and H\"older-continuous, for some $L\geq 1$. We define the homogenization error
\begin{equation}\label{Prop4Massive:Eq2}
\begin{aligned}
\lefteqn{{\mathcal E_T}(x,y)}\\
&:=
\nabla\nabla G_T(x,y)+\partial_{ij}\bar G_T(x-y)
(e_i+\nabla\phi_i^{(1)})(x)\otimes (e_j+\nabla{\phi_j^*}^{(1)})(y)\\
&+\partial_{ijm}\bar G_T(x-y)
(\phi_i^{(1)}e_m+\nabla\phi_{im}^{(2)})(x)\otimes (e_j+\nabla{\phi_j^*}^{(1)})(y)\\
&-\partial_{ijm}\bar G_T(x-y)
(e_i+\nabla\phi_i^{(1)})(x)\otimes ({\phi_j^*}^{(1)}e_m+\nabla\phi_{jm}^{*(2)})(y).
\end{aligned}
\end{equation}
There exists a constant $C$ depending on $d$ and $\lambda$ such that 
\begin{equation}\label{Prop4Massive:Eq1}
\left\langle\vert\mathcal{E}_T(x,y)\vert^p\right\rangle_L^{\frac{1}{p}}\lesssim_{p,L} (\ln \vert x-y\vert)\vert x-y\vert^{-d-2}\exp(-\frac{L^{-1}\vert x-y\vert}{C\sqrt{T}})
\end{equation}
provided $T\geq L^2$, $\frac{L}{2}\leq \vert x-y\vert<\infty$ and for all $p<\infty$.
\end{proposition}
The proof of Proposition \ref{Prop4Massive:Statement} closely follows the strategy of Proposition \ref{P:4} with some changes due to the presence of the massive term.  To begin with, thanks to the moment bounds \eqref{fw33}, \eqref{Prop4Massive:Eq1} will follow from the deterministic estimate
\begin{align}\label{fw56}
L^d|{\mathcal E}_T(x,y)|\le C(L^\alpha[a]_{\alpha})(\ln\frac{|x-y|}{L})L^{d+2}|x-y|^{-d-2}\exp(-\frac{L^{-1}\vert x-y\vert}{C\sqrt{T}}),
\end{align}
which here is written in a scale-invariant way so that w.~l.~o.~g.~we may consider $L=1$.
In order to use \eqref{fw33}, we need that $C$ grows at most polynomially in its argument 
$[a]_\alpha$.
Hence what we need to establish is a (fine) result on homogenization of a 1-periodic
H\"older-continuous coefficient field $a$. The proof of Proposition \ref{Prop4Massive:Statement} relies on the periodic and massive counterpart of the Lipschitz estimate of Lemma \ref{L:2} and the bound of the homogenization error of Lemma \ref{L:3}. We start with the Lipschitz estimate.
\begin{lemma}\label{LipschitzMassive}
Let the function $u$ satisfy $\frac{1}{T}u-\nabla\cdot a\nabla u=0$ in the ball $B_R$ of radius $R$ with $R\leq \sqrt{T}$. Then we have 
\begin{equation}\label{cw'02}
\vert\nabla u(0)\vert\leq C([a]_\alpha) \big(\fint_{B_R}\vert(\nabla u,\tfrac{1}{\sqrt{T}}u)\vert^2\big)^{\frac{1}{2}},
\end{equation}
where the constant $C([a]_\alpha)$ depends polynomially on the Hölder norm $[a]_{\alpha}$ next to $d$ and $\lambda$.
\end{lemma}
\begin{proof}
W.~l.~o.~g. we may assume $R\ll\sqrt{T}$. Second, we decompose $B_R$ into annuli $B_{2^{-k+1}R}\backslash B_{2^{-k}R}$, for $k\in\mathbb{N}$ and we
define $u_k$ as the Lax-Milgram solution of
\begin{align*}
-\nabla\cdot a\nabla u_k=-\frac{1}{T}\mathds{1}_{B_{2^{-k+1}R}\backslash B_{2^{-k}R}}u
\end{align*}
and set $u_0:=u-\sum_{k\ge 1}u_k$. Since $u_k$ is $a$-harmonic in $B_{2^{-k}R}$ we have
by the standard Lipschitz estimate \cite[Theorem 4.1.1]{Shen} for $k\ge 0$
\begin{align}\label{cw05}
\fint_{B_r}|\nabla u_k|^2\lesssim \fint_{B_{2^{-k}R}}|\nabla u_k|^2\quad\mbox{for}\;r\le 2^{-k}R,
\end{align}
where $\lesssim$ means $\leq$ up to a multiplicative constant that only depends polynomially on $[a]_{\alpha}$ next to $d$ and $\lambda$. For $k\ge 1$ we upgrade \eqref{cw05} by the energy estimate (where we appeal to $d>2$ to
bring the non-divergence form r.~h.~s.~into divergence form)
\begin{align*}
\fint_{B_r}|\nabla u_k|^2\stackrel{(\ref{cw05})}{\lesssim}
\frac{1}{(2^{-k}R)^d}\int|\nabla u_k|^2\lesssim\big(\frac{2^{-k}R}{T}\big)^2
\fint_{B_{2^{-k+1}R}}u^2,
\end{align*}
which trivially also holds for $r\ge 2^{-k}R$. Hence we obtain by the triangle inequality
\begin{align}\label{cw01}
\big(\fint_{B_r}|\nabla u|^2\big)^\frac{1}{2}
\lesssim\big(\fint_{B_R}|\nabla u|^2\big)^\frac{1}{2}+\sum_{k\ge 1}
\frac{2^{-k}R}{T}\big(\fint_{B_{2^{-k+1}R}}u^2\big)^\frac{1}{2}
\end{align}
for all $r\le R$. By Poincar\'e's inequality (with mean-value zero) we have
\begin{align*}
\Big|\big(\fint_{B_{2^{-k}R}}u^2\big)^\frac{1}{2}
-\big(\fint_{B_{2^{-k+1}R}}u^2\big)^\frac{1}{2}\Big|
\lesssim 2^{-k}R\big(\fint_{B_{2^{-k+1}R}}|\nabla u|^2\big)^\frac{1}{2}.
\end{align*}
By convergence of the geometric series, this allows us to upgrade (\ref{cw01}) to
\begin{align*}
\sup_{r\le R}\big(\fint_{B_r}|\nabla u|^2\big)^\frac{1}{2}\lesssim\big(\fint_{B_R}|\nabla u|^2\big)^\frac{1}{2}
+\frac{R}{T}\big(\fint_{B_R} u^2\big)^\frac{1}{2}
+\frac{R^2}{T}
\sup_{r\le R}\big(\fint_{B_r}|\nabla u|^2\big)^\frac{1}{2}.
\end{align*}
Since we are in the perturbative regime $R\ll\sqrt{T}$, we may buckle
to obtain for any $r\leq R$
$$\big(\fint_{B_r}\vert\nabla u\vert^2\big)^{\frac{1}{2}}\lesssim \big(\fint_{B_R}\vert \nabla u\vert^2\big)^{\frac{1}{2}}+\frac{1}{\sqrt{T}}\big(\fint_{B_R} u^2\big)^{\frac{1}{2}},$$
which turns into \eqref{cw'02} by letting $r\downarrow 0$.
\end{proof}
Next, we prove an estimate of the second-order stochastic homogenisation error.
\begin{lemma}\label{Lemma3MassiveVersion}
Given a
deterministic and smooth function $f$ supported in $B_R(y)$ with $\vert y\vert=2R$ and some $R < \infty$, let u and $\bar u$ be the decaying solutions of
\begin{equation}\label{Lemma3Massive:Eq1'}
\frac{1}{T}u-\nabla\cdot a\nabla u=f=\frac{1}{T}\bar u-\nabla\cdot \bar a \nabla \bar u.
\end{equation}
Then $w:=u-(1+\phi^{(1)}_i\partial_i+(\phi^{(2)}_{ij}-\phi^{(2)}_{ij}(0))\partial_{ij})\bar{u}$ satisfies for some constant $C$ depending on $d$ and $\lambda$
\begin{equation}\label{Prop4Massive:Eq6}
\vert\nabla w(0)\vert\leq C(a)(\ln R) \exp(-\frac{R}{C\sqrt{T}})R\sup\vert\nabla^2 f\vert,
\end{equation}
where the constant $C(a)$ has the same meaning as in Lemma \ref{LipschitzMassive}.
\end{lemma}

\begin{proof}
We split the proof into three steps. In the following, $\lesssim$ has the same meaning as in the proof of Lemma \ref{LipschitzMassive} and the constant $C$ denotes a general constant depending on $d$, $\lambda$ and $[a]_{\alpha}$ which may change from line to line.

\medskip

We split the arguments between the non-perturbative regime $R\geq 2\sqrt{T}$, that we address in the three first steps, and the perturbative regime $R\leq 2\sqrt{T}$ that we address in the last step. 

\medskip

{\sc Step 1. Pointwise bounds on $\bar u$ and its derivatives. }We claim that 
\begin{equation}\label{Lemma3Massive:Eq2}
\vert (\nabla^3\bar u(x),\tfrac{1}{R}\nabla^2\bar u(x),\tfrac{1}{R^2}\nabla\bar u(x))\vert\lesssim  \exp(-\frac{(\vert x-y\vert-R)_+}{C\sqrt{T}})R\sup\vert\nabla^2 f\vert.
\end{equation}
The estimate \eqref{Lemma3Massive:Eq2} follows from the representation formula $\bar u(x)=\int_{B_R(y)}dz\, f(z)\bar G_T(x-z) $ and the bounds on the constant-coefficient Green function $\bar G_T$
\begin{equation}\label{BoundsBarGT}
\vert\nabla^j\bar G_T(x)\vert\lesssim \vert x\vert^{-j-d+2}\exp(-\frac{\vert x\vert}{C\sqrt{T}})\quad \text{for any $j\geq 0$},
\end{equation}
in form of
\begin{align*}
\vert (\nabla^3\bar u(x),\tfrac{1}{R}\nabla^2\bar u(x),\tfrac{1}{R^2}\nabla\bar u(x))\vert\lesssim & \sup \vert\nabla^2 f\vert\int_{B_R(y)}dz\,(\vert\nabla \bar G_T(x-z)\vert+\frac{1}{R}\vert\bar G_T(x-z)\vert)\\
&+\frac{1}{R^2}\sup \vert \nabla f\vert\int_{B_R(y)}\dd z\,\vert \bar G_T(x-z)\vert,
\end{align*}
which turns into \eqref{Lemma3Massive:Eq2} by appealing to $\sup\vert \nabla f\vert\leq R\sup\vert \nabla^2 f\vert$.

\medskip

{\sc Step 2. The two-scale expansion error. }Following the same computations as for \eqref{ao66}, the error $w$ satisfies
\begin{equation}\label{Lemma3Massive:Eq4}
\frac{1}{T}w-\nabla\cdot a\nabla w=\nabla\cdot h+f',
\end{equation}
where the non-divergence form r.~h.~s term $f'$ comes from the massive term and reads
$$f':=\frac{1}{T}((h^{(1)}_i-h^{(1)}_i(0)) e_j-(\phi^{(2)}_{ij}-\phi^{(2)}_{ij}(0)))\partial_{ij}\bar u,$$
where we introduce the periodic vector field (appealing to \eqref{ao08})
\begin{equation}\label{PeriodicFieldLemma3Massive}
\nabla\cdot h^{(1)}_i=\phi^{(1)}_i.
\end{equation}
The divergence form r.~h.~s term contains an additional term coming from the massive term and reads
$$h:=((\phi^{(2)}_{ij}-\phi^{(2)}_{ij}(0))a-(\sigma_{ij}^{(2)}-\sigma_{ij}^{(2)}(0)))\nabla\partial_{ij}\bar u-\frac{1}{T}(h^{(1)}_i-h^{(1)}_i(0))\partial_i\bar u.$$
We now give an estimate on $h$ and $f'$ that will be useful in the next steps. By Schauder's theory we have
\begin{equation}\label{Prop4Massive:Eq3}
\|(\phi^{(1)},\phi^{(2)},\sigma^{(2)},h^{(1)})\|_{C^{1}([0,1)^d)}\lesssim 1,
\end{equation}
and the estimate \eqref{Lemma3Massive:Eq2} combined with \eqref{Prop4Massive:Eq3} leads to
\begin{equation}\label{Lemma3Massive:Eq3}
\vert(h(x),\sqrt{T}f'(x))\vert\lesssim \frac{R^2}{T}\exp(-\frac{(\vert x-y\vert-R)_+}{C\sqrt{T}}) \min\{\vert x\vert,1\}R\sup\vert\nabla^2 f\vert,
\end{equation}
where we used $R\geq 2\sqrt{T}$.
\medskip

{\sc Step 3. Dyadic decomposition argument. } 
We restrict the r.~h.~s. of \eqref{Lemma3Massive:Eq4} to dyadic annuli:
\begin{equation}\label{Lemma3Massive:Eq5}
h_k:=\mathds{1}_{B_{2^k}\backslash B_{2^{k-1}}}h, \quad f'_k:=\mathds{1}_{B_{2^k}\backslash B_{2^{k-1}}}f'\quad \text{for $2^{k}\leq\sqrt{T}$,}
\end{equation}
and set
\begin{equation}\label{Lemma3Massive:Eq6}
h_\infty:=\mathds{1}_{\mathbb{R}^d\backslash B_{\sqrt{T}}}h,\quad f'_\infty:=\mathds{1}_{\mathbb{R}^d\backslash B_{\sqrt{T}}}f'.
\end{equation}
This induces the decomposition $w=\sum_{2^k\leq \sqrt{T}} w_k+ w_\infty$, where $w_k$ and $w_{\infty}$ are the bounded solutions of
$$\frac{1}{T}w_k-\nabla\cdot a\nabla w_k=\nabla\cdot h_k+f'_k, \quad  \frac{1}{T}w_{\infty}-\nabla\cdot a\nabla w_{\infty}=\nabla\cdot h_{\infty}+f'_{\infty}.$$
We now treat separately the near-field part and the far-field part. For the near-field part, we apply the Lipschitz estimate of Lemma \ref{LipschitzMassive} up to the scale $2^{k-1}$ which we combine with an energy estimate to obtain
\begin{align*}
\vert \nabla w_k(0)\vert \lesssim \big(\fint_{B_{2^{k-1}}}\vert(\nabla w_k,\tfrac{1}{\sqrt{T}} w_k)\vert^2\big)^{\frac{1}{2}}\lesssim  \big(\fint_{B_{2^{k}}}\vert( h,\sqrt{T}f')\vert^2\big)^{\frac{1}{2}}.
\end{align*}
Since we are in the regime $R\geq 2\sqrt{T}$, \eqref{Lemma3Massive:Eq3} implies
$$\big(\fint_{B_{2^{k}}}\vert(h,\sqrt{T}f')\vert^2\big)^{\frac{1}{2}}\lesssim  \exp(-\frac{R}{C\sqrt{T}})\min\{2^{ k},1\}R\sup\vert\nabla^2 f\vert,$$
which provides
$$\vert\sum_{2^k\leq \sqrt{T}}\nabla w_k(0)\vert\lesssim (\ln R)\exp(-\frac{R}{C\sqrt{T}})R\sup\vert\nabla^2 f\vert.$$
For the far-field part, we apply Lemma \ref{LipschitzMassive} up to the scale $\sqrt{T}$ which we combine with the localized energy estimate \cite[(169)]{Gloria_Neukamm_Otto_2019} in form of : there exists a constant $C$ depending on $d$ and $\lambda$ such that for $\eta_{\sqrt{T}}:=\sqrt{T}^{-d}\exp(-\frac{\vert\cdot\vert}{C\sqrt{T}})$,
\begin{align*}
\vert\nabla w_{\infty}(0)\vert\stackrel{\eqref{cw'02}}{\lesssim}\big(\fint_{B_{\sqrt{T}}}\vert(\nabla w_{\infty},\tfrac{1}{\sqrt{T}}w_{\infty})\vert^2\big)^{\frac{1}{2}}\lesssim \big(\int \eta_{\sqrt{T}}\vert(\nabla w_{\infty},\tfrac{1}{\sqrt{T}}w_{\infty})\vert^2\big)^{\frac{1}{2}}\lesssim\big(\int \eta_{\sqrt{T}} \vert ( h,\sqrt{T}f')\vert^2\big)^{\frac{1}{2}}.
\end{align*}
We then split the integral into :
\begin{equation}\label{Lemma3Massive:Eq7}
\big(\int \eta_{\sqrt{T}} \vert (h,\sqrt{T}f')\vert^2\big)^{\frac{1}{2}}\leq \big(\int_{B_{\frac{3}{2}R}(y)} \eta_{\sqrt{T}} \vert (h,\sqrt{T}f')\vert^2\big)^{\frac{1}{2}}+\big(\int_{\mathbb{R}^d\backslash B_{\frac{3}{2}R}(y)} \eta_{\sqrt{T}} \vert (h,\sqrt{T}f')\vert^2\big)^{\frac{1}{2}}.
\end{equation}
Using \eqref{Lemma3Massive:Eq3} and $\eta_{\sqrt{T}}(x)\lesssim \sqrt{T}^{-d}\exp(-\frac{R}{C\sqrt{T}})$ for any $x\in B_{\frac{3}{2}R}(y)$ and $\exp(-\frac{(\vert x-y\vert-R)_+}{C\sqrt{T}})\lesssim \exp(-\frac{R}{C\sqrt{T}})$ for any $x\in \mathbb{R}^d\backslash B_{\frac{3}{2}R}(y)$, we obtain
$$\big(\int \eta_{\sqrt{T}} \vert (h,\sqrt{T}f')\vert^2\big)^{\frac{1}{2}}\lesssim (1+\big(\frac{R}{\sqrt{T}}\big)^{\frac{d}{2}})\frac{R^2}{T}\exp(-\frac{R}{C\sqrt{T}})R\sup\vert\nabla^2 f\vert\lesssim\exp(-\frac{R}{C\sqrt{T}})R\sup\vert\nabla^2 f\vert,$$
where we absorbed the ratio $\big(\frac{R}{\sqrt{T}}\big)^{\frac{d}{2}+2}$ into the exponential. This leads to
$$\vert\nabla w_{\infty}(0)\vert\lesssim   \exp(-\frac{R}{C\sqrt{T}})R\sup\vert \nabla^2 f\vert,$$
and concludes the proof for the regime $R\geq 2\sqrt{T}$.

\medskip

{\sc Step 4. The perturbative regime $R\leq 2\sqrt{T}$.}
For this regime, we proceed in the vein of the proof of Lemma \ref{L:3}. First, by absorbing the ratios $\frac{\vert x\vert}{\sqrt{T}}$ into the exponential, we deduce from \eqref{BoundsBarGT}
$$
\vert\nabla^2 \bar G_T(x)\vert+
\big(\frac{1}{\vert x\vert}+\frac{1}{\sqrt{T}}\big)\vert\nabla \bar G_T(x)\vert+\big(\frac{1}{\vert x\vert^2}+\frac{1}{T}\big)\vert\bar G_T(x)\vert\lesssim \vert x\vert^{-d}.$$
This estimate together with the same computations done for \eqref{Lem3:Eq1} leads to
\begin{equation}\label{Lemma3Massive:Eq2Bis}
\vert(\nabla^3\bar u(x),\tfrac{1}{\sqrt{T}}\nabla^2\bar u(x),\tfrac{1}{T}\nabla\bar u(x))\vert\lesssim \big(\frac{R}{R+\vert x-y\vert}\big)^d R\sup\vert\nabla^2 f\vert.
\end{equation}
We then proceed as in Step 2 of this proof and from \eqref{Lemma3Massive:Eq2Bis} and \eqref{Prop4Massive:Eq3} we have
\begin{equation}\label{Lemma3Massive:Eq3BisBis}
\vert(h(x),\sqrt{T}f'(x))\vert\lesssim \big(\frac{R}{R+\vert x-y\vert}\big)^d\min\{\vert x\vert,1\}R\sup\vert\nabla^2 f\vert.
\end{equation}
Finally, we do the dyadic decomposition of Step 3 up to the scale $\frac{R}{2}$ where the near-field part is controlled using Lemma \ref{LipschitzMassive}, \eqref{Lemma3Massive:Eq3BisBis} and the energy estimate by $\vert\sum_{2^k\leq \frac{R}{2}}\nabla w_k(0)\vert\lesssim (\ln R)R\sup\vert\nabla^2 f\vert$, whereas the far-field part is estimated as follows
\begin{align*}
\vert\nabla w_\infty(0)\vert\stackrel{\eqref{cw'02}}{\lesssim} \big(\fint_{B_{\frac{R}{2}}}\vert(\nabla w_\infty,\tfrac{1}{\sqrt{T}} w_\infty)\vert^2\big)^{\frac{1}{2}} \lesssim & R^{-\frac{d}{2}}\big(\int \vert (h_\infty,\sqrt{T}f'_\infty)\vert^2\big)^{\frac{1}{2}}\\
\stackrel{\eqref{Lemma3Massive:Eq3BisBis}}{\lesssim}& R^{-\frac{d}{2}}R\sup\vert\nabla^2 f\vert\big(\int dx\,\big(\frac{R}{R+\vert x-y\vert}\big)^{2d}\big)^{\frac{1}{2}}\\
\lesssim & R\sup\vert\nabla^2 f\vert,
\end{align*}
which concludes the proof of \eqref{Prop4Massive:Eq6} in the regime $R\leq 2\sqrt{T}$.
\end{proof}
We now turn to the proof of Proposition \ref{Prop4Massive:Statement}.
\begin{proof}[Proof of Proposition \ref{Prop4Massive:Statement}]In the following $\lesssim$ has the same meaning as in the proof of Lemma \ref{LipschitzMassive} and the constant $C$ denotes a general constant depending on $d$ and $\lambda$ which may change from line to line.

\medskip

The proof relies on a weaker version of Lemma \ref{L:1} for the operator $\frac{1}{T}-\nabla\cdot a\nabla$ : for any $u$ such that $\frac{1}{T}u-\nabla\cdot a\nabla u=0$ in $B_R(y)$ for some $y\in\mathbb{R}^d$ and $R\leq \sqrt{T}$, we have 
\begin{equation}\label{BellaMassive}
\vert\nabla u(y)\vert\lesssim\sup_{f\in C^{\infty}_{0}(B_R(y))}\frac{\big\vert\fint_{B_R(y)}fu\big\vert}{R^3\sup\vert\nabla^2 f\vert}.
\end{equation}
The proof of \eqref{BellaMassive} is more elementary, since we ask for a strong control of $\nabla u$ in terms of weak-norms of $u$ itself and relies only on Caccippoli's inequality and an interpolation estimate. In the following, we display the proof of \eqref{BellaMassive}.

\medskip
First, since the constant in \eqref{BellaMassive} depends on $a$ only through $d$ and $\lambda$,
we may rescale and w.~l.~o.~g.~assume $R=1$, $T=1$ and $y=0$. Second, we fix a cut-off function $\eta$ for
$B_\frac{1}{2}$ in $B_\frac{3}{4}$. Using the Lipschitz estimate in Lemma \ref{LipschitzMassive}, \eqref{BellaMassive} will follow from
\begin{align}\label{as01}
\big(\int\eta^6|\nabla u|^2\big)^\frac{1}{2}\lesssim
\sup_{f\in C^{\infty}_{0}(B_1)}\frac{\int f u}
{\big(\int|\nabla^2 f|^2\big)^\frac{1}{2}}.
\end{align}
The starting point is the standard
Caccioppoli estimate
\begin{align*}
\big(\int\eta^6|\nabla u|^2\big)^\frac{1}{2}
=\big(\int(\eta^3|\nabla u|)^2\big)^\frac{1}{2}\lesssim 
\big(\int (u|\nabla \eta^3|)^2\big)^\frac{1}{2}
\lesssim\big(\int \eta^4 u^2\big)^\frac{1}{2}.
\end{align*}
The main ingredient is the interpolation estimate
\begin{align}\label{as03}
\int\eta^4v^2
\lesssim\big(\int\eta^6|\nabla v|^2\big)^\frac{2}{3}
\big(\int((1-\triangle)^{-1}v)^2\big)^\frac{1}{3}
+\int((1-\triangle)^{-1}v)^2,
\end{align}
which we apply to $v=\eta_0u$, where we fix a cut-off function $\eta_0$ for $B_\frac{3}{4}$
in $B_1$, to the effect of $\eta^4 u^2$ $=\eta^4 v^2$
and $\eta^6|\nabla v|^2=\eta^6|\nabla u|^2$. 
The remaining ingredient is
\begin{align}\label{as02}
\big(\int((1-\triangle)^{-1}\eta_0u)^2\big)^\frac{1}{2}
\lesssim \sup_{f\in C^{\infty}_{0}(B_R)}\frac{\int fu}
{\big(\int|\nabla^2 f|^2\big)^\frac{1}{2}}.
\end{align}

\medskip

The argument for (\ref{as02}) is straightforward: As can be easily seen by means
of the Fourier transform, the l.~h.~s.~is identical to
\begin{align*}
\sup_{ f}\frac{\int f \eta_0u}
{\big(\int( f^2+2|\nabla f|^2+|\nabla^2 f|^2)\big)^\frac{1}{2}},
\end{align*}
so that (\ref{as02}) follows from Leibniz' rule in form of
\begin{align*}
\int|\nabla^2\eta_0 f|^2\lesssim \int( f^2+2|\nabla f|^2+|\nabla^2 f|^2).
\end{align*}

\medskip

Using the abbreviation $w=(1-\triangle)^{-1}v$, the interpolation estimate (\ref{as03})
follows from the two interpolation estimates
%
%
\begin{align}
\int\eta^4v^2
\lesssim\big(\int\eta^6|\nabla v|^2\big)^\frac{1}{2}
\big(\int\eta^2|\nabla w|^2\big)^\frac{1}{2}+\int\eta^2(w^2+|\nabla w|^2),\label{as04}\\
\int\eta^2|\nabla w|^2
\lesssim\big(\int\eta^4v^2\big)^\frac{1}{2}
\big(\int w^2\big)^\frac{1}{2}+\int w^2,\label{as05}
\end{align}
namely by inserting (\ref{as05}) into (\ref{as04}) and appealing to Young's inequality.
For (\ref{as04}), we write the l.~h.~s.~as $\int\eta^4v(1-\triangle)w$, hence by 
integration by parts
\begin{align*}
\int\eta^4v^2=\int\eta^4\nabla v\cdot\nabla w+3\int\eta^3v\nabla\eta\cdot\nabla w+\int\eta^4vw,
\end{align*}
so that (\ref{as04}) follows from Cauchy-Schwarz and Young.
For (\ref{as05}), we use integration by parts to rewrite the l.~h.~s.~as
$\int\eta^2 w(-\triangle) w+\int w^2\triangle\frac{1}{2}\eta ^2$.
Hence we obtain
$\int\eta^2|\nabla w|^2$ $=\int\eta^2 w v+\int w^2(\triangle\frac{1}{2}\eta ^2-\eta^2)$,
so that (\ref{as05}) reduces to Cauchy-Schwarz.

\medskip

As for Lemma \ref{Lemma3MassiveVersion}, we split the argument between the non-perturbative regime $R:=\frac{\vert x_0-y_0\vert}{2}\geq \sqrt{T}$, that we address in the five first steps, and the perturbative regime $R\leq \sqrt{T}$ that we address in the last step.

\medskip

{\sc Step 1. Passage to the full error in the two-scale expansion.}
The computations done in the Step 1 of Section \ref{ProofOfProp4Random} extend in a straightforward way to the present setting : defining the full error of the second-order
two-scale expansion of the constant-coefficient fundamental solution $\bar G_T$,
which is given by
    \begin{equation}\label{Def_VBis}
    \begin{aligned}
      w_{x_0,y_0,T}(x,y) := G_T(x,y)-&\big(1+\phi^{(1)}_{i}(x)\partial_i
      +(\phi^{(2)}_{im}(x)-\phi^{(2)}_{im}(x_0))\partial_{im}\big)\\
      &\times\big(1-\phi^{*(1)}_j(y)\partial_j+(\phi^{*(2)}_{jn}(y)
      -\phi_{jn}^{*(2)}(y_0))(y)
      \partial_{jn}\big)\bar{G}_T(x-y),
    \end{aligned}
    \end{equation}
we have, from \eqref{Prop4Massive:Eq3} and the bounds on $\bar G_T$ \eqref{BoundsBarGT}
 \begin{equation}\label{Num:252Bis}
      |\nabla\nabla w_{x_0,y_0,T}(x_0,y_0) - \mathcal{E}_T(x_0,y_0)| \lesssim |x_0-y_0|^{-d-2}\exp(-\frac{\vert x_0-y_0\vert}{C\sqrt{T}}).
    \end{equation}
Therefore, to obtain the desired estimate $\eqref{fw56}$ for $L=1$
it suffices to show
\begin{equation}\label{Num:2370Bis}
   |\nabla \nabla w_{x_0,y_0,T}(x_0,y_0)|
  \lesssim (\ln \vert x_0-y_0\vert)|x_0-y_0|^{-d-2} \exp(-\frac{\vert x_0-y_0\vert}{C\sqrt{T}}).
\end{equation}

\medskip

{\sc Step 2. A decomposition of the full error $\nabla\nabla w_{x_0,y_0,T}(x_0,y_0)$.}
In this step, we shall
derive a characterizing PDE (\ref{Eq_sur_VBis}) of the full error \eqref{Def_VBis} in order to split
it into a far-field part $w_{x_0,y_0,T,\infty}$ and dyadic near-field parts $w_{x_0,y_0,T,k}(x,\cdot)$, which will be 
explicitly given later on. The distinction between
far and near fields refers to the scale $R=|x_0-y_0|/2$.
Recall that $w_{x_0,y_0,T}(x,y)$ involves the two-scale expansion in both the $x$
and $y$ variables; we now freeze $x=x_0$ and consider $y$ as the ``active'' variable. For the ease of the statement, we make use of the notation
\begin{equation}\label{NotationuBarExpGBis}
\bar u_{x_0,T}(x,\cdot):=
\big(1+\phi_i^{(1)}(x)\partial_i+(\phi^{(2)}_{im}(x)-\phi_{im}^{(2)}(x_0))\partial_{im}\big)
\bar G_T(x-\cdot).
\end{equation}
This amounts to rewriting $ w_{x_0,y_0,T}(x_0,\cdot)$ as follows:
  \begin{equation*}\label{ExpandVBis}
  \begin{aligned}
      w_{x_0,y_0,T}(x_0,\cdot)=
    & G_T(x_0,\cdot)-
    \big(1-\phi^{*(1)}_{j}\partial_j+
    (\phi_{jn}^{*(2)}-\phi_{jn}^{*(2)}(y_0))
    \partial_{jn}\big)
    \bar{u}_{x_0,T}(x_0,\cdot).
  \end{aligned}
  \end{equation*}
We note that $(\frac{1}{T}-\nabla\cdot a^*\nabla)G_T(x_0,\cdot)=(\frac{1}{T}-\nabla\cdot \bar a^*\nabla)\bar u_{x_0,T}(x_0,\cdot)=0$ in $\mathbb{R}^d\backslash \{x_0\}$.
Hence, as in \eqref{Lemma3Massive:Eq4}, we obtain the representation of the error in the second order two-scale expansion 
\begin{equation}\label{Eq_sur_VBis}
 (\frac{1}{T} -\nabla  \cdot  a^* \nabla) w_{x_0,y_0,T}(x_0,\cdot)= \nabla \cdot h_{x_0,y_0,T}(x_0,\cdot)+f_{x_0,y_0,T}  \quad
  \text{in}\quad \R^d \setminus \{x_0\},
\end{equation}
where
\begin{equation}\label{Def_hBis}
\begin{aligned}
  h_{x_0,y_0,T}(x,\cdot) := &\big((\phi^{*(2)}_{jn}-\phi^{*(2)}_{jn}(y_0))a^* -(\sigma^{*(2)}_{jn}-\sigma_{jn}^{*(2)}(y_0))\big) \nabla\partial_{jn}\bar{u}_{x_0,T}(x,\cdot)\\
  &-\frac{1}{T}(h^{(1)}_i-h^{(1)}_i(y_0))\partial_i\bar u_{x_0,T}(x,\cdot)
  \end{aligned}
\end{equation}
and 
\begin{equation}\label{Def_HBis}
f_{x_0,y_0,T}(x,\cdot):=\frac{1}{T}\big((h^{(1)}_i-h^{(1)}_i(y_0))e_j-(\phi^{(2)}_{ij}-\phi^{(2)}_{ij}(y_0))\big)\partial_{ij}\bar u_{x_0,T}(x,\cdot),
\end{equation}
with $h^{(1)}_i$ given by \eqref{PeriodicFieldLemma3Massive}.
Next, we define the dyadic near-field (scalar) functions :
\begin{equation}\label{Num:220Bis}
\begin{aligned}
(\frac{1}{T}-\nabla\cdot a^*\nabla) w_{x_0,y_0,T,k}(x,\cdot)
=&\nabla\cdot
\mathds{1}_{B_{2^k}(y_0) \backslash B_{2^{k-1}}(y_0)}h_{x_0,y_0,T}(x,\cdot)\\
&+\mathds{1}_{B_{2^k}(y_0) \backslash B_{2^{k-1}}(y_0)}f_{x_0,y_0,T}(x,\cdot),
\end{aligned}
\end{equation}
and the far-field function:
\begin{equation}\label{Num:220-1Bis}
\begin{aligned}
w_{x_0,y_0,T,\infty}:=w_{x_0,y_0,T}-\sum_{2^k\le \sqrt{T}}w_{x_0,y_0,T,k},
\end{aligned}
\end{equation}
In fact, we
are interested in the quantities
$\nabla_{x} w_{x_0,y_0,T,k}(x_0,\cdot)$ and
$\nabla_{x}  w_{x_0,y_0,T,\infty}(x_0,\cdot)$, which we address in two steps.

\medskip

We finally give an estimate on $h_{x_0,y_0,T}$ and $f_{x_0,y_0,T}$ that will be useful in the next steps : We start by estimating its constitutive element $\bar u_{x_0,T}$ (see \eqref{NotationuBarExpGBis}) and we obtain from \eqref{Prop4Massive:Eq3}, \eqref{BoundsBarGT} and $R\geq 1$
  \begin{equation}\label{Num:2401Bis}
    \sup_{y\in B_{\sqrt{T}}(y_0)}|\nabla^j
    \partial_{x_i}\bar{u}_{x_0,T}(x_0,y)|
    \lesssim R^{-j-d+1} \exp(-\frac{R}{C\sqrt{T}})\quad \text{for any $j\geq 0$,}
  \end{equation}
so that applying once more \eqref{Prop4Massive:Eq3} and absorbing the ratios $\frac{R}{\sqrt{T}}$ into the exponential leads to 
\begin{equation}\label{Lemma3Massive:Eq3Bis}
\vert(\partial_{x_i}h_{x_0,y_0,T}(x,\cdot),\sqrt{T}\partial_{x_i}f_{x_0,y_0,T}(x,\cdot))\vert\lesssim R^{-d-2}\exp(-\frac{R}{C\sqrt{T}})\min\{\vert \cdot-y_0\vert,1\}.
\end{equation}

{\sc Step 3. Estimate of the near-field parts
$\nabla\nabla w_{x_0,y_0,T,k}(x_0,\cdot)$. }
Note that applying $\nabla_x$ and evaluating at $x=x_0$ commutes with the differential operator $\frac{1}{T}-\nabla\cdot a^*\nabla
$.
For the ease of notation, we fix an arbitrary coordinate direction $i=1,\cdots,d$ and introduce the abbreviation\footnote{Throughout the proof,
we use $\partial_{x_i}$ (or $\nabla_x$) if
the partial derivative (or the gradient) is
taken w.r.t. the first variable.
But we may write $\nabla$ for $\nabla_y$
since $y$ here is the ``active'' variable.} 
$w_{T,k,i}(y):=\partial_{x_i} w_{x_0,y_0,T,k}(x_0,y)$.  By applying the Lipschitz estimate of Lemma \ref{LipschitzMassive} combined with the energy estimate and \eqref{Lemma3Massive:Eq3Bis}, we have
\begin{equation}\label{num:220-7Bis}
  \begin{aligned}
    \vert \nabla
    w_{T,k,i}(y_0) \vert
    \stackrel{\eqref{cw'02}}{\lesssim} &\big(\fint_{B_{2^{k-1}}(y_0)}\vert(\nabla w_{T,k,i},\tfrac{1}{\sqrt{T}}w_{T,k,i}) \vert^2\big)^{\frac{1}{2}}\\
    \stackrel{\eqref{Num:220Bis}}{\lesssim}&  \big(\fint_{B_{2^{k}}(y_0)}\vert
    (\partial_{x_i} h_{x_0,y_0,T}(x_0,\cdot),\sqrt{T}\partial_{x_i}f_{x_0,y_0,T})\vert^2\big)^{\frac{1}{2}}\\
  \stackrel{\eqref{Lemma3Massive:Eq3Bis}}{\lesssim} & R^{-d-2}\exp(-\frac{R}{C\sqrt{T}})\min\{2^{k},1\}.
  \end{aligned}
\end{equation}
Then, summing up yields
\begin{equation}\label{pri:4.23Bis}
\sum_{2^k\le \sqrt{T}}
    \vert\nabla w_{T,k,i}(y_0) \vert
    \lesssim (\ln R)R^{-d-2} \exp(-\frac{R}{C\sqrt{T}}).
\end{equation}
\medskip

{\sc Step 4. Estimate of the near-field parts $\nabla_{x} w_{x_0,y_0,T,k}(x_0,\cdot)$ in a weak norm. }
In the sequel, $f=f(y)$ always denotes an
arbitrary smooth function compactly supported in $B_{\sqrt{T}}(y_0)$.
We now justify a weak control on $\nabla_{x}  w_{x_0,y_0,\infty}(x_0,\cdot)$ that will be useful in Step $5$ when appealing to \eqref{BellaMassive} :
\begin{equation}\label{Num:220-3Bis}
\sum_{2^k\le \sqrt{T}}\big|
\int fw_{T,k,i}\big| \lesssim \exp(-\frac{R}{C\sqrt{T}})R^{-1}\sup|f|,
\end{equation}
where we recall that $w_{T,k,i}=\partial_{x_i}w_{x_0,y_0,T,k}$ where the latter is defined in \eqref{Num:220Bis}. The estimate \eqref{Num:220-3Bis} is a consequence of Cauchy-Schwarz' inequality followed by the energy estimate and \eqref{Lemma3Massive:Eq3Bis} : 
\begin{align*}
\big\vert\int fw_{T,k,i}\big\vert\leq& \big(\int\vert f\vert^2\big)^{\frac{1}{2}}\big(\int \vert w_{T,k,i}\vert^{2}\big)^{\frac{1}{2}}\\
\lesssim &  \sqrt{T}^{\frac{d}{2}+1}\sup\vert f\vert\big(\int \vert \frac{1}{\sqrt{T}}  w_{T,k,i}\vert^2\big)^{\frac{1}{2}}\\
\stackrel{\eqref{Num:220Bis}}{\lesssim}& \sqrt{T}^{\frac{d}{2}+1}\sup\vert f\vert\big(\int_{B_{2^k}(y_0)} \vert (\partial_{x_i}h_{x_0,y_0,T}(x_0,\cdot),\sqrt{T} \partial_{x_i} f_{x_0,y_0,T})\vert^2\big)^{\frac{1}{2}}\\
\stackrel{\eqref{Lemma3Massive:Eq3Bis}}{\lesssim} &2^{\frac{kd}{2}}\sqrt{T}^{\frac{d}{2}+1} R^{-d-2} \exp(-\frac{R}{C\sqrt{T}})\sup \vert f\vert.
\end{align*}
Thus, summing up yields \eqref{Num:220-3Bis}, by appealing to $R\geq \sqrt{T}$.

\medskip

{\sc Step 5. Estimate of the far-field part $\nabla\nabla_{x}  w_{x_0,y_0,T,\infty}(x_0,\cdot)$ by a duality argument. }Again,
for the ease of notation we introduce the abbreviation
$w_{T,\infty,i}(y):=\partial_{x_i}w_{x_0,y_0,T,\infty}(x_0,y)$
with an arbitrary coordinate direction
$i=1,\cdots,d$,
which solves $(\frac{1}{T}-\nabla\cdot a^*\nabla)w_{T,\infty,i}(y)=0$ on $B_{2^{k_0}}(y_0)$.
While in the previous
two steps (mostly) relied on homogenization in the $y$-variable, we now (primarily) need homogenization in the $x$-variable,
in form of Lemma \ref{Lemma3MassiveVersion}. We start with
an application of \eqref{BellaMassive} which we combine with \eqref{Num:220-1Bis} and the triangle inequality :
\begin{equation}\label{Num:235Bis}
\begin{aligned}
 |\nabla  w_{T,\infty,i}(y_0)|
\stackrel{\eqref{BellaMassive}}{\lesssim}&
 \sup_{f\in C^\infty_0(B_{\sqrt{T}}(y_0))}
\frac{\big|
\fint_{B_{\sqrt{T}}(y_0)}f  w_{T,\infty,i} \big|}{\sqrt{T}^{3}\sup|\nabla^2 f|}\\
 \overset{\eqref{Num:220-1Bis}}{\lesssim}&
\sup_{f\in C^\infty_0(B_{\sqrt{T}}(y_0))}
\frac{\big|\int
f\partial_{x_i} w_{x_0,y_0,T}(x_0,\cdot)\big|}{\sqrt{T}^{d+3}
\sup|\nabla^2 f|}+
\sup_{f\in C^\infty_0(B_{\sqrt{T}}(y_0))}
\frac{\sum_{2^k\leq \sqrt{T}}\big|
\int f w_{T,k,i} \big|}{\sqrt{T}^{d+3}\sup|\nabla^2 f|}.
\end{aligned}
\end{equation}
The second contribution is a direct consequence of $\eqref{Num:220-3Bis}$, 
$$\sup_{f\in C^\infty_0(B_{\sqrt{T}}(y_0))}
\frac{\sum_{2^k\leq \sqrt{T}}\big|
\int f w_{T,k,i} \big|}{\sqrt{T}^{d+3}\sup|\nabla^2 f|}\lesssim \Big(\frac{R}{\sqrt{T}}\Big)^{d+3}R^{-d-2}\exp(-\frac{R}{C\sqrt{T}})\lesssim R^{-d-2}\exp(-\frac{R}{C\sqrt{T}}),$$
where we absorbed the ration $(\frac{R}{\sqrt{T}})^{d+3}$ into the exponential. We now need a similar estimate on the first contribution, namely,
    \begin{equation}\label{pri:4.22Bis}
    \begin{aligned}
    \big|
    \int f \partial_{x_i} w_{x_0,y_0,T}(x_0,\cdot) \big|
    \lesssim 
(\ln R)\exp(-\frac{R}{C\sqrt{T}})R\sup |\nabla^2 f|.
    \end{aligned}
    \end{equation}
Equipped with \eqref{pri:4.22Bis}, \eqref{Num:2370Bis} follows from applying the triangle inequality to \eqref{Num:220-1Bis}, into which we insert \eqref{pri:4.23Bis} and \eqref{Num:235Bis}, where we use \eqref{pri:4.22Bis} and \eqref{Num:220-3Bis}.

\medskip

We now focus on the
argument for $\eqref{pri:4.22Bis}$.
Let $f\in C_0^\infty(B_{\sqrt{T}}(y_0))$ be arbitrary, and let $u$, $\bar u$ as in \eqref{Lemma3Massive:Eq1'}. We recall the definition of the error in the two-scale expansion that we express in terms of the Green functions $G_T,\bar G_T$ using \eqref{Lemma3Massive:Eq1'}:
  \begin{equation*}
  \begin{aligned}
   w_{x_0}(x)    &:= u(x) - \big(1+\phi_{i}^{(1)}(x)\partial_i+
   (\phi_{im}^{(2)}-\phi_{im}^{(2)}(x_0))(x)
   \partial_{im}
   \big)
   \bar{u}(x)\\
   &\overset{\eqref{NotationuBarExpGBis}}{=} \int
   f(G_T(x,\cdot) - \bar u_{x_0,T}(x,\cdot)),
  \end{aligned}
  \end{equation*}
 Then,  taking derivatives on the both sides with respect to the $x$-variable leads to
  \begin{equation}\label{Num:220-6Bis}
  \begin{aligned}
  \partial_{x_i} w_{x_0}(x)
   =\int
  f
  \big(\partial_{x_i} G_T(x,\cdot) - \partial_{x_i}\bar{u}_{x_0,T}(x,\cdot)\big).
  \end{aligned}
  \end{equation}
We now express the integral on the l.~h.~s. of \eqref{pri:4.22Bis} with help of \eqref{Num:220-6Bis}. We split the integral on the l.~h.~s. of \eqref{pri:4.22Bis} as follows:
%
  \begin{equation}\label{Split_vBis}
  \begin{aligned}
  \lefteqn{\int f \partial_{x_i}  w_{x_0,y_0,T}(x_0,\cdot)}\\
  &\overset{\eqref{Def_VBis},\eqref{NotationuBarExpGBis}}{=}\int
f\big(\partial_{x_i} G_T(x_0,\cdot)-
  \partial_{x_i}\bar{u}_{x_0,T}(x_0,\cdot) \big)\\
  &\quad\quad+
  \int
 f \big((\phi^{*(1)}_{j}\partial_j-(\phi_{jn}^{*(2)}
  -\phi_{jn}^{*(2)}(y_0))
  \partial_{jn})\partial_{x_i}\bar{u}_{x_0,T}(x_0,\cdot)\big)
  \\
  &\overset{\eqref{Num:220-6Bis}}{=}
  \partial_{x_i} w_{x_0}(x_0)+ \int
 f \big((\phi^{*(1)}_{j}\partial_j-(\phi_{jn}^{*(2)}
  -\phi_{jn}^{*(2)}(y_0))
  \partial_{jn})\partial_{x_i}\bar{u}_{x_0,T}(x_0,\cdot)\big).
  \end{aligned}
  \end{equation}
  For the first r.~h.~s. term of $\eqref{Split_vBis}$,
  it follows from \eqref{Prop4Massive:Eq6} that
  \begin{equation}\label{pri:4.24Bis}
|\nabla w_{x_0}(x_0) |
  \lesssim (\ln R)\exp(-\frac{R}{C\sqrt{T}}) R\sup|\nabla^2 f|.
  \end{equation}
  For the second r.~h.~s.\ term, we exploit the periodic vector field $h^{(1)}_j$ defined in \eqref{PeriodicFieldLemma3Massive} to obtain
\begin{align*}
\int
& f \big((\phi^{*(1)}_{j}\partial_j-(\phi_{jn}^{*(2)}
  -\phi_{jn}^{*(2)}(y_0))
  \partial_{jn})\partial_{x_i}\bar{u}_{x_0,T}(x_0,\cdot)\big)\\
  \stackrel{\eqref{PeriodicFieldLemma3Massive}}{=}&\int
 f \big((\nabla\cdot h^{(1)}_j\partial_j-(\phi_{jn}^{*(2)}
  -\phi_{jn}^{*(2)}(y_0))
  \partial_{jn})\partial_{x_i}\bar{u}_{x_0,T}(x_0,\cdot)\big)\\
  =&-\int
 f \big((h^{(1)}_{jn}-(\phi_{jn}^{*(2)}
  -\phi_{jn}^{*(2)}(y_0))
  )\partial_{jn}\partial_{x_i}\bar{u}_{x_0,T}(x_0,\cdot)\big)\\
  &-\int \nabla f\cdot h^{(1)}_j\partial_j\partial_{x_i}\bar u_{x_0,T}(x_0,\cdot)
\end{align*}
Thus, using \eqref{Num:2401Bis} and \eqref{Prop4Massive:Eq3}, we deduce
\begin{align*}
\big\vert \int f \big((\phi^{*(1)}_{j}\partial_j-(\phi_{jn}^{*(2)}
  -\phi_{jn}^{*(2)}(y_0))
  \partial_{jn})\partial_{x_i}\bar{u}_{x_0,T}(x_0,\cdot)\big)\big\vert\lesssim& \exp(-\frac{R}{C\sqrt{T}})(R^{-1}\sup\vert f\vert+\sup\vert\nabla f\vert)\\
  \lesssim&  \exp(-\frac{R}{C\sqrt{T}})R\sup \vert\nabla^2 f\vert,
\end{align*}
which concludes the proof of \eqref{pri:4.22Bis}.

\medskip

{\sc Step 6. The perturbative regime $R\leq \sqrt{T}$.} In this regime, we change Step 2 by stopping the dyadic decomposition at scale $R$ and replacing \eqref{Num:2401Bis} by
$$\sup_{y\in B_R(y_0)}\vert\nabla^j\partial_{x_i}\bar u_{x_0,T}(x_0,y)\vert\lesssim R^{-j-d+1}.$$
We get as in \eqref{Lemma3Massive:Eq3Bis}, using $R\leq \sqrt{T}$
$$\vert(\partial_{x_i}h_{x_0,y_0,T}(x,\cdot),\sqrt{T}\partial_{x_i}f_{x_0,y_0,T}(x,\cdot))\vert\lesssim R^{-d-2}\min\{\vert \cdot -y_0\vert,1\}.$$
The further steps are unchanged up to replacing $\sqrt{T}$ by $R$ (and ignoring the exponential terms).
\end{proof}
\section{Intermediate results for the heuristic result of Section \ref{S:heur}}
We prove in this section some intermediate results for the heuristic argument in Section \ref{S:heur}.
\subsection{Argument for (\ref{rr08}) : symmetry of $c^{\text{sym}}_L$}\label{SS:App2}

By periodicity (\ref{rr09}), it is enough to consider $x_d,y_d\in[-\frac{L}{2},\frac{L}{2}]$.
If $x_d\in[-\frac{L}{2},0]$ we appeal to reflection symmetry (\ref{rr06}) to replace
the corresponding argument on the l.~h.~s.~of (\ref{rr08}) by $-x_d\in[0,\frac{L}{2}]$; 
if $x_d\in[0,\frac{L}{2}]$ we use both periodicity and reflection symmetry to replace
the corresponding argument on the r.~h.~s.~ of \eqref{rr08} by $\frac{L}{2}-x_d\in[0,\frac{L}{2}]$.
We proceed the same way for $y_d$. This way, all four arguments are in the range $[0,\frac{L}{2}]$
where (\ref{rr07})
applies. We also note that the isotropy of $c$ leads to reflection symmetry of $c$ in $x_d$,
which by (\ref{rr17}) transmits to $c_L'$, that is,
\begin{align}\label{rr33}
c_L'(z',z_d)=c_L'(z',-z_d).
\end{align}
It remains to distinguish the three (relevant) cases, 
\begin{align*}
&\mbox{if}\;(x_d,y_d)\in[0,{\textstyle\frac{L}{2}}]^2\;\mbox{then}\nonumber\\
&\mbox{}\;c^{sym}_{L}(x+\tfrac{L}{2}e_d,y+\tfrac{L}{2}e_d)=c^{sym}_{L}((x',\tfrac{L}{2}-x_d),(y',\tfrac{L}{2}-y_d))\stackrel{\eqref{rr07},\eqref{rr33}}{=}c^{sym}_L(x,y),\\
&\mbox{if}\;(x_d,y_d)\in[0,{\textstyle\frac{L}{2}}]\times[-{\textstyle\frac{L}{2}},0]
\;\mbox{then}\nonumber\\
&\mbox{}\;c^{sym}_{L}(x+\tfrac{L}{2}e_d,y+\tfrac{L}{2}e_d)=c^{sym}_{L}((x',\tfrac{L}{2}-x_d),y+\tfrac{L}{2}e_d)\stackrel{\eqref{rr07},\eqref{rr06},\eqref{rr33}}{=}c^{sym}_L(x,y),\nonumber\\
&\mbox{if}\;(x_d,y_d)\in[-{\textstyle\frac{L}{2}},0]^2\;\mbox{then}\nonumber\\
&\mbox{}\;c^{sym}_{L}(x+\tfrac{L}{2}e_d,y+\tfrac{L}{2}e_d)\stackrel{\eqref{rr07},\eqref{rr33}}{=}c^{sym}_L(x,y).
\end{align*}


\subsection{Computation of the integral (\ref{rr31})}\label{SS:App1}
We change variables in (\ref{rr30}) according to
\begin{align*}
z_d=x_d-y_d,\quad w_d=x_d+y_d,
\end{align*}
so that $\int_0^\infty \int_{-\infty}^0 dx_ddy_d$ $=2\int_0^\infty dz_d\int_{-z_d}^{z_d}dw_d$
$=2\int_{-\infty}^\infty dw_d\int_{|w_d|}^{\infty}dz_d$. We note that because of our isotropy assumption,
$r\partial_r{\mathcal A}(c)$ is radial and thus the integrand in (\ref{rr30}) is in particular invariant under $w_d\mapsto-w_d$.
This allows to substitute $\int_0^\infty dz_d\int_{-z_d}^{z_d}dw_d$
$=2\int_0^\infty dz_d\int_{0}^{z_d}dw_d$ and $\int_{-\infty}^\infty dw_d\int_{|w_d|}^{\infty}dz_d$
$=2\int_{0}^\infty dw_d\int_{w_d}^{\infty}dz_d$. Hence from (\ref{rr30}) we obtain
\begin{align*}
I=
16\int_{\mathbb{R}^{d-1}\times(0,\infty)}dx(r\partial_r){\mathcal A}(c)(x)
\Big(\int_{x_d}^\infty\dd y_d \partial_1^2G_{\rm hom}(x',y_d)  - x_d \partial_1^2G_{\rm hom}(x) \Big)
\end{align*}
Once more by the isotropy assumption, the value of this expression is the same when
$\partial_1^2$ is replaced by $\partial_i^2$ for $i=1,\cdots,d-1$. By the characterizing 
property of the fundamental solution we have $\sum_{i=1}^{d-1}\partial_i^2G_{\rm hom}=
-\delta-\partial_d^2G_{\rm hom}$. The contribution from the Dirac $\delta$ vanishes since
$r\partial_r{\mathcal A}(c)$ vanishes (actually to second order) at $x=0$.
Integrating in $y_d$, this yields
\begin{align*}
I=\frac{16}{d-1}\int_{\mathbb{R}^{d-1}\times(0,\infty)}dx
(r\partial_r){\mathcal A}(c)\big(x_d\partial_d^2G_{\rm hom} + \partial_dG_{\rm hom}\big).
\end{align*}
By radial symmetry of $G_{\rm hom}$ we have $\partial_dG_{\rm hom}$ 
$=\frac{x_d}{r}\partial_rG_{\rm hom}$, so that together with the characterizing property
$\partial_rG_{\rm hom}$ $=-\frac{1}{|\partial B_1|}\frac{1}{r^{d-1}}$
we obtain $\partial_dG_{\rm hom}$ $=-\frac{x_d}{|\partial B_1|}\frac{1}{r^{d}}$
and thus by $x_d\partial_d^2G_{\rm hom}$ $=-\frac{x_d}{|\partial B_1|}\frac{1}{r^{d}}$
$-\frac{x_d}{|\partial B_1|}(-\frac{d x_d}{r^{d+1}})\frac{x_d}{r}$. Hence we obtain
\begin{align*}
I=\frac{16}{d-1}\int_{\mathbb{R}^{d-1}\times(0,\infty)}dx
(r\partial_r){\mathcal A}(c)\frac{1}{|\partial B_1|}\frac{1}{r^{d+2}}\big(-2x_dr^2+dx_d^3\big).
\end{align*}
By radial coordinates, this expression factorizes into
\begin{align*}
I=\frac{16}{d-1}\Big(\int_0^\infty dr(r\partial_r){\mathcal A}(c)\Big)
\frac{1}{|\partial B_1|}\int_{\partial B_1\cap\{x_d>0\}}dx
\big(-2x_d+dx_d^3\big).
\end{align*}
Thanks to the normalization ${\mathcal A}(0)=0$, we may integrate by parts in the
first expression in order to obtain
\begin{align*}
I=\frac{16}{d-1}\frac{1}{|\partial B_1|}
\int_0^\infty dr{\mathcal A}(c)\,I'\;\;\mbox{where}\;\;I':=
\int_{\partial B_1\cap\{x_d>0\}}dx\big(2x_d-dx_d^3\big).
\end{align*}
It remains to compute $I'$.

\medskip

We first note that
\begin{align*}
I'
=2\int_{\{|x|\le 1\}\cap\{x_d>0\}}x_d,
\end{align*}
which can be seen from identifying the l.~h.~s.~integrand with the normal
component of the vector field $2x_d x-dx_d^2e_d$, and applying the divergence
theorem on $\{|x|\le 1\}\cap\{x_d>0\}$. The r.~h.~s.~ integral
can easily be made explicit:
Using first Fubini's theorem based on $x=(x',x_d)$, 
and then radial coordinates $|x'|$, we obtain
\begin{align*}
I'=\int_{\{|x'|\le 1\}}(1-|x'|^2)
=(1-\frac{d-1}{d+1})\int_{\{|x'|\le 1\}}1=\frac{2}{d+1}|B_1'|.
\end{align*}
\subsection{Argument for \eqref{rr14} : representation formula for the $L$-derivative}\label{SS:App3}
Let $\hat{c}_L(x,y):=c_L(Lx,Ly)$ be the rescaled covariance function. The family of covariance functions $\hat{c}_L$ is $1$-periodic and we assume \eqref{rr21}, which takes the form
\begin{equation}\label{MJ:Num:0002}
	\frac{d}{d L}\hat{c}_L(x,x) = 0 \qquad \text{ for all } x \in [-\tfrac{1}{2},\tfrac{1}{2})^d.
\end{equation}
In this case, $\langle\cdot\rangle_L$ is determined by its covariance function $\hat{c}_L$ that we identify with its push-forward under $a=A(g)$. For any $\xi,\xi^*\in\mathbb{R}^d$, we define the corrector $\phi:=\sum_{i}\xi_i\phi_i$ (and $\phi^*$ accordingly) and
\begin{align*}
\xi^* \cdot \abar \xi := \int_{[-\frac{1}{2},\frac{1}{2})^d} \xi^* \cdot a (\xi + \nabla \phi).
\end{align*}
We shall prove, as a consequence Price's formula, that
\begin{equation}\label{MJ:Num:0001}
\begin{aligned}
\frac{d}{d L} \langl \xi^*\cdot\abar\xi  \rangl_{L}
=
-\int_{[-\frac{1}{2},\frac{1}{2})^d} \dd x \int_{[-\frac{1}{2},\frac{1}{2})^d} \dd y \big\langle (a'(\xi^* +\nabla \phi^*))(x) \cdot \nabla_x \nabla_y G^{per}(x,y)(a' (\xi +\nabla \phi))(y) \big\rangle_{L}
\frac{d}{dL} \hat{c}_L(x,y),
\end{aligned}
\end{equation}
where $G^{per}$ denotes the Green function associated with the operator $-\nabla\cdot a\nabla$ on the torus $[0,1)^d$. Note that the rescaling in $L$ of \eqref{MJ:Num:0001} is exactly \eqref{rr14}.
\medskip

In contrast to Section \ref{SS:formula}, here we assume that the Gaussian field $g$ is not defined on the whole space, but only restricted to the torus $[0,1)^d$.
	However, we may establish \eqref{MJ:Num:0001} by following the lines leading to \eqref{ao04}.
	
	Indeed, applying Price's formula \cite{PriceFormula}, we get
	\begin{equation}\label{MJ:Num:0003}
	\begin{aligned}
	\frac{d}{dL} \langl \xi^*\cdot\abar\xi  \rangl_{L}
	=
	\frac{1}{2} \int_{[-\frac{1}{2},\frac{1}{2})^d} \dd x \int_{[-\frac{1}{2},\frac{1}{2})^d} \dd y \int_{[-\frac{1}{2},\frac{1}{2})^d} \dd z \big\langle \xi^* \cdot  \frac{\partial^2 (a (\xi + \nabla \phi))(z)}{\partial g(x) \partial g(y)} \big \rangle_{L}
	\frac{d}{dL} \hat{c}_L(x,y),
	\end{aligned}
	\end{equation}
	where, in contrast to \eqref{ao02}, the partial derivatives $\frac{\partial}{\partial g}$ have to be understood in a periodic sense.
	Then, notice that we may apply the operator $\frac{\partial^2 }{\partial g(x) \partial g(y)}$ to \eqref{pde:9.1_quad} for $e_i \rightsquigarrow \xi$ and test it against $\phi^*$; similarly, we may test \eqref{ao01} with the periodic function $\frac{\partial^2 \phi}{\partial g(x) \partial g(y)}$.
	Since the boundary contributions vanish because every considered function is periodic, this yields
	\begin{align*}
	\int_{[-\frac{1}{2},\frac{1}{2})^d}\nabla \phi^* \cdot  \frac{\partial^2 a (\xi + \nabla \phi)}{\partial g(x) \partial g(y)} = 0
	=
	\int_{[-\frac{1}{2},\frac{1}{2})^d} (\xi^* + \nabla \phi^*) \cdot a \frac{\partial^2 \nabla \phi}{\partial g(x) \partial g(y)}.
	\end{align*}
	Using Leibniz' rule on \eqref{MJ:Num:0003} into which we insert the above identities yields
	\begin{equation}\label{MJ:Num:0004}
	\begin{aligned}
	\frac{d}{dL} \langl \xi^*\cdot\abar\xi  \rangl_{L}
	=~&
	\frac{1}{2} \int_{[-\frac{1}{2},\frac{1}{2})^d} \dd x \int_{[-\frac{1}{2},\frac{1}{2})^d} \dd y \int_{[-\frac{1}{2},\frac{1}{2})^d}\Big(\big\langle (\xi^* + \nabla \phi^*) \cdot  \frac{\partial^2 a}{\partial g(x) \partial g(y)} (\xi + \nabla \phi)\big \rangle_{L}
	\\
	&+\big\langle (\xi^* + \nabla \phi^*) \cdot  \frac{\partial a}{\partial g(x)} \frac{\partial \nabla \phi}{\partial g(y)} \big \rangle_{L}+\big\langle (\xi^* + \nabla \phi^*) \cdot  \frac{\partial a}{\partial g(y)} \frac{\partial \nabla \phi}{\partial g(x)} \big \rangle_{L}
	\Big) \frac{d}{dL} \hat{c}_L(x,y).
	\end{aligned}
	\end{equation}
	
	As in Section \ref{SS:formula}, we remark that by \eqref{Def_a} we have $\frac{\partial a(z)}{\partial g(x)}$ $=a'(x)\delta(x-z)$ and $\frac{\partial^2 a(z)}{\partial g(x)\partial g(y)}$ $=a'(x)\delta(x-z) \delta(y-z)$.
	Moreover, applying the operator $\frac{\partial}{\partial g(x)}$ on \eqref{pde:9.1_quad} for $e_i \rightsquigarrow \xi$, we obtain the representation
	\begin{align*}
	\frac{\partial \nabla \phi(z)}{\partial g(x)}=-\nabla\nabla G^{per}(z,x)a'(x)(\nabla\phi+\xi)(x).
	\end{align*}
	Inserting these identities into \eqref{MJ:Num:0004}, recalling \eqref{MJ:Num:0002}, and using the symmetry
	$\hat{c}_L(x,y) = \hat{c}_L(y,x) $, we get \eqref{MJ:Num:0001}.

\bibliographystyle{plain}
\bibliography{0_Bib_unifie}

\end{document}